 \documentclass[SecEq,CM,GP]{degruyter-crelle} 

\firstpage{}\volume{}\volumeyear{}\copyrightyear{}\doiyear{}\doi{}\received{XXX}\revised{XXX}

\title[$p$-adic adelic metrics and Quadratic Chabauty~I]{$p$-adic adelic metrics and Quadratic Chabauty~I}

\author{Amnon}{Besser}{}{Be'er-Sheva}
\author{J.\thinspace{}Steffen}{M\"uller}{}{Groningen}
\author{Padmavathi}{Srinivasan}{}{Boston}

\contact{Amnon Besser, Department of Mathematics, Ben-Gurion University of the Negev, P.O.B. 653, Be'er-Sheva 84105, Israel}{bessera@math.bgu.ac.il}

\contact{J.\thinspace{}Steffen M\"uller,
  Bernoulli Institute, 
  Rijksuniversiteit Groningen,
  Nijenborgh 9,
  9747 AG Groningen,
  The Netherlands
}{steffen.muller@rug.nl}

\contact{Padmavathi Srinivasan,
Boston University,
Center for Computing and Data Sciences,
665 Commonwealth Avenue, 
Boston, MA 02215. USA.}{padma\_sk@bu.edu}

\researchsupported{The first-named author was supported by grant no 912/18 from the Israel Science Foundation. 
The second-named author was supported by DFG grant MU 4110/1-1 and by an NWO Vidi grant.
The third-named author was supported by Simons Foundation grant 546235 for the collaboration ``Arithmetic Geometry, Number Theory, and Computation'' and by NSF DMS 2401547. 
We would like to thank the NWO DIAMANT cluster for supporting a visit of the first- and
third-named author to the
University of Groningen.
}

\usepackage{url}
\usepackage[usenames,dvipsnames]{xcolor}

\usepackage{url,latexsym, amscd, amsfonts, eucal, mathrsfs, amsmath, amssymb, amsthm, rotating, graphicx}

\usepackage{tikz,caption,stmaryrd,amsfonts,amssymb,systeme,comment,graphicx,colonequals,faktor,xfrac}
\usepackage{xypic, braket}
\usepackage{colonequals} 

\usepackage{tikz-cd}

\newtheorem{thm}{Theorem}[section]

\newtheorem{prop}[thm]{Proposition}

\newtheorem{lemma}[thm]{Lemma}
\newtheorem{cor}[thm]{Corollary}

\theoremstyle{definition}

\newtheorem{defn}[thm]{Definition}

\theoremstyle{remark}

\newtheorem{rk}[thm]{Remark}
\newtheorem{ex}[thm]{Example}

\newcommand\into{\hookrightarrow}
\newcommand\Q{\mathbb{Q}}

\newcommand\Qp{\mathbb{Q}_p}

\newcommand\Qpb{\overline{\mathbb{Q}}_p}
\newcommand\Qqb{\overline{\mathbb{Q}}_q}

\newcommand\ba{\underline{\alpha}}
\newcommand\C{\mathbb{C}}

\newcommand\G{\mathbb{G}}

\newcommand\A{\mathbb{A}}
\newcommand\Z{\mathbb{Z}}
\newcommand\R{\mathbb{R}}

\newcommand\NS{\mathop{\rm NS}\nolimits}

\renewcommand{\L}{L}
\newcommand{\M}{M}
\renewcommand{\P}{\mathcal{P}}
\renewcommand{\O}{\mathcal{O}}

\newcommand\Spec{\mathop{\rm Spec}\nolimits}

\renewcommand\O{\mathcal{O}}

\newcommand\ord{\mathop{\rm ord}\nolimits}
\newcommand\supp{\mathop{\rm supp}\nolimits}

\newcommand\res{\mathop{\rm Res}\nolimits}

\newcommand{\End}{\operatorname{End}}
\newcommand{\Hom}{\operatorname{Hom}}
\newcommand\isom{\cong}

\newcommand\frka{\mathfrak{a}}

\newcommand{\Div}{\operatorname{Div}}
\newcommand{\rank}{\operatorname{rk}}
\newcommand{\Pic}{\operatorname{Pic}}

\newcommand{\im}{\operatorname{im}}
\newcommand{\id}{\operatorname{id}}

\renewcommand{\div}{\operatorname{div}}
\newcommand{\hdr}{H_{\textup{dR}}}
\newcommand{\hhdr}{H^{\textup{dR}}}
\newcommand{\Symm}{\operatorname{Sym}}

\newcommand{\p}{\mathfrak{p}}
\newcommand{\q}{\mathfrak{q}}

\newcommand{\dd}{d}

\newcommand{\ocol}{\Omega_V}
\newcommand{\ocola}{\Omega_{V,1}}
\newcommand{\ocols}{\Omega_{\textup{st}}}
\newcommand{\ocolas}{\Omega_{\textup{st},1}}
\newcommand{\oloc}{\Omega_{\textup{loc}}}
\newcommand{\acol}{\O_V}

\newcommand{\inject}{\hookrightarrow}
\newcommand{\kst}{K_{\textup{st}}}
\newcommand{\delbr}{\bar{\partial}}
\newcommand{\Curve}{\operatorname{Curve}}
\newcommand{\Span}{\operatorname{Span}}
\newcommand{\Gm}{\mathbb{G}_\textup{m}}
\newcommand{\selmervariety}{H^1_f(G_\p,U_Z)}
\newcommand{\globalselmer}{H^1_f(G_\Q,U_Z)}

\DeclareMathOperator{\ch}{ch}                                       
\DeclareMathOperator{\Res}{Res}
\DeclareMathOperator{\Mor}{Mor}
\DeclareMathOperator{\picm}{Pic^{rn}}

\begin{document}
\setcounter{secnumdepth}{3} 

\begin{abstract} \setlength{\parskip}{1ex} \setlength{\parindent}{0mm}

  We give a new construction of $p$-adic heights on varieties over number fields
  using $p$-adic Arakelov theory. 
  In analogy with Zhang's construction of real-valued
  heights in terms of adelic metrics, these heights are given in terms of 
  $p$-adic adelic metrics on line bundles. 
  In particular, we describe a construction of canonical $p$-adic 
  heights on abelian varieties and we show that we recover the canonical
  Mazur--Tate height and, for Jacobians, 
  the height constructed by Coleman and Gross.
  Our main application is a new and simplified approach to the Quadratic Chabauty method
  for the computation of rational points on certain curves over the rationals, by pulling back the canonical height on the Jacobian with respect to a carefully chosen line bundle. 
  We show that our construction allows us to reprove, without using $p$-adic Hodge
  theory or arithmetic fundamental groups, several results due to Balakrishnan and Dogra. Our method also extends to primes $p$ of bad reduction.
    One consequence of our work is that for any canonical height ($p$-adic or
  $\R$-valued) on an abelian variety (and hence on pull-backs to other
  varieties), the local contribution at a finite prime $q$ can be constructed using $q$-analytic methods. 
\end{abstract}



\tableofcontents

 \hfill

\section{Introduction}\label{sec:intro}
  The explicit solution of polynomial equations in rationals or integers is one of the oldest
  problems in mathematics. Even the simplest case, that of finding rational (or integral)
  points on a smooth projective algebraic curve $C$ over the rationals, does not have a satisfactory solution except in special
  cases. If the genus of $C$ is at least $2$, then one knows, thanks to the celebrated theorem of
  Faltings~\cite{Fal83}, that the set $C(\Q)$ of rational points is finite. But the theorem does not
  give an effective way of bounding its size, let alone a computationally feasible way of
  finding it.

  Chabauty's theorem~\cite{Chab41}, made effective
  by Coleman~\cite{Col85a}, raised the hope of finding such a
  method using $p$-adic techniques. If the Mordell--Weil rank $r = \mathrm{rk}(J/\Q)$ of the Jacobian $J=J_C$ of $C$ is smaller than
  the genus $g$ of $C$, it shows that $C(\Q)\subset C(\Q_p)$ is contained in the zero set
  of 
  a locally analytic function that can often be made explicit. Chabauty proved that this zero set is finite.
  When $p$ is a prime of good reduction, 
  Coleman identified the relevant function as a $p$-adic line  integral
  in the sense of~\cite{Col85} and used this to  give
  an effective bound on the number of points. 
  The method of Chabauty--Coleman can often be used to compute $C(\Q)$ in practice;
  see~\cite{MCP12} for an exposition.
   All of the above and, at least in principle, all that follows, extends to the study of $K$-rational points on a
  curve $C/ K$, where $K$ is a number field; see for instance~\cite{Sik13}
  and~\cite{BBBM21}. 

 The revolutionary paper of Kim~\cite{Kim05} outlined a program to push
the method of Chabauty and Coleman to the case of curves which
  do not satisfy the Chabauty bound $r<g$.
  There it
  was demonstrated for the first time that one may use the arithmetic theory of the fundamental
  group to identify more general Coleman integrals, namely, iterated integrals, which would
  vanish on $C(\Q)$. Shortly afterwards, Kim conjectured~\cite{Kim09} that one could in fact recover
  $C(\Q)$ completely using these methods. Verifying these conjectures is the main open problem of
  the subject. Nevertheless, a more practical problem is to recover $C(\Q)$ using $p$-adic
  methods, and (being slightly more restrictive and precise), trying to find $C(\Q)$ inside the
zero set of some iterated Coleman integral. (For a different approach toward
making the theorem of 
Faltings effective that does not use Coleman integrals see~\cite{LV20};
another method that does achieve effectivity is discussed
in~\cite{Alp20}. Both approaches have not led to practical algorithms.) 

One method of this sort is Quadratic
Chabauty~\cite{BBM16, BBM17, BBBM21, BD18, BD21, DLF19, BDMTV19,BBBLMTV19,  BDMTV2,
AABCCKW, EL21, CLYX}. 
In its simplest form, it requires $r=g$, but also has additional
assumptions on the curve.
Stated broadly (but still not covering the method of geometric Quadratic
Chabauty~\cite{EL21, CLYX}), Quadratic Chabauty can be phrased as an equality of functions
on $J(\Q)$ of the form $Q=H$. Here, $Q$ is in fact a local function 
\begin{equation*}
  Q\colon J(\Q_p)\to \Q_p\,.
\end{equation*}
In contrast, $H$, which usually comes from some sort of a $p$-adic height, is global in
nature. Fixing an Abel--Jacobi embedding $\iota\colon C\to J$ over $\Q$, one has
a non-canonical decomposition 
\begin{equation*}
  H\circ \iota= \sum_{q<\infty} \lambda_q\;,
\end{equation*}
where $q$ runs through all finite places of $\Q$ and $\lambda_q$ is a function $C(\Q_q)\to
\Q_p$.
To obtain the equality $Q=H$ one first proves that $H$ is a quadratic function. One may
then use this to write $H$
in terms of a basis of such functions which can be extended to $J(\Q_p)$. 
For instance, when the closure of $J(\Q)$ has finite index in
$J(\Q_p)$, one can write $H$
as a quadratic polynomial in the linear functionals on $J(\Q)$ coming from abelian logarithms (which are Coleman integrals of
holomorphic forms) 
to obtain a local function $Q\colon J(\Q_p)\to \Q_p$ such that 
$Q$ restricts to $H$ on $J(\Q)$. (If the Coleman integrals of a basis of 1-forms are linearly dependent, then regular Chabauty applies and the problem is easier.) 
By pulling back along $\iota$, one obtains an equation,
satisfied for $x\in C(\Q)$, 
\begin{equation}\label{eq:qc}
  Q\circ \iota(x)-\lambda_p(x)= \sum_{q<\infty,\, q \ne p} \lambda_q(x)\;.
\end{equation}
For running the Quadratic Chabauty method, we need two finiteness
properties. (i) We need the left hand side to be a locally analytic function $C(\Q_p)\to \Q_p$ with finite
fibers. (ii) The function $H$ needs to be constructed in such a
way that the right hand side of the above equation takes values in a finite computable set $T$.
Typically, $\lambda_q$ is going to be identically $0$ for all primes of good reduction and for primes of
bad reduction $\lambda_q$ can take a finite computable set of values.

In the first instance of Quadratic Chabauty~\cite{BBM16, BBM17} the curve $C$ was
hyperelliptic, given by an equation $y^2=f(x)$ with $f\in \Z[x]$ of 
degree $2g+1$, satisfying $r=g$, the embedding $\iota$ was 
taken via the point at infinity and the function $H$ was simply the $p$-adic height.
However, while~\eqref{eq:qc} was satisfied for $x\in C(\Q)$, the method only worked
for integral points (with respect to the given model). 
The reason was that $\lambda_q(x)$ can be expressed as an intersection multiplicity on a regular model
of $C$, and it was possible to rewrite the value of $\lambda_q$ for
$q$-integral points  $x\in C(\Q_q)$ in terms of the intersection matrix of
the special fiber, and compute these values ahead of time. This procedure
to rewrite and precompute the possible values of $\lambda_q$ does not
extend to rational points that are not $q$-integral. In fact, $\lambda_q$
takes on an infinite set of values on $C(\Q_q)$.

The next substantial step was taken by Balakrishnan and Dogra
in~\cite{BD18}. On the one hand, this important paper related the
theory of~\cite{BBM16} with the arithmetic of the fundamental group of $C$, by constructing $H$ using an
extension of the Nekov\'a\v{r} height pairing~\cite{Nek93} via an appropriate nonabelian
unipotent quotient of the fundamental
group. On the other hand, it showed how additional geometric data such as a nontrivial correspondence on $C$ could allow the method to be extended to recover the rational points and
not just the integral points. This extension of Quadratic Chabauty has been used since to find
the rational points in a number of cases that resisted resolution using previous
methods,
most importantly in the case of the Cartan modular curve of
level~13~\cite{BDMTV19}, related to Serre's uniformity problem about
the image of residual Galois representations associated to elliptic curves. Other modular examples have
been computed using this method in~\cite{BBBLMTV19, AABCCKW, BDMTV2, AM,
ACKP}.

It is important to note that the surprising parts of the extension to Quadratic Chabauty in~\cite{BD18} and subsequent work such as~\cite{BD21} are the two finiteness properties stated below equation~\eqref{eq:qc}. Indeed, given any endomorphism of $J$ (or equivalently, a correspondence on $C$), one can simply consider the $p$-adic height pairing between a point of $J$ and its image under the endomorphism, but in general, there is no hope that this will satisfy the finiteness property when pulled back to $C$.


The paper~\cite{BD18} makes full use of both the theory of the fundamental group \`{a} la Kim and
Nekov\'a\v{r}'s cohomological height pairing. However, in the end, the resulting function $H$ is fairly
simple -- when pulled back to $C$, it can be expressed as the height
pairing applied to two explicit divisors
constructed out of a point $x\in C(\Q)$ and the correspondence.
A natural problem that arises is thus to explain the results of~\cite{BD18} in a more direct way
without the detour via the theory of the fundamental group. This problem was the starting
point of the present work.
We provide a solution based on $p$-adic Arakelov theory, developed by the
first-named author in~\cite{Bes05}.
One
important consequence is that the function $\lambda_p$, which one needs to compute and locally
expand as a power series in order the make the method work, is now (essentially) explicitly given as an
iterated Coleman integral. Another side effect is that the theory extends without any further
work to the case of primes $p$ of bad reduction, using the theory of Vologodsky
integration~\cite{Vol03}
instead of Coleman integration. 

To give our version of Quadratic Chabauty, we replace the use of the Nekov\'a\v{r}
construction of the 
height pairing by a new $p$-adic analogue of the theory of
adelic metrics and associated heights due to Zhang~\cite{Zha95} which we believe
could be of independent interest. Let us recall that an adelic metric on a line bundle $\L$ over a
variety $X$ over a number field $K$ assigns to each (finite or infinite) place $v$ of
$K$ a 
real-valued norm $\|\cdot\|_v$ on the completion $\L_v$, compatible with a set of absolute
values on $K$ satisfying the product formula and subject to some natural conditions. These
guarantee, among other things, that the associated height 
\begin{equation*}
  h_{\L}(x) = \sum_v \log \|u\|_{v} \in \R\,,
\end{equation*}
where $u$ is any nonzero vector in the fiber of $\L$ above $x$, is a finite sum, hence
well-defined.
This height decomposes non-canonically into local contributions by picking a nonzero section (which
could in principle be only set-theoretical) $s$ of $\L$, giving
\begin{equation}\label{ZhangDecomp}
  h_{\L}(x) = \sum_v \log \|s(x)\|_{v} \in \R\,.
\end{equation}

Following Tate's method for constructing canonical heights on abelian varieties and
similar methods for arithmetic dynamical systems, Zhang showed how to obtain canonical
adelic metrics for arithmetic dynamical systems. Suppose that $f\colon X\to X$ is an endomorphism of a
variety $X$ over a number field $K$, $\L$ is a line bundle on $X$ and $\beta\colon \L^{\otimes d
}\to f^\ast \L$ is an isomorphism with $d>1$; we will refer to this as a
{\em dynamical} situation. Then, for any place $v$ of $K$, starting with
any norm $\|\cdot\|_v$ on $\L_v$, repeatedly replacing $\|\cdot\|_v$ with $(\beta^{-1}f^\ast
\|\cdot\|_v)^{\frac{1}{d}}$ and taking a limit, we obtain a canonical norm on $\L_v$ for which
$\beta$ is an isometry, and these together give a canonical adelic metric. As an example, if
$X = (E,P_\infty)$ is an elliptic curve given by a Weierstrass equation with point at
infinity $P_\infty$, $f=[2]$ is the multiplication by $2$ map, $\L= \O(2P_\infty)$ and $d=4$,
one obtains an adelic metric on $E$ whose associated height is easily recognized as the
standard 
canonical height on $E$. In many important cases, for instance when $X$ is an abelian
variety and $f$ is multiplication by $2$, 
the canonical norm for a
non-archimedean place $v$ is given in the form $\|u\|_v= c^{v_L(u)}$
where $1>c \in \R$ and $v_L$ is a $\Q$-\textit{valuation} on 
$\L_v$ -- a $\Q$-valued function $v$ on $L^\times_v$, the total space of $L_v$ without the zero section. See
Definition~\ref{D:val}, and note that 
valuations have been studied by Betts in~\cite{Bet17}, where they are
called {\em log-metrics}.

To construct a $p$-adic analogue of the theory above, we begin by replacing the system of
absolute values satisfying the product formula (more precisely, the log of this system) by an
id\`{e}le class character 
\begin{equation*}
  \chi = \sum_v \chi_v\colon \mathbb{A}^\times_K/K^{\times}\to \Q_p\;.
\end{equation*}
The logs of the norms on a line bundle are replaced by {certain functions}
$\log_v\colon \L^\times_v \to \Q_p$. In analogy with logs of norms, these
satisfy the relation $\log_v(a w)= \chi_v(a)+ \log_v(w)$ for any $a\in
K_v$ and nontrivial $w\in \L_x$. There is an obvious notion of isometry in the
theory; furthermore, the functions $\log_v$ and the associated heights behave functorially.

As is often the case, it is easy to see that the limiting method used to construct canonical norms in the classical
theory will not work in the $p$-adic case and a different method 
has to be used. For finite places $v\nmid p$ one can use valuations, such as the ones
provided by the classical theory as before, to obtain $\log_v$. Indeed, since for
such a place the character $\chi_v$ factors via $\ord_v$, it is easy to see that
setting $\log_v= \chi_v(\pi_v) \cdot v_L$ gives a function with the desired properties.

Even this method fails for places above the prime $p$. One of the key observations of the
present work is that in this case one may use $p$-adic Arakelov
theory~\cite{Bes05} to obtain canonical metrics. More
precisely, we use the theory of $p$-adic log functions and their curvature forms. For
simplicity, let us restrict now to varieties $X$ over $\Q_p$. The theory isolates among all log functions a class of nice log functions which
are iterated integrals (in the sense of Vologodsky) of a particularly simple form. To those log functions on a line
bundle $\L$ on $X$ the theory associates a curvature form $\alpha\in
\hdr^1(X)\otimes \Omega^1(X)$ with the property that $\cup \alpha = \ch_1(\L)\in \hdr^2(X)$.
Furthermore, the condition about the cup product of $\alpha$ is necessary and sufficient for
the construction of a log function on $\L$ with curvature $\alpha$. Unlike the classical
archimedean
theory, this log function is now no longer uniquely determined by $\alpha$ 
up to a constant, but can be changed by adding
the integral of a form $\omega\in \Omega^1(X)$. This suggests a way of finding a canonical log
function for the dynamical situation with $ \beta \colon \L^{\otimes d
}\to f^\ast \L$ described above:
 First find a curvature form $\alpha$
cupping to $\ch_1(\L)$ and satisfying $f^\ast \alpha = d \cdot \alpha$ (this equality
will hold after applying the cup product, so when the kernel of the cup product
is not so big this is quite reasonable -- for instance, for abelian
varieties). 
This makes the isomorphism $\beta$ an isometry up to the integral of
a holomorphic form. Now adjust the log function by such an integral to make
$\beta$ an isometry.

We consider this procedure in the case that $X=A$ is an abelian variety and $f$ is the
multiplication by $2$ map. We prove that it gives  log functions at places above $p$ that
produce heights which are quadratic for symmetric line bundles and linear for antisymmetric
line bundles. We call such a log function \textit{good}. However, unlike the classical real-valued theory, the resulting log functions are not unique. For a symmetric line bundle
$\L$ there is a unique good log function for each curvature form cupping to
$\ch_1(\L)$ (see Theorem~\ref{T:symgood}),  whereas
for an antisymmetric line bundle, every log function with trivial curvature
is good by Theorem~\ref{T:antigood}. To get
canonical log functions, and consequently canonical heights on abelian
varieties over number fields, we pick an appropriate log
function on the Poincar\'e line bundle on $A\times \hat{A}$. This depends on a choice, well known
in the theory of $p$-adic heights, of a complementary subspace to $\Omega^1(A_v)$ inside
$\hdr^1(A_v)$ for every place $v$ above $p$, see Proposition~\ref{P:cangood}. From this we obtain canonical adelic metrics and
canonical heights for antisymmetric line bundles by restricting to the relevant fiber above
$\hat{A}$ and to symmetric line bundles by pulling back via an appropriate map $A\to A\times
\hat{A}$, see Definition~\ref{D:arblog}.

It is a natural question how our $p$-adic height relates to other constructions in the
literature. We show in~\S\ref{subsec:MT} that on any abelian variety, our canonical $p$-adic height fits into
the general framework of 
$p$-adic heights pairings due to Mazur--Tate~\cite{MT83}.
For Jacobians, we show in Theorem~\ref{T:CGcomp} that we recover the pairing constructed by Coleman and
Gross~\cite{CG89} and hence, via the comparison results of~\cite{Bes04, Bes17}, the
pairing of Nekov\'a\v{r}~\cite{Nek93} when the curve has semistable reduction at all
places above $p$. We will show in future work  that this remains true for general
abelian varieties.

Having the theory of $p$-adic heights in place, we have a simple description of Quadratic
Chabauty when $r=g$, which we now outline.
Let $\iota \colon C\to J$ be an embedding over $\Q$ of the curve into its
Jacobian. Suppose that there exists  a line
bundle $\L$ on $J$ such that $\iota^\ast \L$ is the trivial line bundle $\O_C$ and such that
$L$ itself is not algebraically equivalent to~0. The function $H$ mentioned
above
is going to be $H=h_\L$, the canonical height associated to the line bundle $\L$. Thus,
$H$ is a
quadratic function by our theory of heights. On the other hand, the local components of
the canonical height
pull back to functions $\log_v$
on $\O_C$, so we have a
decomposition on $C(\Q)$:
\begin{equation*}
  H\circ \iota = \sum_v \lambda_v\,,
\end{equation*}
where $\lambda_v(x) = \log_v \circ 1(x)$ and $1$ is a nowhere vanishing
section of $\O_C$.
Let   $T = \{\sum_{q\ne p}l_q\cdot \chi_q(q)\}$, where $l_q$ runs through the values that the
  function 
  $\lambda_{q}$ takes on $C(\Q_q)$.
Our main result is the following:
\begin{thm}\label{T:QCA}
  Suppose that $\rank J(\Q)=g$ and that the closure of $J(\Q)$
  has finite index in $J(\Q_p)$. Then the function 
  $$F\colonequals h_{\L} \circ \iota - \log_p\circ 1\colon C(\Q_p)\to \Q_p$$
  is locally analytic and 
takes values on $C(\Q)$ in $T$. Moreover, $T$ is finite and
  for every $t\in T$, there are only finitely many $x\in C(\Q_p)$ with $F(x)=t$.
\end{thm}
See Theorem~\ref{T:QC} for a more precise statement.
If we can make all quantities in Theorem~\ref{T:QCA} explicit, then we obtain a finite
subset of $C(\Q_p)$ containing $C(\Q)$. We note that the two finiteness properties stated below ~\eqref{eq:qc} follow from the fact that the height function with respect to the line bundle $\L$ on $J$ pulls back to a height function on the trivial bundle on the curve. In \S\ref{S:nonconstant}, we show that our method in fact extends to the more
general setting $$\rank J(\Q) < g+\rank \NS(J)-1\,,$$ as in~\cite{BD18}. In order to prove that
the relevant locally analytic functions are nonconstant (so they have
finitely many zeroes), we generalize the $\delbr d$ differential operators
introduced in \cite{Bes05} to certain higher order differential operators
$D_k$ in \S\ref{S:higherdelbar}.

One advantage of the approach to Quadratic Chabauty presented in this paper is that
the function $\lambda_p$ is just $\log_p(1)$, where $\log_p$ is a log function on  
the trivial bundle with easily computed curvature form. 
As explained before, this
determines it up to the integral of an unknown holomorphic differential. An important
observation is that knowledge of this differential is not needed in order to make the method
work. We explain how this function can
be explicitly written as a Coleman or Vologodsky iterated integral.
In future work, we plan to implement this approach in practice, using algorithms
for single~\cite{BKK10, Bal15, BT20, KK21, Kay} and double integrals~\cite{Bal13}.

For $q\ne p$, we show that the function $\lambda_q$ factors through the reduction graph of
$C\otimes \Q_q$, by relating it to a N\'eron function on $J\otimes \Q_q$ with respect
to $L\otimes \Q_q$ (see Proposition~\ref{P:factor}). In
particular, $\lambda_q$ vanishes when $C$ has potentially good reduction at $q$. 
In general, it is a difficult problem to compute 
the possible values of
$\lambda_q$. This is analogous to the situation in~\cite{BD18, BDMTV19}; an approach
to this problem via Chabauty--Kim theory and harmonic analysis on graphs
was developed by Betts and Dogra~\cite{BD}, see
also~\cite[Theorem~3.2]{BDMTV2}. However, these results do not give a
direct algorithm to
compute the functions $\lambda_q$ and their images for $q\ne p$ in the setting
of~\cite{BD18}. Consequently, most examples for which the rational points were
computed using Quadratic Chabauty had the convenient property that there is
a global semistable regular model with irreducible special fibers, which
directly implies that $\lambda_q$ vanishes identically for all $q$ (but see
\cite[Theorem~1.4, \S8.3.1]{BD18},~\cite{Bia20} and~\cite{BP24} for
formulas for $\lambda_q$ and its image in the
case of bielliptic genus~2 curves, and \cite[\S4.4]{BDMTV2}. See also the
    recent paper \cite{BDRHS}, where Betts, Duque-Rosero, Hashimoto and
    Spelier give an algorithm for the computation of local heights for
    hyperelliptic curves using the work of Coleman--Iovita~\cite{CI99} and
    the recent machinery of cluster pictures introduced in~\cite{M2D2}). 

Motivated by this, we give a new construction of 
canonical valuations on abelian varieties using Vologodsky integration. This gives
an essentially uniform way to construct canonical metrics at all places. We
believe that this could be of independent interest, also for canonical
metrics at non-archimedean places in classical or non-archimedean Arakelov
theory. 
In future work, we will give algorithms to compute the (possible values of the) functions $\lambda_q$ in practice using this approach.

Our approach is somewhat similar to work of Colmez in~\cite{Col98}. Colmez
constructs $p$-adic integrals and $p$-adic  Green functions 
on abelian varieties
and, by pullback, on more general varieties.
By~\cite[Appendix~B]{Bes05} this theory is essentially equivalent to the
theory of log functions. For the special case of the theta divisor on
a Jacobian, he defines a certain kind of Green function, so-called
symmetric Green functions, and he uses these to give a new analytic
construction of local height pairings on curves for all finite places, 
recovering local Coleman--Gross height pairings and local N\'eron symbols.
In Theorem~\ref{T:loggreen}, we show that Colmez's symmetric Green functions correspond  to canonical log
functions, with the same auxiliary choices. Hence we recover Colmez's
construction of local height pairings.

We also discuss the relation between our Quadratic Chabauty approach and
the one due to Balakrishnan--Dogra~\cite{BD18}. 
Via the above-mentioned comparison between our new $p$-adic height and the one of
Nekov\'a\v{r}, we show in Proposition~\ref{P:CGcomp} that the global function $H$ is in
fact the same in both constructions up to a
constant factor, and in Theorem~\ref{T:localcomparison} we show that the (non-canonical) local decompositions also agree up to a constant factor if we make compatible auxiliary choices for the two sides. 
As a consequence, we get a new proof,
not using Chabauty--Kim theory, that the local contributions away from $p$
in~\cite{BD18} take only finitely many values and vanish for potentially
good primes.

It is natural to compare our construction to the geometric approach to
Quadratic Chabauty due to Edixhoven--Lido (see~\cite{EL21}), which we now very briefly summarize.
Edixhoven--Lido also use line bundles on the Jacobian $J$ that are trivial when pulled
back to the curve $C$. Hence their requirements for Quadratic Chabauty are the same as ours. They use such line bundles to construct a map from $C$ to a $\mathbb{G}_m$-torsor $T$ over
$J$, spread out over $\Z$. The torsor $T$ is pulled back from the Poincar\'e torsor over $J$.
Edixhoven--Lido then carefully parametrize points lying in the closure
of the integral points $T(\mathbb{Z})$ inside the $\mathbb{Z}_p$-points of
$T$ using $p$-adic formal and analytic geometry. They then pull functions
vanishing on the curve in $T$ along this parametrization back to $C$ to
bound the number of rational points on each residue disc. In \cite[\S9.3]{EL21},
they explain how their approach for writing down equations that cut out
$C(\mathbb{Q})$ inside $C(\mathbb{Q}_p)$ using $T$ is related to the theory
of $p$-adic heights. Our approach circumvents writing down formal analytic
coordinates on higher dimensional varieties such as $T$, by explicitly writing down iterated integrals on the curve $C$ using curvature forms for the line bundle on $J$, once again purely expressed in terms of differential forms on the curve. We have not attempted a direct comparison between their
construction and ours, but note that in~\cite{DRHS}, Duque-Rosero, Hashimoto and Spelier
have recently shown that the subset of $C(\Q_p)$ cut out
using~\cite{EL21} is always a subset of the points cut out
using~\cite{BD18}. By our Theorem~\ref{T:localcomparison}, the same is true
for the points cut out by our construction.

The outline of the paper is as follows. In Section~\ref{sec:vals} we reformulate the theory
of (logarithms of) metrics over non-archimedean local fields, especially of
canonical metrics on abelian varieties, in terms of the notion of valuations.
Section~\ref{sec:background} summarizes the necessary $p$-adic Arakelov theory developed
in~\cite{Bes05}, with a focus on the case of curves. We then use this theory to construct
canonical log functions on abelian varieties over $\Q_p$ in Section~\ref{sec:log}, the
central section of the paper. The global theory of adelic $p$-adic metrized line bundles
and their associated heights is presented in Section~\ref{sec:global}. We then give
comparison results between our heights and those of Mazur--Tate and (in the case of
Jacobians) Coleman--Gross in Section~\ref{sec:comparison}. Section~\ref{sec:qc} contains
our approach to Quadratic Chabauty; in particular, we prove Theorem~\ref{T:QCA}.
We relate the global and local functions (coming from $p$-adic heights) used in our
approach to Quadratic Chabauty to those due to Balakrishnan and Dogra factor in
Section~\ref{sec:BDcomparison}. Our construction of canonical valuations
using Vologodsky integration is contained in Section~\ref{sec:VolVal}.

\acknowl{
We would like to thank Alexander Betts, Francesca Bianchi, Netan Dogra, Bas Edixhoven, Robin de Jong,
Eric Katz and Klaus K\"unnemann for helpful discussions, and Pierre Colmez for
comments on a first version of this paper.
We thank the anonymous referees for particularly careful and useful
reviews.}

\subsection{Notation}\label{S:notation}
We call a variety $X$ over a field $K$ {\em nice} if it is smooth, projective and
geometrically integral. We write $x\in X$ as a
shorthand for $x \in X(\overline{K})$, where $\overline{K}$ is an algebraic closure of $K$. 
For a line bundle $\L/X$, we let $\L^\times$
denote the complement of the zero section in the total space of $\L$. A
\textit{rigidification} of $\L$ is a choice of an element $r \in
L_{x_0}(K)$ for some ${x_0} \in
X(K)$. An \textit{isomorphism of rigidified line bundles} on $X$ is an isomorphism of the
underlying bundles that maps the rigidification on one side to the rigidification on the
other under the induced map on total spaces. \textit{Tensor products and pullbacks of
rigidified line bundles} can analogously be defined. (See \cite[9.5.6]{BG06}.) 

For every prime $p$, we fix an embedding of $\Qp$ into a fixed algebraic closure $\Qpb$. Let $\ord_p$ be the extension of the discrete valuation of $\Q_p$ to $\Qpb$. 

\section{Valuations and canonical local heights away from $p$}\label{sec:vals}
In this section, $X/\Qqb$ denotes a smooth proper variety and $\L$ a line bundle on
$X$. 

\begin{defn}\label{D:val}
  A {\em valuation on $\L$} is a function $v_{\L} \colon \L^\times(\Qqb) \to \Qqb$ such
  that for every fiber $\L_x$
    we have $$v_{\L}(\lambda u) = \ord_q(\lambda )+ v_{\L}(u)\,$$
  for every nonzero $u \in \L_x$, and
  every $\lambda \in \Qqb$. 
    A {\em $\Q$-valuation on $\L$} is a valuation on $\L$ with values in $\Q$.
\end{defn}

\begin{rk}
 When $\L$ and $X$ are defined over a local field $K$ with $\Q_q \subset K \subset \Qqb$, the function $v_{\L}$ is invariant under the
  action of the absolute Galois group $G_K$ of $K$, as we now explain. Pick $x$  in $X(K)$ and a non-vanishing
  section $u$ in $L_x(K)$. Any point in $L_x(F)$ for some extension field
  $F$ of $K$ contained in $\Qqb$ is of
  the form $\lambda u$ for some $\lambda \in F$,  and $\sigma(\lambda u) =
  \sigma(\lambda ) u$ for every $\sigma\in G_K$.
  Now $$v_{\L}(\sigma(\lambda u)) = \ord_q(\sigma(\lambda ))+ v_{\L}(u) = \ord_q(\lambda )+
  v_{\L}(u) = v_{\L}(\lambda u)\,.$$
\end{rk}
\begin{ex}\label{E:alg_val} Let $X$ be a smooth proper variety defined over a local field $K$, and let $\L$ be a line bundle on $X$. Let $R$ be the ring of integers of $K$.  Let $\mathcal{L}$ be a line bundle
  over a proper, flat reduced scheme $\mathcal{X}$ over $R$. Let $X$ be the generic fiber
  of $\mathcal{X}$, and $\L \colonequals \mathcal{L}_K$ the corresponding line bundle over
  $X$. We now define a valuation $v_{\mathcal{L}}$ on $\L$ called the {\em model
  valuation} associated to $\mathcal{L}$.

Let $F$ be a finite extension of $K$, with ring of integers $R_F$ and $x \in X(F)$.  
   Let $s$ be a nonzero (algebraic) meromorphic section of $\L$ such that $x \notin \div(s)$. By
   flatness, there is a unique extension $s_{\mathcal{L}}$ of $s$ to $\mathcal{L}$, and a
   unique 
   section $\bar{x} \colon \Spec R_F \rightarrow \mathcal{X}_{R_F}$ extending $x$. Note that
   $\mathrm{H}^0(\Spec R_F, \bar{x}^*\mathcal{L})$ is a free rank $1$ module over $\Spec R_F$,
   with a rational section $\bar{x}^*s_{\mathcal{L}} \in
   \mathrm{H}^0(\Spec R_F, \bar{x}^*\mathcal{L}) \otimes_{R_F} F \cong F$ ,
   so it makes sense to talk about the valuation of the element $\bar{x}^*s_{\mathcal{L}}$ measured with respect to the rank $1$ free module $\bar{x}^*\mathcal{L}$. For $u = s(x)$,
  define 
  $$
    v_{\mathcal{L}}(u) \colonequals \textup{ valuation of the rational section } \bar{x}^*s_{\mathcal{L}} \textup { of  } \bar{x}^*\mathcal{L}\,.
  $$
  Then $v_{\mathcal{L}}$ is a
  $\Q$-valuation on $\L$. 
  See~\cite[Example~2.7.20]{BG06} for more details. 
\end{ex}
More generally, we can obtain $\Q$-valuations as follows.
\begin{ex}\label{E:real_val}
Let  $\|\cdot\|$ be a locally bounded and continuous real-valued metric on
  $\L$ (see~\cite[\S2.7]{BG06}). Let $\log_{\R} \colon \R_{>0} \rightarrow \R$ denote the usual logarithm on the positive real numbers $\R_{>0}$.
Then $\log_\R \|\cdot\|$ has the scaling property that we want from a valuation. So if 
\begin{equation}\label{E:real_Qval}
  v_{\|\cdot\|}(u) \colonequals -\log_\R \|u\|\cdot {\log_\R(p)}^{-1}
\end{equation} 
  is $\Q$-valued, then it defines a $\Q$-valuation on $\L$.
\end{ex}

\begin{defn}
  Let $f\colon X'\to X$ be a morphism of smooth proper varieties over $\Qqb$. Let $v_{\L}$ be
  a valuation on $\L$. Then we define the {\em pullback valuation} $f^*v_{\L}$ on
  $\L'\colonequals f^*\L$ as follows:  For every $x' \in X'$, there is a canonical identification of fibers $\L'_{x'} \cong \L_{f(x')}$, which glue to give the map $\tilde{f} \colon \L'^\times\to \L^\times$ in the commutative diagram below.
   \begin{equation}\label{}
  \xymatrix{
    \L'^\times\ar[d] \ar[r]^{\tilde{f}}& \L^\times\ar[d] \\
      X' \ar[r]^f & X\\
  }
    \end{equation}
  Then
$$
  f^*v_{\L} \colonequals v_{\L}\circ \tilde{f}\,.
$$ 
\end{defn}

\begin{defn}
  Let $v_{\L}$ be a valuation on $\L$ and let $\M$ be another line bundle on $X$ with a
  valuation $v_{\M}$.
  \begin{enumerate}
    \item We define the {\em sum valuation} on $\L\otimes\M$ by
      $$(v_{\L}+v_{\M})(u\otimes w) \colonequals v_{\L}(u) + v_{\M}(w)\,, $$ 
      where $u\in L_x$ and $w\in M_x$ are nonzero, and $x\in X$. 
    \item We call an isomorphism $g\colon \L \to \M$ an {\em isometry} (with respect to $v_{\L}$
      and $v_{\M}$) if $$v_{\M}\circ g = v_{\L}\,.$$
  \end{enumerate}
\end{defn}
\begin{rk}\label{R:Qval}
The class of $\Q$-valuations is closed under taking sums and pullbacks. 
\end{rk}

Now let $K/\Q_q$ be a finite extension and let $A/K$ be an abelian variety. 
\begin{rk}\label{R:rigidify}
Recall the notion of a rigidification of a line bundle from \S\ref{S:notation}. Suppose that $\L_1$ and $\L_2$ are isomorphic line bundles on $A$ with respective rigidifications $r_1$ and $r_2$. Then an isomorphism of the rigidified line bundles $(\L_1, r_1)$ and
$(\L_2,r_2)$ is an isomorphism $\varphi\colon \L_1\to \L_2$ such that $\varphi(r_1)=r_2$.
This isomorphism exists and is unique. Henceforth, we rigidify every line bundle over $A$
  by fixing a $K$-point in the fiber above~0, following~\cite[\S9.5]{BG06}.
\end{rk}

In particular, 
if $(\L,r)$ is a rigidified symmetric (respectively antisymmetric) line bundle on $A$, there is a unique  isomorphism $[2]^* \L \cong \L^{\otimes 4}$ (see ~\eqref{m-iso1}) (respectively $[2]^* \L \cong \L^{\otimes 2}$ (see ~\eqref{m-iso3})) of rigidified line bundles (where $[2]^*\L,\L^{\otimes 2}, \L^{\otimes 4}$ have the rigidifications induced from $\L$ via pullbacks and tensor powers), and by~\cite[Theorem~9.5.4]{BG06} there is a unique valuation such that this isomorphism is an isometry. This valuation also has the following nice properties.

\begin{prop}\label{P:good}(\cite[Theorem~9.5.7]{BG06},~\cite[Lemma~3.1]{Bet17})
  For every rigidified line bundle $(L,r)$  on $A$, there is a unique valuation
  $v_L$ on $L$ with the following properties:
  \begin{enumerate}[\upshape (a)]
    \item $v_L$ only depends on $(L,r)$ up to isomorphism of rigidified line bundles.
    \item\label{P:goodadd} $v_{L\otimes M} = v_L + v_M$.
    \item\label{P:TrivVal} $v_{\O_A}(x, a) = \ord_q(a)$ if we choose the rigidification~1 of $\O_A$.
    \item\label{P:valfunc} $v_{\varphi^*L} = \varphi^*v_L$ for homomorphisms $\varphi\colon A'\to A$ of
      abelian varieties over $K$.
    \item $v_L$ is locally constant on $L^\times(K)$. 
    \item $v_L$ is $\Q$-valued and has bounded denominator on $L^\times(K)$. 
    \item $v_L(r) = 0$.
  \end{enumerate}
\end{prop}
  We call $v_L$ the {\em canonical valuation}
  associated to $(L,r)$. 
Betts calls the canonical valuation the {\em N\'eron log-metric} in~\cite{Bet17}. We have chosen different terminology to avoid confusion with the log functions discussed in the next few sections.
We call the valuation on $\O_A$ in \eqref{P:TrivVal} the \textit{trivial
valuation}.

\begin{rk}\label{R:change_rig}
It is easy to see that changing the rigidification $r$ to a rigidification $r'$ changes
  the canonical valuation by the constant $\ord_q(\lambda)$, where $r'=\lambda r$.
  See~\cite[Remark~9.5.9]{BG06}.
\end{rk}

\begin{rk}\label{R:can_real}
  In~\cite[\S9.5]{BG06}, canonical valuations are constructed using a dynamical approach.
See~\cite[Example~8.15]{GK17} for a construction of canonical valuations based on tropical
geometry.
\end{rk}

\begin{rk}\label{R:good_good}
  If $A$ has good reduction, then canonical valuations are model valuations on the N\'eron
  model by~\cite[Example~9.5.22]{BG06}.  In general, canonical valuations are not model
  valuations; they are not even induced by a formal model, see~\cite{Gub03}.
\end{rk}

\section{$p$-adic Arakelov theory}\label{sec:background}
In this section we recall the parts of Vologodsky (and Coleman) integration theory and $p$-adic Arakelov theory that will be used in later sections.  The main result is Proposition~\ref{P:excurve}, which associates a certain $p$-adic analytic function called a log function to a line bundle equipped with a curvature form. Log functions  (Definition~\ref{D:logfunction}) are the $p$-adic analytic analogue of the valuations in Definition~\ref{D:val}, and will be used in Section~\ref{sec:global} to define the local contribution at $p$ in the decomposition of the global $p$-adic height function. Essentially everything we need can be found in Sections 2 and 4 of \cite{Bes05}.

\subsection{Vologodsky functions}
Coleman's integration theory \cite{Col82,Col85}, originally developed as a theory of iterated integration on
certain overconvergent spaces with good reduction over closed subfields of $\C_p$, was
recast as a theory of canonical paths in fundamental groupoids of the same spaces in
\cite{Bes02}, and a theory of Coleman
functions is derived from the theory of paths. Shortly afterwards~\cite{Vol03}, canonical paths for the
fundamental groupoid of arbitrary smooth varieties over finite extensions of $\Q_p$ were
constructed by Vologodsky, and the associated theory of Vologodsky functions, which we now
recall, was derived in \cite[Section 2]{Bes05}.
Note that Vologodsky functions below are called (Vologodsky) Coleman functions in \cite{Bes05}.

\begin{rk}\label{locandef}
 In what follows, when we refer to a locally analytic function, we mean a
  function $f$ such that every point $x$ has a neighborhood, in the
  $p$-adic topology, and a power series centered at $x$ converging on that
  neighborhood to $f$. For curves, a locally meromorphic function (or section of a line bundle) is a function (or section) such that every point $x$ has a punctured neighborhood in the $p$-adic topology and a finite to the left Laurent series that converges to $f$ on that neighborhood.
\end{rk}

Let $K$ be a finite extension of $\Q_p$, with a choice of embedding $K \rightarrow \overline{K}$ into an algebraic closure. Let $X$ be a smooth, geometrically connected
algebraic variety over $K$. We first summarize the results in \cite{Bes05}. For this we
fix a branch, denoted $\log$, of the $p$-adic logarithm. We insist that it takes $K$-values on
$K^\times $.

\begin{thm}\label{T:volsum}
Let $X$ and $\log$ be as above.
\begin{enumerate}[\upshape (a)]
\item \label{t1} For any locally free sheaf $\mathcal{F}$ on $X$ there is a $K$-vector space $\mathcal{F}_V(X)$ of
Vologodsky functions with values in $\mathcal{F}$, and in particular Vologodsky functions and
differential forms, 
  $\acol(X)=\ocol^0(X)$ and $\ocol^i(X)$, and differentials $\dd\colon  \ocol^i(X) \to \ocol^{i+1}(X)$,  such that the sequence  
\begin{equation}\label{dexact}
    0\to K \to \acol(X)\xrightarrow{\dd} \ocol^1(X)
  \xrightarrow{\dd} \ocol^2(X)
\end{equation}
    is exact.
  There are products $\ocol^i(X)\otimes \ocol^j(X)\to \ocol^{i+j}(X)$ compatible with the differentials.
  \item \label{t2} Let $\Omega^i(X)$ be the space of global holomorphic
    $i$-forms on $X$, and let $\oloc^i(X)$ (\cite[Definition~3]{Bes02}) be
    the space of locally analytic $\overline{K}$-valued $i$-forms on $X$,
    as in the penultimate paragraph of page 321 in \cite{Bes05}. There are embeddings 
  \begin{equation}\label{ocolinj}
    \Omega^i(X)\inject \ocol^i(X) \inject \oloc^i(X)
  \end{equation}
  compatible with
  differentials and products.
  \item \label{t3} Let $z$ be a coordinate on $\mathbb{A}^1$. 
    There exists a function $\log(z)\in \acol(\mathbb{A}^1-\{0\})$ with the property that
  $d\log(z) = \frac{dz}{z}$ in $ \Omega^1(X)\inject \ocol^1(X) $ and that as a locally analytic function it is our chosen branch of the $p$-adic
  logarithm (with values in $\overline{K}$).
  \item \label{t4}  Let  $f\colon X\to Y$ be a map of varieties. Then there are pullback maps $f^\ast$ on spaces
  of Vologodsky functions and differential forms which are compatible with the differentials,
  the product structure and with the embeddings~\eqref{ocolinj}.
  \item \label{t5} The induced presheaves $U\to \ocol^i(U)$ on the Zariski site of $X$ are sheaves.
  \item \label{t6} There is an $\acol(X)$-submodule $\ocola^1(X)\subset \ocol^1(X)$ of ``one time iterated forms" (i.e. forms that locally look like
      $\sum_i \omega_i \int \eta_i$ for forms $\omega_i,\eta_i \in \Omega^1(X)$) and a ``delbar"
  operator $\delbr\colon \ocola^1(X) \to \Omega^1(X) \otimes \hdr^1(X)$ sitting in a short exact
  sequence
    \begin{equation}\label{delbarexact}
    0\to \Omega^1(X)\to \ocola^1(X) \xrightarrow{\delbr} \Omega^1(X) \otimes \hdr^1(X)\;,
  \end{equation}
  such that when $X$ is affine, the sequence is exact on the right and $\delbr (\omega \int
  \eta) = \omega \otimes [\eta]$ for any two forms $\omega,\eta$ on $X$,
  with $[\eta]$ the cohomology class of $\eta$.
\end{enumerate}
\end{thm}
\begin{proof}
This follows from Section 2 of~\cite{Bes05}. We point to the relevant references. 
The product structure and the differential in \eqref{t1} and the interpretation as locally
analytic functions in~\eqref{t2} are all discussed at the
bottom of page 321.
The exact diagram in~\eqref{t1} is Theorem 2.3.
Functoriality as in~\eqref{t4} is discussed at the bottom of page 322.

  In Vologodsky's theory with a variable branch of the logarithm (see \S~\ref{subsec:VarVol} for more details), the integral of $\frac{dz}{z}$ is the universal logarithm $\log_{\mathrm{univ}}$.

\begin{equation}\label{E:loguniv}
  \log=\log_{\mathrm{univ}}: \overline{K}^\times\to  \overline{K} \oplus
  \overline{K}\log(p),\; \log_{\mathrm{univ}}(p)=
  \log(p)
\end{equation}
where $\log(p)$ is a formal variable. This is
proved in~\cite[Theorem~1.16(5)]{Vol03}. We get \eqref{t3} by specializing
  to our chosen branch.
Returning to~\cite{Bes05}, the sheaf property~\eqref{t5}  is Proposition~2.6 there, and the $\delbr$
operator is recalled immediately following the proof of this proposition, with~\eqref{t6} being
Proposition~2.7.
\end{proof}

\begin{rk}\hfill
  \begin{enumerate}[\upshape (i)]
    \item The space $\O_V(X)$ consists of locally analytic functions that locally look
      like iterated Vologodsky 
            integrals. We will discuss these integrals in more detail in the special case
            of curves below.
    \item 
Compared to~\cite{Bes05}, we have chosen to reverse the order of the terms $\Omega^1(X)
\otimes \hdr^1(X)$ in (f).
\item 
  When $X$ has good reduction, Vologodsky functions essentially 
      coincide with Coleman functions and the
      resulting theories of iterated integrals are the same in this case. (See
      \cite[Remark~2.13]{Bes05}.) 
    \item The space $\O_V(\Spec(K))$ is identified with $K$ by taking the value at the
      underlying physical point. In particular, by functoriality, the values of Vologodsky
      functions on $K$-rational points of $X$ are always in $K$.\label{KtoK}
      \item\label{R:AbtVol} We stress that Vologodsky integration is a theory for algebraic varieties and not rigid analytic varieties.
  \end{enumerate}
  \label{Volthmcom}
\end{rk}

\subsection{Log functions for line bundles with curvature forms}\label{subs32}
Let $\L$ be a line bundle over $X$ defined over $\Qpb$. 

\begin{defn}\label{D:logfunction}\cite[Definition~4.1]{Bes05}
 A \emph{log function} on $\L$ is a function $\log_{\L} \in \acol(\L^{\times})$ 
 that satisfies the following two conditions. 
\begin{itemize}
 \item For any $x\in X(\Qpb)$, any nonzero $u$ in the fiber
$\L_x$ and a nonzero constant $\lambda \in \Qpb$ one has $\log_\L(\lambda u)= \log(\lambda)+\log_\L(u)$.
 \item $d \log_\L\in \ocola^1(\L^\times)$.
\end{itemize}
The pair $(\L,\log_{\L})$ will be called a \emph{metrized line bundle} on $X$, and
  $\log_{\L}$ a metric on $\L$. 
\end{defn}

The second condition in the definition of log functions allows one to associate a certain
curvature form to log functions, analogous to the construction of metrized line bundles
over $\R$. The key result about $p$-adic log functions is \cite[Proposition~4.4]{Bes05},
which shows that log functions are determined (up to the addition of the integral of a holomorphic form) by their corresponding curvature forms:
   \begin{prop}\label{P:excurve}                                          
  Suppose that $X$ is proper. Let $\pi\colon \L^\times \to X$ be the projection.
  \begin{enumerate} [\upshape (a) ]
  \item
    Suppose that $\overline{\L} = (\L, \log_{\L})$ is a metrized line bundle on $X$ such
      that 
$$\ch_1(\L)\in \im\left(\cup\colon  \Omega^1(X)\otimes \hdr^1(X )\to         
    \hdr^2(X)\right).$$
      Then there exists a unique element,
      $\Curve(\overline{\L})\in \Omega^1(X)\otimes \hdr^1(X )$,                               
      such that $\pi^\ast \Curve(\overline{\L})=\delbr d \log_\L$. The element
      $\Curve(\overline{\L})$ is called the \emph{curvature form} of the metrized
      line bundle $\overline{L}$ and it satisfies the relation $\cup            
      \Curve(\overline{\L})=\ch_1(\L)$.     
  \item Conversely, suppose that $\L $ is a line bundle on
    $X$ and  that $\alpha\in \Omega^1(X)\otimes \hdr^1(X )$ satisfies $\cup                    
    \alpha = \ch_1(\L)$. Then there exists a log function $\log_\L$ on $\L$ such that the
      curvature of $(\L, \log_\L)$ is $\alpha$.
  \end{enumerate}          
\end{prop}
If $\overline{\L} = (\L, \log_\L)$, then we sometimes write $\Curve(\log_\L)$
for $\Curve(\overline{L})$, and we call
$\Curve(\log_\L)$ the curvature form of $\log_L$.
\begin{rk}\label{R:logfunctionuniqueness}
If $\log_\L$ and $\log'_\L$ are two different log functions for the same curvature form, then 
\[ d(\log'_\L-\log_\L) \in \pi^*\Omega^1(X) \subset \Omega^1(\L^\times) \subset \ocola^1(\L^\times), \]
since $\ker(\delbr) = \Omega^1(\L^\times)$ by Theorem~\ref{T:volsum}~(\ref{t1}), and since
  the difference of any two log functions is constant along fibers of $\pi \colon
  \L^\times \to X$ by the first defining property of a log function.
  In the case of curves, this can also be seen from the explicit
  construction of log functions we give in~\S\ref{subsec:logcurves}, where we solve for an (algebraic) meromorphic form with
  prescribed residues. The space of such meromorphic forms, if non-empty, is a torsor for
  $\Omega^1(X)$.  
\end{rk}
\begin{rk}
When $X$ is an abelian variety, then $\ch_1(\L)\in \im\left(\cup\colon  \Omega^1(X)\otimes \hdr^1(X )\to         
    \hdr^2(X)\right)$ for every line bundle $\L$. One can explicitly write down
    curvature forms for the Poincar\'e bundle (see
    Proposition~\ref{P:curvspace}), and this induces curvature forms,
    and hence log functions on every line bundle on $X$. (See Definition~\ref{D:arblog}.)
    Moreover, the curvature form determines the log function up
    to a linear form.
\end{rk}

\begin{ex}\label{E:trivmetric}[Trivial log function on the trivial bundle]
 For any variety $X$ as above, we have the trivial log function $\log^{\mathrm{triv}}_{\O_X}$ on $\O_X$ with curvature $0$ defined by
 \[ \log^{\mathrm{triv}}_{\O_X}(x,u) \colonequals \log(u), \]
 where we have used the isomorphism $\O_X^{\times} \cong X \times \mathbb{G}_m$ and the fixed branch of the $p$-adic logarithm. 
\end{ex}

\begin{defn}\label{D:relatedlogfunctions}\hfill
\begin{enumerate}
 \item (Tensor products,  \cite[Definition~4.3]{Bes05}) If $(\L,\log_{\L})$ and $(\M,\log_{\M})$ are two metrized line bundles with corresponding curvature forms $\alpha$ and $\beta$, then $\log_{\L} \otimes \mathrm{id}_M + \mathrm{id}_L \otimes \log_{\M}$ is a log function for $\L \otimes \M$ with curvature $\alpha+\beta$.
 \item\label{D:dividinglogfunctions}  (Roots of curvatures and log functions) 
 Let $\L$ be a line bundle and let $\M \colonequals \L^{\otimes m}$ for some nonzero integer $m$. If $\log_{\M}$ is a log function for $\M$ with curvature form $\alpha$, then we have an associated log function $\log_{\L} \colonequals \frac{1}{m} \log_{\M}$ for $\L$ with curvature form $\alpha/m$ defined as follows. Let $s \in \L^\times$. Then
\[ \log_{\L}(s) = \left( \frac{1}{m} \log_{\M} \right) (s) \colonequals  \frac{1}{m} \left( \log_{\M}(s^{\otimes m}) \right). \]
 \item (Pullbacks, \cite[Proposition~4.6]{Bes05}) If $(\L,\log_{\L})$ is a metrized line
   bundle on a smooth, geometrically connected variety $Y/K$ with curvature $\alpha$ and $f \colon X \rightarrow Y$ is a morphism, then $(f^*\L,f^*(\log_{\L}))$ is a metrized line bundle on $X$ with curvature $f^*\alpha$.
 \item (Isometries) Let $(\L,\log_{\L}), (\M,\log_{\M})$ be metrized line bundles on $X$, with an isomorphism of line bundles $\L \cong \M$. Let $\tilde{f} \colon {\L}^\times \rightarrow \M^\times$ be the induced morphism. We say that the isomorphism $L \cong M$ is an {\textit{isometry}} if $\tilde{f}^*(\log_{\M}) = \log_{\L}$. 
\end{enumerate} 
\end{defn}

\subsection{The case of curves}
We now provide more details on iterated integrals and log functions in the
one-dimensional case.
This is all we need for our application to Quadratic Chabauty in Section~\ref{sec:qc}.
In particular, we sketch how to explicitly construct a log function starting with a given
curvature form. 

\subsubsection{Iterated integrals as Vologodsky functions}\label{331}

Suppose that $X=C$ is one-dimensional. As there are no locally analytic $2$-forms on $C$ we
have $\oloc^2(C)=0$ and therefore the differential $d\colon \acol(C)\to \ocol^1(C)$ is
surjective. If $\omega_1,\ldots,
\omega_k \in \Omega^1(C)$, then we can iteratively 
define \begin{equation}\label{iterated-general}
  \int_x^z \omega_1 \circ \cdots \circ \omega_k = \int_x^z \left( \omega_1 \int_x^z \omega_2 \circ
  \cdots \circ \omega_k \right)\,,
\end{equation}
where $\int_x^z$ means the unique preimage under $d$ which vanishes at $x$. We can view this
iterated integral as the $v_0$-component of a solution of the differential equation 
\begin{equation*}
  d v_k=0,\; dv_{k-1}= \omega_k v_k,\; \ldots, dv_0 = \omega_1 v_1 
\end{equation*}
with $v_k=1$, or, in a different language, as the $v_0$-component of the horizontal section for
the connection 
\begin{equation}\label{basicconnection}
 \nabla(v) = dv- A\cdot v,\text{ with }  A= 
     \begin{pmatrix}
     0 & \omega_1 & 0 & \cdots & 0\\
     0 & 0 & \omega_2 & \cdots & 0\\
     \vdots&\vdots &  & \ddots & \\
     0 & 0 & 0 & \cdots & \omega_k\\
     0 & 0 & 0 & \cdots & 0
     \end{pmatrix}.  
\end{equation}

\subsubsection{Single and double integrals on curves}
Let $C/K$ be a nice curve. 
Let $x \in C$ and let $z$ be a coordinate in a Zariski neighborhood of
$x$. By abuse of notation, we also let $\log(z)$ denote the Vologodsky function defined on
a punctured neighborhood of $x$ obtained by pulling back the function in
Theorem~\ref{T:volsum}~\eqref{t3} by the coordinate $z$ and using
Theorem~\ref{T:volsum}~\eqref{t4}. We also suppress the choice of base point used to normalize iterated integrals below, and assume that both sides are normalized correctly so that equality holds.
\begin{lemma}\label{L:merointexist}\label{P:intmero}
 Let $\eta$ be an (algebraic) meromorphic form on $C$ and let $\omega \in \Omega^1(C)$. Let $c_x \colonequals \Res_x(\eta)$. 
Then there is an open neighborhood $V_x$ of $x$ in the $p$-adic topology, and a Laurent
  series $g$ centered at $x$ such that, with respect to a local parameter $z$ at $x$ on $V_x$,
  we have on $V_x \setminus \{x\}$ the equality $\int^z \eta =
  g(z) + c_x \log(z)$. Furthermore, $\res_x (g \omega)$ is independent
  of the choice of $g$, and if $c_x = 0$, then $\int^z (\omega \int^z \eta)
  = h(z) + \res_x (g \omega) \log(z)$ on $V_x \setminus \{x\}$ for some
  Laurent series $h$.
\end{lemma}
\begin{proof}
  For two parameters $z$ and $z'$ as above the function $\log(\frac{z}{z'})$ is analytic near
  $x$, which shows that the validity of the lemma is independent of the choice of $z$. Thus we
  may assume that $z$ is an algebraic uniformizer, defined on a Zariski open
  neighborhood $U_x$ of $x$ with no poles of $\eta$ and no poles or zeros of $z$ except possibly at $x$. Let $\eta
\colonequals \sum_{i\ge \ord_x(\eta)} a_i z^i dz$ be the Laurent expansion of $\eta$ on $V_x$
centered at $x$, so that $a_{-1} = c_x$. We note that such a Laurent series
  (and associated integrals below) are locally analytic sections in the
  punctured neighborhood $V_x \setminus x$, and as explained in the
  penultimate paragraph of page 321 of \cite{Bes05}, they converge on a sufficiently small punctured neighborhood of $x$ in the p-adic topology. Let $\eta' \colonequals \sum_{i<-1} a_i z^i dz$ and
$\eta'' \colonequals \eta - \eta' - c_x \frac{dz}{z}$. Observe that $\eta' = d f$ with $f(z)
\colonequals \sum_{i<-1} \frac{a_i z^{i+1}}{i+1} \in \acol(U_x \setminus \{x\})$, and so we
have $\int \eta' =f$ (up to adjusting $f$ by a constant) on $V_x \setminus \{x\}$. Also, note
that $\eta'' \in \ocol^1(U_x)$, so we have $\int \eta'' \in \acol(U_x)$ and this has a power
series expansion in $z$ on some analytic neighborhood $V_x$ of $x$ by part~\eqref{t2} of Theorem~\ref{T:volsum} and Remark~\ref{locandef}. 
  Setting $g \colonequals
\int \eta''+f$ , we see that on $V_x \setminus \{x\}$, the function $g$ has a Laurent series
expansion $g(z)$ and
 \[ \int^z \eta = \int^z \eta'' + \int^z \eta' + \int^z c_x \frac{dz}{z} = g(z) + c_x \log(z). \]

Since $dg = \eta-c_x \frac{dz}{z}$, and any two choices of $g$ satisfying this differ by a constant,
  $\Res_x(g \omega)$ is independent of the choice of $g$. When $c_x = 0$, we have $g \omega - \res_x (g \omega) \frac{dz}{z} \in \oloc^1(V_x \setminus \{x\})$, and by writing down an explicit Laurent series expansion with no residue and arguing as before, we can find a Laurent series $h$ centered at $x$ such that on $V_x \setminus \{x\}$ we have $dh = g \omega - \res_x (g \omega) \frac{dz}{z}$, and hence
\[ \int^z (\omega \int^z \eta) = \int^z g \omega = \int^z dh + \int^z \res_x (g \omega) \frac{dz}{z} = h(z) + \res_x (g \omega) \log(z).\qedhere  \]
\end{proof}

\begin{rk}
 Locally around any point $x \in C$, iterated Coleman integrals are also defined
 and are polynomials in $\log(z)$ with coefficients which are Laurent
 series~\cite[Section~5, p.~40, \textup{3rd paragraph}]{Bes02}. One can alternatively prove the lemma by instead using the local
  comparison between Vologodsky iterated integrals and Coleman iterated integrals. For the
  first case of the lemma~\cite{Bes-Zer13} suffices, whereas for the second part one also
  needs~\cite{Kat-Lit21}
\end{rk}

\begin{ex}\label{E:meroint}
Let $C/K$ be a nice curve.
\begin{enumerate}[\upshape (a)]
 \item\label{P:logf} For any $f \in K(C)$, invertible on an open subset $U$, using Theorem~\ref{T:volsum}~(\ref{t1},\ref{t3},\ref{t4}), it follows that there exists a function $\log(f) \in \acol(U)$, unique up to an additive constant, such that
$d\log(f) = df/f$ and such that it is equal to $\log \circ f$ as a locally analytic function. 
 \item\label{P:itint} Let $\omega \in \Omega^1(C)$ and let $\eta$ be a form of second kind on $C$, holomorphic on an open subset $U$ of $C$. 
 From Lemma~\ref{L:merointexist}, it follows that there are well-defined functions $f,g
    \in \acol(U)$ such that $dg = \eta, df = g\omega$, such that $g$ admits a Laurent
    series expansion around points in $C \setminus U$, and such that $f$ has an expansion
    as the sum of a Laurent series and a constant multiple of $\log(z)$ around points in
    $C \setminus U$. 
    Furthermore, in this case, Theorem~\ref{T:volsum}~(\ref{t6}) implies $\delbr d f
    = \omega \otimes [\eta] \in \Omega^1(U) \otimes \hdr^1(U)$.
\end{enumerate}
\end{ex}
\subsubsection{Log functions for curves}\label{subsec:logcurves}

Granting the existence of curvature forms for log functions, we first prove the following lemma that will be useful in the sketch of construction of log functions on curves that follows.
\begin{lemma}\label{L:iflogexists}
 Assume that $(\L,\log_{\L})$ is a metrized line bundle on a nice curve $C/K$. 
 Let $s$ be a section of $\L$, invertible on a Zariski open subset $U$. Then $\log_\L(s) \in \acol(U)$ and the form $d \log_\L(s) \in \ocol^1(U)$
 admits a locally meromorphic extension to all points $x \in C$, with at worst simple poles and such that $\res_x(d\log_{\L}(s)) = \ord_x(s)$ for every $x \in C$. 
\end{lemma}
\begin{proof}
  We have $\log_\L(s) \in \acol(U)$ as the pullback of $\log_\L$ via $s\colon U\to L^{\times}$. In
  particular, it is locally analytic on $U$ as claimed. 
 Let $x \in C \setminus U$, and let $z$ be a coordinate in a neighborhood of $x$. Let $n =
 \ord_x(s)$. Then $z^{-n}s$ is invertible in a Zariski open neighborhood $V_x$ of $x$, so in turn $\log_{\L}(z^{-n}s) \in \acol(V_x)$ and $\gamma_x \colonequals d\log_{\L}(z^{-n}s) \in \ocol^1(V_x)$. 
 Now by the second property of the log function, on $V_x \setminus \{x\}$, we have 
 \begin{equation}\label{zimn}
  \log_{\L}(s) = n \log(z) + \log_{\L} (z^{-n}s),
 \end{equation}
 and hence by definition of $\log(z)$,
 \[ d \log_{\L}(s) = n \frac{dz}{z} + \gamma_x. \]
 Since the $\gamma_x$ are locally analytic on $V_x$, the expression on the right hand side
 gives a locally meromorphic extension of $d \log_{\L}(s)$ to points $x \in C \setminus U$, as a form with a simple pole of residue $\ord_x(s)$.
\end{proof}

\begin{proof}[Sketch of construction of log functions for curves]
We briefly sketch a construction of a log function with a given curvature $\alpha$ when
  $C$ is a curve, assuming the existence of log functions, since this is what we will need
  for the application to Quadratic Chabauty. In the case of curves, the problem reduces to
  solving for an (algebraic) meromorphic differential $\gamma$ with prescribed polar parts, as we
  explain below. We refer the reader to the proof of \cite[Proposition~4.4]{Bes05} which
  shows that log functions exist in any dimension using a careful argument using \v{C}{e}ch cocycles.

  Let $\alpha = \sum_i \omega_i \otimes[ \eta_i] \in  \Omega^1(C) \otimes \hdr^1(C)$ 
 for some holomorphic forms $\omega_i \in \Omega^1(C)$ and forms of second kind $\eta_i$ with corresponding cohomology classes $[\eta_i] \in \hdr^1(C)$, 
 such that $\cup \alpha =\ch_1(\L)$. We know by Proposition~\ref{P:excurve} that a log function on $\L$ with curvature
 $\alpha$ exists and is unique up to the integral of a holomorphic form on $C$.
Pick a nonzero (algebraic) meromorphic section $s$ of $\L$, invertible on a Zariski open subset $U$ of
$C$ as in Lemma~\ref{L:iflogexists}. As we will see later, a log function on $\L$ is completely
  determined by $\log_{\L}(s)$, so to construct one log function with curvature $\alpha$
  (hence all), 
 it suffices to determine $\log_\L(s)$ up to the integral of a holomorphic form.
We now restrict further to some $U'\subset U$ where all the $\eta_i$ are holomorphic.
Since $\delbr d \log_{\L}(s) = \alpha|_{U'}$, and we know $\ker(\delbr)$ from
  Theorem~\ref{T:volsum}~\eqref{t6}.
  Example~\ref{E:meroint}~\eqref{P:itint} shows that on $U'$, we have
 \begin{equation}\label{E:logfromitint} d \log_{\L}(s) = \sum \omega_i \int \eta_i + \gamma, \end{equation}
 for some meromorphic form $\gamma$ on $C$ that is holomorphic on $U'$. Since $\log_\L(s)$ satisfies
 Lemma~\ref{L:iflogexists}, 
the polar parts of $\gamma$ at
  points in $U$ are exactly the negative of those of $\sum \omega_i \int \eta_i$ and the same
  is true at $x\in C\setminus U$ --  except that there is a difference in the logarithmic part, which is
  determined by the condition (that automatically holds also at $x\in U$), 
  $$\res_x(\gamma) = \ord_x(s)-\res_x\left(  \sum \omega_i \int \eta_i \right) \;.$$
  Because $\gamma$ is a meromorphic form and therefore satisfies the Residue Theorem, a
  necessary condition of the existence of $\gamma$ is that $\sum_{x \in C} \Res_x(\sum \omega_i \int \eta_i) =
  \sum_{x \in C} \ord_x(s)$, and by Riemann--Roch this is also a
  sufficient condition. It is easy to see, independently of the general theory of log
  functions, that this condition is indeed satisfied because 
  $$\sum_{x \in C} \Res_x(\sum \omega_i \int \eta_i) = \cup \alpha = \ch_1(\L)$$ by
  assumption and we know that $\sum_{x \in C} \ord_x(s) = \ch_1(\L)$ (see for example, \cite[Corollary~6.10, Proposition~6.11]{Hartshorne}). As the conditions above
  completely determine the polar parts of $\gamma$ at every point, the degree of freedom of
  $\gamma$ is exactly $\Omega^1(C)$, and therefore determining $\gamma$ with the right polar
  conditions is equivalent to
  determining $\log_\L$. 

It remains to show how $\log_{\L}$ is determined by $\log_\L(s) \in \acol(U)$.
By the very definition of log functions (Definition~\ref{D:logfunction}) it clearly determines $\log_\L$ above $U$ so it suffices to extend
$\log_\L$ to $x\in C\setminus U$.  For such an $x$ let $z$ be a local coordinate as in the proof of
Lemma~\ref{L:iflogexists}. Then \eqref{zimn} gives
$  \log_{\L} (z^{-n}s)= \log_{\L}(s) - n \log(z) $.
As $\log_\L$ satisfies the conditions of Lemma~\ref{L:iflogexists}, the function $ \log_{\L}(s) - n \log(z) $ is
analytic in an analytic neighborhood of $x$. Hence it extends to $x$ and we can set $\log_\L(z^{-n}
s)(x)=(\log_\L(s)- n \log(z))(x)$. It is easy to see
that this extension is independent of the choice of $z$. 
\end{proof}

\begin{ex}\label{E:logtangent}[Log functions for tangent bundles on hyperelliptic curves, \cite{BBM16}]
  Let $C \colon y^2=f(x)$ be an odd degree hyperelliptic curve 
  In \cite{BBM16}, an explicit
  log function for the tangent bundle of $C$ is constructed, extending the elliptic case
  treated in~\cite{BB14}. This is then applied to express a suitably normalized local
  Coleman--Gross height pairing in terms of double integrals, which is one of the main
  ingredients for the Quadratic Chabauty methods for integral points on $C$ when $f\in
  \Z[x]$  is monic and 
  $\mathrm{rk}(\text{Jac}(C)/\Q)=g(C)$.
  
 We recall the main features; for details see
  the proof of \cite[Theorem~2.2]{BBM16}. Let $\{ \omega_i \colonequals x^i dx/2y
  \}_{i=0}^{2g-1}$ be the standard basis for $\hdr^1(C)$. Let $\{ \overline{\omega_i}
  \}_{i=0}^{g-1}\subset \textup{Span}(\omega_0,\ldots,\omega_{2g-1})$ be
  forms whose cohomology classes $[\overline{\omega_i}]$ form a basis for a complementary subspace $W$ to $\Omega^1(C) \subset
  \hdr^1(C)$ dual to the standard basis $\{ \omega_i \}_{i=0}^{g-1}$ of $\Omega^1(C)$
  under the cup product pairing. We choose the curvature form $\alpha \colonequals -2 \sum_{i=0}^{g-1} \omega_i \otimes [\overline{\omega_i}] $ for the tangent bundle   $\mathcal{T}$. Fix the
  section $\theta$ of the tangent bundle dual to the holomorphic form $\omega_0$; this
  section has a pole of order $2g-2$ at the unique point at $\infty$ and no other zeroes
  or poles. Since both $d\log_{\mathcal{T}}(\theta)$ and the form $-2 \sum_{i=0}^{g-1}
  \omega_i \int \overline{\omega_i}$ have at worst simple poles at the unique point at
  $\infty$ and have the same residue at $\infty$, the meromorphic
  differential $\gamma$ we need to solve for is in fact holomorphic on $C$. It follows that for any $\omega \in \Omega^1(C)$,
  \[ \log_{\mathcal{T}}(\theta) =  -2 \sum_{i=0}^{g-1}  \omega_i \int \overline{\omega_i} + \int \omega, \]
  is a log function for $\mathcal{T}$ with curvature $\alpha$, and that every log function
  for $\mathcal{T}$ with curvature $\alpha$ is of this form for some $\omega \in \Omega^1(C)$.
\end{ex}

\subsection{Higher $\delbr$-operators on Vologodsky functions}\label{S:higherdelbar}
In this section, we develop the theory of higher order differential operators $D_k$ which
are analogues of the $\delbr d$ operators on once iterated Vologodsky functions. The
results of this section will only be used in \S\ref{S:nonconstant} to justify that our
Quadratic Chabauty method works in the more general setting of $\rank J(\Q) < g+\rank \NS(J)-1$. Readers only interested in understanding the theory in the more restrictive setting of $\rank J(\Q) = g$ and $\rank \NS(J) > 1$ may skip this section.

All the results are summarized in the following theorem, so the reader may
also choose to read its statement and skip the rest of the section. We will continue to use the notation introduced in Theorem~\ref{T:volsum}. Let $X$ be as in Theorem~\ref{T:volsum}.
\begin{thm}\label{hdelbth}
 There exists an increasing filtration $K=I^0\subset I^1\subset \cdots $ of $K$-vector spaces
  on the algebra $\acol(X)$ of all $\O_X$-valued Vologodsky functions on $X$ such that the   following hold: 
  \begin{enumerate}[\upshape (a)]
    \item \label{hd1} The filtration is compatible with multiplication, i.e., the multiplication map on $\acol(X)$ restricts to a map $I^k \otimes I^m \xrightarrow{\times} I^{k+m}$. 
    \item \label{hd2}There exist linear maps $D_k: I^k \to \hdr^1(X)^{\otimes k}$ such that
  the diagram 
  \begin{equation*}
   \xymatrix{
     I^k \otimes I^m \ar[r]^{\times} \ar[d]^{D_k \otimes D_m} & I^{k+m} \ar[d]^{D_{k+m}}\\
     {\hdr^1(X)}^{\otimes k}\otimes \hdr^1(X)^{\otimes m} \ar[r]^{\qquad \ast} & \hdr^1(X)^{\otimes k+m}
   }
  \end{equation*}
  commutes; here $\ast$ is the shuffle product defined as follows: 
        \begin{equation}\label{shuffle}
        (v_1 \otimes \cdots \otimes v_k)\ast  (v_{k+1} \otimes \cdots
         \otimes v_{k+m}) = \sum_\sigma v_{\sigma^{-1}(1)} \otimes \cdots
         \otimes v_{\sigma^{-1}(k+m)}\,,
      \end{equation}
      where the sum is over all $(k,m)$ shuffles $\sigma$, that is, permutations of $k+m$ elements that keep the internal order between each of the two groups of indices $\{1,2,\ldots,k\}$ and $\{k+1,\ldots,k+m\}$.
    \item\label{hd3} For each $k$, the kernel of $D_k$ contains $I^{k-1}$.
    \item\label{hd4}
      The space $I^2$ is the space of integrals of one time iterated forms (\eqref{t6} of Theorem~\ref{T:volsum}) and for $F\in I^2$ one has $\delbr
      dF \in \Omega^1(X)^{d=0}\otimes \hdr^1(X)$ and $D_2 F = [\delbr
      dF]$, where $\delbr$ is the delbar operator (same reference) and $[~]:
  \Omega^1(X)^{d=0}\otimes \hdr^1(X)\to \hdr^1(X)^{\otimes 2}$ is the map that
  takes the first coordinate to its cohomology class. 
  \end{enumerate}
The elements of $I^k$ are called
  \emph{$k$-iterated integrals}. 
\end{thm}

We proceed in two steps. In Definition~\ref{highdbdef}, we first define the operator $D_k$ on a slightly larger class of Vologodsky functions called $k$-iterated Vologodsky functions defined in Definition~\ref{D:miterated}, and establish some basic properties of the operator $D_k$ in Proposition~\ref{dkprop}. In Definition~\ref{D:kitint} we define the vector subspace $I^k$ of $k$-iterated integrals of the vector space of $k$-iterated functions, prove some structural results and combine these results to prove Theorem~\ref{hdelbth}.  

\subsubsection{$k$-iterated Vologodsky functions and the operator $D_k$}\label{subsec:Vologodsky_functions}
Let us begin by fixing some terminology. 

\begin{defn}\label{D:mextn}
A connection $(M,\nabla)$ is an \textit{$m$-extension} if there is a filtration $M=
M^0 \supset M^1 \supset \cdots \supset M^{m+1}=0$ in the category of vector
bundles with connection on $X$ such that the successive quotients
$M^i/M^{i+1}$ are all isomorphic to direct sums of the trivial connection $(\O_X,d)$.
\end{defn}

Recall the notion of an abstract Vologodsky function valued in a locally free sheaf as in
\cite[Definition~4.1]{Bes02}. 

\begin{defn}\label{D:miterated}
We will say that a
Vologodsky function with values in a sheaf $\mathcal{F}$ is \textit{$m$-iterated} if it is has a
representative $(M,v,s)$ where $M$ is an $m$-extension. In~\cite[Definition~5.4]{Bes02} this is
called a Coleman function of degree at most $m$. An \textit{$m$-iterated differential
form} is an
$m$-iterated Vologodsky function with values in the sheaf of differential forms (of some
degree). 
\end{defn}

We will define a map $D_k$ on $k$-iterated Vologodsky functions with values in a sheaf $\mathcal{F}$.
For this, let $\mathfrak{g}$ be the Lie
algebra of the de Rham fundamental group $\pi^{\mathop{dR}}_1(X)$ of
$X$. Recall that $\pi^{\mathop{dR}}_1(X)$ is the pro-algebraic group associated by Tannakian duality to the category of integrable unipotent connections on $X$ with the natural fiber functor $N \mapsto N_x$ at a $K$-rational point $x$. Recall from~\cite[\S1.5.2, Proposition~20]{Amn_Heidelberg} adapted to the de Rham situation, that we have an isomorphism $\mathfrak{g}/ \mathfrak{g}_1\isom \hhdr_1(X)$
where $\mathfrak{g}_1=[\mathfrak{g},\mathfrak{g}]$ and $\hhdr_1(X)=\Hom(\hdr^1(X),K)$ is the
first de Rham homology of $X$. By the proof of the above result in loc. cit., the isomorphism
is given by 
\begin{equation}\label{onttmat}
  g\mapsto l,\; g \text{ acts as } 
  \begin{pmatrix} 0 & l([N]) \\ 0 & 0 \end{pmatrix} 
  \text{ on } N_x, \text{ where } 0\to \O_X \to N \to \O_X\to 0
\end{equation}
and $[N]$ is the extension class of $N$ in the category of connections on $X$,  
\begin{equation}\label{hdr}
  \operatorname{Ext}^1((\O_X,d),(\O_X,d)) \isom \hdr^1(X)\;.
\end{equation}

Let $M$ be a $k$-extension. Thus, it has a filtration 
\begin{equation*}
  M= M^0 \supset M^1 \supset \cdots \supset M^{k+1}=0\;,
\end{equation*}
with graded pieces $M_i = M^i / M^{i+1}$ which are trivial. Let $V$ (respectively $V_i$,
respectively $V^i$) be the fiber at $x$ of $M$ (respectively $M_i$, respectively $M^i$).
Since the $M_i$ are trivial we have a natural isomorphism of connections $M_i \isom V_i \otimes
(\O_X,d)$. The cup product and the identification~\eqref{hdr}
induce an isomorphism 
\begin{equation*}
  \operatorname{Ext}^1(M_i,M_{i+1}) \isom \hdr^1(X) \otimes \Hom_{\nabla}(M_i,M_{i+1}) \isom \hdr^1(X) \otimes
  \Hom(V_i,V_{i+1})\;.
\end{equation*}
We recall that elements of $\mathfrak{g}$ act on the fiber at $x$ of every unipotent
connection, commuting with the maps induced on the fibers by morphisms of connections.
\begin{lemma}\label{lieaction}
  Let $g\in \mathfrak{g}$, acting on $V$. Then $g V^i \subset V^{i+1}$. The induced map $V_i
  \to V_{i+1}$ is obtained from the image $l$ of $g$ in $\hhdr_1(X)$ by applying $l$ to the
  class of the extension
  \begin{equation}\label{twostep}
    0\to M_{i+1}\to M^i / M^{i+2} \to M_{i}\to 0 
  \end{equation}
  in $\hdr^1(X)\otimes \Hom(V_i,V_{i+1})$.
\end{lemma}
\begin{proof}
  By assumption $g$ respects the filtration on $V$. As $g$ acts trivially on the fiber of $(\O_X,d)$, it follows that $g$ acts trivially on $V_i$, hence maps $V^i$ to $V^{i+1}$. The induced map $V_i \to V_{i+1}$ can
  be read off from the action on the fiber of~\eqref{twostep} and the result follows from the
  description~\eqref{onttmat} of the map
  $\mathfrak{g}\to \hhdr_1(X)$.
\end{proof}
\begin{cor}\label{vkcor}
  Let $M$ be a $k$-extension as above and let $g_1,\ldots,g_k \in \mathfrak{g}$. Then the
  composition $g_k\cdot g_{k-1} \cdots g_1$ is the same as the composition
  \begin{equation*}
    V \to V_0 \to V_1 \to V_2 \to \ldots \to V_k \to V\;,
  \end{equation*}
  where the first and last map are the canonical projection and injection while for $i\in
  \{0,\ldots,k-1\}$ the map
  $V_{i}\to V_{i+1}$ is induced by $g_{i+1}$ as in Lemma~\ref{lieaction} and in particular
  depends only on the image of $g_{i+1}$ in $\hhdr_1(X)$. Therefore the
  composition $g_k\cdot g_{k-1} \cdots g_1$ also only depends on the images $l_i$ of each $g_i$ in $\hhdr_1(X)$. 
\end{cor}
\begin{defn}\label{highdbdef}
  Let $F$ be a $k$-iterated Vologodsky function with values in $\mathcal{F}$ defined by the
  triple $(M,v,s)$, where $M$ is a
  $k$-extension. We define $D_k(F)\in \mathcal{F}(X)\otimes\hdr^1(X)^{\otimes k}$ as the
  functional on $\hhdr_1(X)^{\otimes k}$ with values in $\mathcal{F}(X)$, defined as follows:
  \begin{equation*}
    l_k \otimes \cdots\otimes l_1 \mapsto s(g_k \cdots g_1(v)) \in \mathcal{F}(X)\;.
  \end{equation*}
  Here, $g_i$ is any lift of $l_i$ to $\mathfrak{g}$ for $1\le i \le k$. We have 
  $g_k \cdots g_1(v) \in V_k$ by Corollary~\ref{vkcor}. Since $M_k$ is trivial we have an embedding
  $V_k \subset M_k(X)\subset M(X)$ and thus $ s(g_k \cdots g_1(v)) \in \mathcal{F}(X)$.
\end{defn}
\begin{prop}\label{dkprop}\hfill
  \begin{enumerate}[\upshape (1)]
    \item The value of $D_k F$ depends only on $F$ and not on the underlying representation of $F$ as a $k$-iterated Vologodsky function.
    \item \label{P:dkisolinear} The map $D_k$ commutes with $\O_X$-morphisms of the coefficient sheaves. In
      particular, it is $\O(X)$-linear.
    \item\label{P:prodtoshuffle} If  $F$ is $k$-iterated with values in $\mathcal{F}$ and $G$ is $m$-iterated with
      values in $\mathcal{G}$ then $F\cdot G$ is $k+m$-iterated with values in
      $\mathcal{F}\otimes \mathcal{G}$ and $D_{k+m}(F\cdot G) = (D_k F)\ast
      (D_m G)$ where $\ast$ is induced by the shuffle product~\eqref{shuffle} ${\hdr^1(X)}^{\otimes k}\otimes \hdr^k(X)^{\otimes m} \xrightarrow{\ast} \hdr^1(X)^{\otimes k+m}$ on the tensor algebra
      $T(\hdr^1(X))$ and the product $\mathcal{F}(X) \otimes \mathcal{G}(X) \to
      \mathcal{F}\otimes \mathcal{G}(X)$.
    \item\label{P:dD1} For $k=1$ and $\mathcal{F} = \Omega^1$, the map $D_1$ coincides with the $p$-adic delbar operator of
      Theorem~\ref{T:volsum}.
 \end{enumerate}
\end{prop}
\begin{proof}
  To see the independence of representation and functoriality in the
  coefficient sheaf we define, more generally, for any abstract Vologodsky
  function $F$ and elements $g_1,\ldots,g_k\in \mathfrak{g}$, the Vologodsky function
  $D(g_1,\ldots,g_k,F)$ with values in $\mathcal{F}$, to simply be the Vologodsky
  function with representation $(M, g_k \cdots g_1(v),s)$.
  As the action of the $g_i$ on the fibers
  commutes with the action of any morphism of connections by definition, it
  is immediate that $D(g_1,\ldots,g_k,F)$ depends only on the Vologodsky
  function associated with $F$ and not on the particular representation, and moreover, 
  is functorial in the coefficient sheaf.
  Under the assumptions of Definition~\ref{highdbdef} the element
  $g_k\cdots g_1(v)$ sits in the fiber of the trivial subconnection $M_k$ 
  and the Vologodsky continuation is then just the embedding $V_k\subset
  M_k(X)$, so that $D(g_1,\ldots,g_k,F)$ is precisely
  $D_k(F)(l_k\otimes\cdots \otimes l_1)$, hence the result.
    
The behavior with respect to multiplication is seen as follows: Suppose $F,G$ have
  representations with connections $M,N$, respectively, with fibers $V,W$ at $x$ respectively.
  Let $l_1,\ldots, l_{k+m}\in \hhdr_1(X)$ and lift those to $g_1,\ldots,g_{k+m}\in
  \mathfrak{g}$. The $g_i$ act on $V\otimes W$ as derivations: $g_i (v\otimes w) = (g_i
  v)\otimes w+ v\otimes (g_i w)$. Iterating this, we get
  \begin{equation*}
    g_{k+m} \cdots g_1 (v\otimes w) = \sum_{A,B} g_A v \otimes g_B w
  \end{equation*}
  where the sum is over all decompositions of $\{1,\ldots,k+m\}$ into a disjoint union of two
  sets $A,B$ and 
  \begin{equation*}
    g_{\{i_1<i_2<\cdots < i_r\}} = g_{i_r} \cdot g_{i_{r-1}}\cdot \cdots \cdot g_{i_1}\,.
  \end{equation*}
  If $A$ (respectively $B$) has size larger than $k$ (respectively $m$) then $g_A v$ 
  (respectively $g_B w$) is $0$. Thus, the sum is only over splittings such that $|A|=k$ and
  $|B|=m$. By Corollary~\ref{vkcor}, since $g_A v$ and $g_B w$ depend only on the images $l_i$ of $g_i$ in $\hhdr_1(X)$, setting
\begin{equation*}
    l_{\{i_1<i_2<\cdots < i_r\}} = l_{i_r} \otimes l_{i_{r-1}}\cdot \cdots \otimes l_{i_1} 
\end{equation*}
  we easily obtain 
  \begin{equation*}
    D_{k+m} \left( F\cdot G \right) (l_1 \otimes \cdots \otimes l_{k+m}) = \sum_{A,B} (D_k F)(l_A) \otimes (D_m
    G)(l_B)
  \end{equation*}
  from which the result \ref{P:prodtoshuffle} follows immediately. Finally suppose that $F$ is represented by
  $(M,y,s)$, where $M$ sits in a short exact sequence 
  \begin{equation*}
    0  \to M_1\to M \to M_0 \to 0
  \end{equation*}
  where $M_0$ and $M_1$ are direct sums of the trivial connection $(\mathcal{O}_X,d)$. 
  We may replace $M_0$ by the one-dimensional subconnection generated by the image of $y$ in
  $M_0$ without changing the Vologodsky function associated
  to the abstract Vologodsky function, 
  hence we may assume that $M_0\isom (\O_X,d) $. Suppose that $M_1$ is $r$-dimensional and
  $s_1,\ldots,s_r$ are coordinate functions on $M_1$ corresponding to a decomposition into a
  direct sum of $(\O_X,d)$'s. We can write the restriction of $s$ to
  $M_1$ as $\sum f_i s_i$, where $f_i \in \mathcal{F}(X)$. Both $D_1$ and the delbar operator
  depend in fact only on the restriction of $s$ to $M_1$ and commute with maps of the
  coefficient sheaves, as can easily be verified using the description of
  the $\delbr$ operator in \cite[\S6]{Bes02}. Thus, we may further assume that also $M_1\isom
  (\O_X,d)$ and that $s$ restricts to the identity of $M_1$. But now it follows from the definitions that both $D_1(F)$ and $\delbr(F)$ are simply given by the extension class of $M$. 
\end{proof}

\subsubsection{$k$-iterated Vologodsky integrals}\label{subsec:Vologodsky_integrals}
For the remainder of this section, by a Vologodsky function $F$, we mean a Vologodsky function with values in $\O_X$. We get a slightly different, and possibly simpler perspective when we
consider these Vologodsky functions.
\begin{defn}\label{D:kitint}
  A \textit{$k$-iterated integral} is a Vologodsky function $F$ 
  such that $dF$ is a $k-1$ iterated
  differential form. We let $I^k$ be the collection of  $k$-iterated
  integrals, which is easily seen to be a subspace of the vector space of $k$-iterated Vologodsky functions.
\end{defn}
\begin{rk}\label{intisnotfun}
Every $k$-iterated integral is clearly a $k$-iterated function. The
converse is false. Suppose for
example that $X=C$ is a non-proper curve and that $0\ne \omega \in \Omega^1(C)$ and $f\in \O(C)$
is a non-constant function. Then $\int \omega$ and $f\int \omega$ are both $1$-iterated functions, but $\int \omega$ is a $1$-iterated integral whereas $f \int \omega$
is only a $2$-iterated integral because $d(f \int \omega)= f\omega+ \int \omega df$ is not a
  holomorphic form.  
\end{rk}
\begin{prop}\label{iterdf}
  A Vologodsky function $F$ 
  is a $k$-iterated integral if and only if it has a representative of the form
  $(M,v,s)$ where $M$ sits in a short exact sequence 
  \begin{equation}\label{itinseq}
    0\to \O_X \xrightarrow{i} M \xrightarrow{p} M' \to 0
  \end{equation}
  with $M'$ a $k-1$ extension and $s: M\to \O_X$ is a splitting of the underlying sequence of
  vector-bundles. Moreover, there is a representation of $M$ as a $k$-extension where the graded piece $M_k$ is isomorphic to the image of $i$, and the section $s$ restricted to $M_k$ is a horizontal isomorphism.
\end{prop}
\begin{proof}
Suppose $dF$ is a Vologodsky form corresponding to the triple $(M',v',s')$ with $M'$ a
$k-1$-extension. Let us
recall the construction
in~\cite[Theorem~4.15]{Bes02} of a preimage under $d$ of $dF$. It is given by the vector bundle
$M= \O_X \oplus M'$ with connection given by $\nabla_M(f,m)= (df-s'(m),\nabla_{M'}m)$, together
with the horizontal map $i(f)= (f,0)$ and the section $s(f,m)=f$. Here, $v=(\alpha,v')$, where
$\alpha$ can be taken to be an arbitrary element of $K$ and causes an addition of a constant to $F$. Thus, as $dF$
determines $F$ up to an addition of a constant, $F$ itself is of this form, proving one
direction. Furthermore, in this case, we can define $M_k$ to be the image
of $i$ and since the section is the inverse of a horizontal isomorphism it
is horizontal as well.

  For the reverse direction,
suppose that $F$ comes from a short exact sequence~\eqref{itinseq}. Let $\nabla$ be the connection on $M$ and $\nabla^*$ the dual connection on the dual bundle $M^*$. As $s$ is a section of the flat $i$, it is easy to see that $\nabla^\ast s \circ i = 0$ so that $\nabla^\ast s$ factors via $M'$.
By definition of the differential~\cite[Definition~4.6]{Bes02}, it now follows that the form $dF$ is a $k-1$-iterated form given by $(M',v', \nabla^\ast s)$, where $v'$ is the image of $v$ under the map $M \rightarrow M'$.
\end{proof}

\begin{prop}\label{P:prod}
The product of a $k$-iterated integral with an $m$-iterated integral is an $m+k$-iterated integral
\end{prop}
\begin{proof}
This follows immediately from the definition upon noting that the product of a $k$-iterated form
with an $m$-iterated function is an $m+k$ iterated form, which is an easy special case of 
Proposition~\ref{dkprop}.
\end{proof}
\begin{prop}\label{P:Dkdesc}
  If $F$ is a $k$-iterated integral, then $D_k F \in \hdr^1(X)^{\otimes
  k}\subset \O(X) \otimes
  \hdr^1(X)^{\otimes k}$. Furthermore, we have 
  \begin{equation}\label{Dko}
    D_k F = [D_{k-1} dF]\;,
  \end{equation}
  where $[~]\colon
  \Omega^1(X)^{d=0} \otimes \hdr^1(X)^{\otimes k-1}\to \hdr^1(X)^{\otimes k}$ is the map that
  sends a form to its cohomology class on the first coordinate (in particular, $D_{k-1} dF$
  belongs to the source of this map).
\end{prop}
\begin{proof}
With the representation of $F$ obtained in Proposition~\ref{iterdf}, the section $s$ gives a horizontal isomorphism
$M_k \isom (\O_X,d)$, and therefore, in the notation of Definition~\ref{highdbdef}, we have that
     $s(g_k \cdots g_1(v))$ is a horizontal section of the trivial bundle, and hence can be be identified with a constant function. This proves the first assertion. 
     
     For the second assertion, we may assume we have the representation obtained from $dF$, as in the
     proof of Proposition~\ref{iterdf},  for $F$, so that $M_k\isom
     (\O_X,d)$ is split and $M_{k-1}\subset M'$.
We have the following: $(M')^{i}= M^i / M_k$ and $M_i^{\prime}= M_i$ for
$i<k$. It follows that we have a similar relation between the $V$'s.
     The extension $M^{k-1}$ is split as an extension of vector
     bundles. Therefore, by a well known argument (see for example~\cite[Proposition~1.1.3]{Chi-leS99} in the
     rigid context) the extension class
     \begin{equation*}
       [M^{k-1}]\in \operatorname{Ext}^1(M_{k-1},M_k) \isom \hdr^1(X) \otimes \Hom(V_{k-1},V_k)
     \end{equation*}
     is the cohomology class of the section $s'|_{M_{k-1}} \in
     \Hom(M_{k-1},\Omega_X^1) $. Let us compare the values of $D^k F$ and
     $[D^{k-1} dF]$ on 
     $g_1,\ldots,g_{k-1} \in \mathfrak{g}$. In the former we start with $v$ and
     first
     project onto $V_0$, while in the latter we start with $v'= v
     \pmod{V_k}$ and project onto $V_0^{\prime}=V_0$, clearly giving the same
     vector.  Then, in both cases we consecutively apply $g_1, g_2,\ldots,
     g_{k-1}$ to obtain a vector $w  \in V_{k-1}$. 
     When computing $D^{k-1} dF$, we now apply $s'$ to $w\in V_{k-1}\subset M_{k-1}(X)$. Then 
     we take the cohomology class so we finally obtain 
     \begin{equation*}
       [D^{k-1} dF (g_1,\ldots,g_{k-1})] = [s'(w)]
     \end{equation*}
     On the other hand, $D^{k} F(g_1,\ldots,g_{k-1})$ is the element of $\hdr^1(X)$ which, as a
     functional on $\mathfrak{g}$, send $g_k$ to $g_k(w)\in V_k\isom K$. But by
     Lemma~\ref{lieaction} this is just
     $[M^{k-1}](w)$ and the result follows.
\end{proof}
\begin{cor}\label{C:Dkonkitint}
  The iterated integral $F = \int \omega_1 \circ \cdots \circ
  \omega_k$ from~\eqref{iterated-general} is a
  $k$-iterated integral and $D_k F =  [\omega_1] \otimes \cdots \otimes [\omega_k]$ where
  brackets indicate cohomology classes. 
  \end{cor}
  \begin{proof}
  We prove this by induction on $k$, with the case $k=1$ being obvious.
  Suppose we know this for $k-1$. 
  We have $dF = \omega_1  \int \omega_2 \circ \cdots \circ
  \omega_k$ by definition. It follows from \eqref{P:dkisolinear} of
  Proposition~\ref{dkprop} that 
  $ \omega_1  \int \omega_2 \circ \cdots \circ \omega_k$
  is a $k-1$ iterated form and
  $ D_{k-1}(\omega_1  \int \omega_2 \circ \cdots \circ \omega_k)= \omega_1
  \otimes D_{k-1}(\int \omega_2 \circ \cdots \circ \omega_k)$.
  Therefore $F$ is
  a $k$-iterated integral and by Proposition~\ref{P:Dkdesc} we have $D_kF = [\omega_1] \otimes  D_{k-1}(\int \omega_2 \circ \cdots \circ
  \omega_{k})$, giving the result by the induction hypothesis.      
  \end{proof}
  
  \begin{rk}
  Now, if $f\in \O(X)$,
  then $D_{k+1}(f\cdot \int \omega_1 \circ \cdots \circ
  \omega_k)=f\otimes  [\omega_1] \otimes \cdots \otimes [\omega_k]$.
  Consider the functions of Remark~\ref{intisnotfun}. The function $\int
  \omega$ is a $1$-iterated integral and $D_1 \int \omega= [\omega]$. On
  the other hand, $f \int \omega$ is only a $2$-iterated integral and $D_2
  (f \int \omega)=0$, which follows either from the formula for $d (f \int
  \omega) $ and~\eqref{Dko} or from Proposition~\ref{dkprop} upon noticing that $f$ is a $1$-iterated integral
  and $D_1 f=0$.
\end{rk}

\begin{proof}[Proof of Theorem~\ref{hdelbth}]
 Part~(a) follows from Proposition~\ref{P:prod}. Since $k$-iterated
  integrals are a vector subspace of $k$-iterated functions, Part~(b)
  follows from Proposition~\ref{dkprop}~(\ref{P:prodtoshuffle}). Part~(c)
  holds since it is clear from the definitions that $D_k$ of any $k-1$
iterated function is $0$.  
  Part~(d) follows from
  Proposition~\ref{P:Dkdesc} and
Proposition~\ref{dkprop}~(\ref{P:dD1}). 
\end{proof}

\section{Canonical log functions on line bundles on abelian varieties}\label{sec:log}
Let $A$ be an abelian variety over a $p$-adic field $K$. 
The goal of this section is a construction of canonical log functions on $A$, similar to the
canonical valuations in Proposition~\ref{P:good}. 
Instead of the limit constructions typically applied in the real-valued setting, we will
use the curvature forms of metrized line bundles introduced in
Proposition~\ref{P:excurve}.

This idea is not new; for instance, Moret-Bailly has used curvature forms to characterize canonical real-valued metrics on
abelian varieties over $\C$ in~\cite{MB85}.
However, in that setting the curvature form determines the metric up to a constant, and
the canonical metric is characterized by having a translation-invariant curvature form in $\Omega^{1,1}(A_{\mathbb{C}})$.
In our setting, curvature forms are valued in $\Omega^1(A) \otimes \hdr^1(A)$ and hence
automatically translation-invariant, for instance by~\cite{Bar57}. 
Furthermore, the curvature form does not uniquely determine the log function: 
If $L$ is a line bundle on $A$, and $\alpha \in \Omega^1(A)\otimes
 \hdr^1(A)$ is such that $\cup\alpha = \ch_1(\L)$, then there are several possible log
 function $\log_\L$ with $\Curve(\log_\L)=\alpha$, one for each holomorphic form
 on $A$. We will use this degree of freedom in Theorem~\ref{T:symgood} to show that when
 $\L$ is symmetric, there exists a unique ``good'' log function (see Definition~\ref{D:goodmetric}),  on $\L$ with curvature
 $\alpha$. This is not true in the antisymmetric case, which is more involved, see Theorem~\ref{T:antigood}.
 To define canonical good log functions, we show that there is a choice of  curvature form
 $\alpha$ for the Poincar\'e
bundle on $A$ so that the good log function with curvature $\alpha$ induces good log functions on all symmetric and
antisymmetric line bundles on $A$ by pullback and restriction. The so obtained log
functions will be our $p$-adic analytic analogue of the canonical valuation. It will also serve as
the component at $p$ of a canonical $p$-adic adelic metrized line bundle when $A$ is
defined over a number field, see  Section~\ref{sec:global}.

\subsection{Geometry of abelian varieties}\label{sec:app}
We recall the theory of line bundles on abelian varieties, the dual abelian variety and
the Poincar\'e
line bundle. In this subsection $A$ is
an abelian variety over an arbitrary field $K$ of characteristic~0.
We write $s, d,\pi_i$ for the addition and subtraction maps and the projections $A\times A\to A$.
Let $\L$ be a line bundle on $A$.
\begin{defn}
    We call $\L$ symmetric (respectively antisymmetric) if there exists an
    isomorphism $[-1]^\ast \L \isom \L$ (respectively $[-1]^\ast \L \isom
    \L^{-1}$).
\end{defn}
\begin{rk}\label{R:symmfunctorial}
Pullbacks of symmetric line bundles by morphisms of abelian varieties are also symmetric, and the tensor product of two symmetric line bundles on an abelian variety is also symmetric.
\end{rk}

The following results are well-known.
\begin{prop}\label{sumdiff}\hfill
  \begin{enumerate}[\upshape (a)]
    \item\label{sumdiffsym} 
    If $\L$ is symmetric, then there is an isomorphism of line bundles on $A\times A$
\begin{equation}\label{m-iso}
  s^\ast \L \otimes d^\ast \L \isom (\pi_1^\ast \L)^2 \otimes (\pi_2^\ast \L)^2\;.
\end{equation}
\item\label{sumdiffanti} The line bundle $\L$ is antisymmetric
  if and only if the cohomology class of $\L$ (in de Rham cohomology) is trivial.
 In this case we have an isomorphism of line
bundles on $A\times A$:
\begin{equation}\label{m-iso2}
  s^\ast \L \isom \pi_1^\ast \L \otimes \pi_2^\ast \L\;.
\end{equation}

  \end{enumerate}
\end{prop}
\begin{proof}
  See~\cite[Proposition~5.2.4]{Lan83} for~\eqref{sumdiffsym} 
  and~\cite[Proposition~5.2.3]{Lan83} for Equation~\eqref{m-iso2}. It is  clear that if a
  line bundle is antisymmetric, then its cohomology class
  is 0 because  $-1$ acts as $1$ on $\hdr^2(A)$. Conversely, we may use the Lefschetz principle
  to reduce to the case of abelian varieties over $\C$ and replace de Rham cohomology with
  Betti cohomology. It then follows from the theorem of Appel--Humbert (\cite[Page 19]{Mum70}) that
  cohomologically trivial line bundles correspond to elements of $\Hom(\mathrm{H}_1(A,\Z),\mathbb{S}^1)$,
  were $\mathbb{S}^1$  is the unit circle inside $\C^\times$, and these are all clearly
  antisymmetric, which proves~\eqref{sumdiffanti}.
\end{proof}
\begin{cor}
If $\L$ is symmetric (respectively antisymmetric), then for any
  morphisms $f,g\colon X\to A$ of $K$-varieties, one
has
\begin{equation}\label{fg-iso}
  (f+g)^\ast \L \otimes (f-g)^\ast \L \isom (f^\ast \L)^2 \otimes (g^\ast \L)^2
\end{equation}
(respectively
\begin{equation}\label{fg-iso1}
  (f+g)^\ast \L \isom f^\ast \L \otimes g^\ast \L)\,.
\end{equation}
In particular, setting $f=g = \id$ or directly pulling back via $\Delta:A\to A\times A$, we get
isomorphisms
\begin{equation}\label{m-iso1}
  [2]^\ast \L \isom \L^{\otimes 4}
\end{equation}
when $\L$ is symmetric and
\begin{equation}\label{m-iso3}
  [2]^\ast \L \isom \L^{\otimes 2}
\end{equation}
when $\L$ is antisymmetric.
\end{cor}

We now recall the theory of the dual abelian variety and the Poincar\'e line bundle. We refer the interested reader to \cite[Chapter~8]{BLR} for the construction of the relative Picard scheme in full generality. We restrict to the setting of abelian varieties since that is all we need. For an abelian variety $A$, the Picard variety of $A$ represents the relative Picard functor 
\begin{equation*}
  S\mapsto \Pic_{A / S}\colonequals \Pic(A\times S) / \Pic(S)
\end{equation*}
where $\Pic$ is the usual Picard functor of isomorphism classes of line bundles and $\Pic(S)$ maps to $\Pic(A\times S)$ by pullback. See \cite[\S8.2, Theorem~3, \S8.4, Theorem~5]{BLR} for more details. The connected component of the
Picard variety represents
the subfunctor mapping $S$ to the isomorphism classes of line bundles which are fiber by fiber
antisymmetric. We denote this connected component by
$\hat{A}$ -- the \textit{dual abelian variety}. Universality provides a line bundle $\P$ on
$A\times \hat{A}$ whose restriction to $A\times \{0\}$ is trivial. We normalize $\P$
up to isomorphism by insisting that  its restriction to $\{0\}\times \hat{A}$ is also
trivial. 

\begin{defn}
  We call $\P$ the {\em Poincar\'e line bundle} on $A\times \hat{A}$. 
\end{defn}

\begin{lemma}\label{L:PoinIsSymm}
The Poincar\'e line bundle is symmetric.
\end{lemma}
\begin{proof}
 See~\cite[Theorem~8.8.4]{BG06}.
\end{proof}

\begin{prop}\label{phiprop}
Let $\L/A$ be a line bundle.
  \begin{enumerate}[\upshape (a)]
    \item The line bundle 
      \begin{equation*}
  \phi_{\L}(a) \colonequals t_a^\ast \L \otimes \L^{-1}
\end{equation*}
is antisymmetric for any $a\in A$. 
\item The map $\phi_L$ is induced from a morphism $\phi_L\colon A\to \hat{A}$ of abelian
  varieties.
  \item If $\L$ is ample, then $\phi_{\L}$
    is an isogeny.
    \item\label{P:phipropanti} 
The bundle $L$ is 
antisymmetric if and only if $\phi_{\L}$ is trivial. 
  \item\label{phiprope}   If $\L$ is symmetric, then there
    exists an isomorphism $\L^{\otimes 2} \isom (\id,\phi_{\L})^\ast \P$, with $(\id
    , \phi_\L) \colon  A \to A\times \hat{A}$.
  \end{enumerate}
\end{prop}
\begin{proof}
  For the first two assertions, see~\cite[Theorem~8.5.1]{BG06}. The third one
  is~\cite[Proposition~8.5.5]{BG06} and for the fourth see~\cite[Theorem~8.8.3]{BG06}.
  We give a short proof of~\eqref{phiprope}.
Let $\L$  be symmetric and let $\id \times \phi_\L\colon A\times A \to A\times \hat{A}$. Then $(\id,\phi_\L) = \Delta \circ (\id \times \phi_\L)$. From the
definition of $\phi_\L$, it follows that $$(\id \times \phi_{\L})^\ast \P\isom s^*\L\otimes \pi_1^*\L^{-1}\otimes
\pi_2^*\L'$$ for some line bundle $\L'$ on $A$. By restricting to $ \{0\}\times A$ it is easy
to see that $\L' \cong \L^{-1}$.
  By symmetry of $\L$ the right hand side in~\eqref{phiprope} becomes
  $$\Delta^*(s^*\L\otimes \pi_1^*\L^{-1}\otimes
  \pi_2^*\L^{-1})\isom [2]^*\L\otimes \L^{\otimes -2} \isom \L^{\otimes 2}\;.\qedhere$$
\end{proof}
We recall some standard facts about quadratic functions and the theorem of
the cube. Our reference is~\cite[\S8.6]{BG06}. Let $q\colon M\to N$ be a
function  between two abelian groups. We say that $q$ is
{\em quadratic} if the associated function 
\begin{equation*}
  b(x,y)=q(x+y)-q(x)-q(y)
\end{equation*} 
is bilinear (8.6.5). A quadratic function is even (respectively odd) if $q$ satisfies 
\begin{equation*}
    q(-x)= (-1)^{\epsilon}q(x),\quad \text{for all } x\in M\;.
\end{equation*} 
where $\epsilon=0$ (respectively $\epsilon=1$).
\begin{lemma}\label{quadprop}
An even quadratic function \( q \) is a quadratic form and satisfies the
parallelogram law: 
\begin{equation*}
  q(x+y) + q(x-y)  = 2 q(x) + 2 q(y)\;.
\end{equation*} 
If \( q \) is an odd quadratic function then $2q$ is a homomorphism.
\end{lemma}
\begin{proof}
In the even case
 \begin{equation*}
  q(x+y) + q(x-y) = 2 q(x) + q(y) + q(-y) + b(x,y)+b(x,-y) = 2 q(x) + 2 q(y)\;.
\end{equation*} 
On the other hand, for an odd quadratic function \( q \) we have 
\begin{equation*}
  b(x,y) = b(-x,-y) = -b(x,y)
\end{equation*} 
so that $b(x,y)$ is always $2$-torsion and therefore $2q$ is a homomorphism.
\end{proof}
By the proof of Lemma 8.6.10 in loc. cit., $q$ is a quadratic function if and only if we
have for all $x=(x_1,x_2,x_3)\in M^3$, 
\begin{equation}\label{thmcube}
    \sum_{I\subset \{1,2,3\}} q(S_I(x)) = 0\;, 
\end{equation} 
where $S_I(x)= \sum_{i\in I}x_i$. The theorem of the cube for an abelian
variety $A/K$ is the statement that for any variety $X/K$ and any $c\in \Pic(A)$
the map $q_c\colon \Mor(X,A) \to \Pic(X)$ given by $q(\phi)= \phi^\ast c$ is a
quadratic function (see~\cite[Theorem~8.6.11]{BG06}). The usual statement of this
theorem~\cite[Cor 6.4]{Milne86b}
is equivalent to this statement by applying~\eqref{thmcube} to
the universal situation 
\begin{equation}\label{E:universalthmcube}
    X=A\times A\times A\;,\quad x=(\pi_1,\pi_2,\pi_3)
\end{equation}
and reads as follows: For any line bundle $\L$ on $A$ one has an isomorphism
\begin{equation}\label{ThmCube}
  \pi_{123}^\ast \L \otimes 
  \pi_{12}^\ast \L^{-1} \otimes 
  \pi_{23}^\ast \L^{-1} \otimes 
  \pi_{13}^\ast \L^{-1} \otimes 
  \pi_{1}^\ast \L \otimes 
  \pi_{2}^\ast \L \otimes 
  \pi_{3}^\ast \L \isom \O_{A^3}\;.
\end{equation} 
For the converse, see Corollary~6.5 there.
\begin{prop}\label{cubeisoms}
The maps~\eqref{fg-iso} for symmetric and~\eqref{fg-iso1} for antisymmetric line bundles on
$A$ are indeed isomorphisms as claimed.
\end{prop}
\begin{proof}
Combine the theorem of the cube with Lemma~\ref{quadprop} and apply to the elements \[
  x=\pi_1, y=\pi_2 \in \Mor(A\times A,A)
  \,.\]
If \( \L \) is symmetric this completes the proof. 
If \( \L \) is anti-symmetric we may, after possibly enlarging the ground field, find a line bundle
\( \L' \) with \( \L^{\prime \otimes 2}\isom \L \). Then, again by same results, we get an isomorphism as required over the larger ground field, but this descends to \( K \) by
Hilbert~90.
\end{proof}

\subsection{Good log functions on symmetric and antisymmetric line
bundles}\label{S:good}
For the remainder of this section we suppose that $K$ is a $p$-adic field and that $A/K$ is an
  abelian variety. 
We now study the $p$-adic analogues of canonical valuations.
  In the following, all line bundles on abelian varieties will be rigidified at~0. All
  (iso)morphisms of line bundles
  will be viewed as (iso)morphisms of rigidified line bundles, although we will often
not write this explicitly, to simplify notation. Likewise, tensor products of line bundles
  will be tensor products of rigidified line bundles.
 For a symmetric or antisymmetric line bundles, the rigidification fixes the
  choice of isomorphism~\eqref{m-iso},~\eqref{m-iso2},~\eqref{fg-iso},
  \eqref{fg-iso1}, \eqref{m-iso1} and~\eqref{m-iso3}.

\begin{defn}\label{D:normalized}
    We say that a log function $\log_L$ on a rigidified 
    line bundle $(\L,r)$ on $A/K$ is {\em normalized}
    if $\log_L(r)=0$. 
\end{defn}
\begin{defn}\label{D:goodmetric}
    We say that a log function on a symmetric 
    line bundle $\L/A$ is {\em good}
    if~\eqref{m-iso} is an isometry.
\end{defn}

\begin{rk}\label{R:goodsymmnorm}
A good log function on a symmetric line bundle is normalized.
\end{rk}
\begin{defn}\label{D:symmetriclog}
 We say that a log function on a symmetric line bundle $\L/A$ is {\em
  symmetric} if the symmetry isomorphism $[-1]^\ast \L \isom \L$ is an isometry.
\end{defn}

\begin{thm}\label{T:symgood}
  Let $\L$ be a symmetric line bundle on $A$
and let $\alpha \in \Omega^1(A)\otimes \hdr^1(A)$  such that $\cup\alpha = \ch_1(\L)$.
  Then there exists a unique good log function $\log_{\L}$ such that $\Curve(\log_L) =
  \alpha$. The good log function is also symmetric.
  \end{thm}

  Before we prove this theorem, we state the analogous definitions and our
  result in the case of antisymmetric bundles. 

\begin{defn}\label{D:goodmetricanti}
    We say that a log function on an antisymmetric
    line bundle $\L$ is {\em good}
    if~\eqref{m-iso2} is an isometry.
\end{defn}
As in the symmetric case, the notion of goodness depends on the choice of a rigidification of $L$, used to
fix the isomorphism~\eqref{m-iso2}.
\begin{rk}\label{R:goodantinorm}
A good log function on an antisymmetric line bundle is normalized.
\end{rk}

\begin{defn}\label{D:antisymmetriclog}
 We say that a log function on an antisymmetric line bundle $\L/A$ is {\em
  antisymmetric} if $[-1]^\ast \L \isom \L^{-1}$ is an isometry.
\end{defn}

\begin{defn}
  We call a metrized line bundle  {\em flat} if the corresponding curvature form is trivial.
\end{defn}

As the
cohomology class $\ch_1(\L)$ of any antisymmetric line bundle $\L/A$  is trivial, any $\alpha \in \Omega^1(A)\otimes
\hdr^1(A)$ cupping to $0$ is the curvature form for some log function on $\L$.

\begin{thm}\label{T:antigood}
  A log function $\log_L$ on an antisymmetric line bundle $\L$ is good if and only if 
  $(L,\log_L)$ is flat and normalized.
\end{thm}

We are indebted to an anonymous referee who suggested the following
simplified approach for proving these theorems we present here. The
main tool is Proposition~\ref{P:metriccube}, which is a version of the theorem of the cube for rigidified line
bundles with a normalized log function. We also
sketch our original approach in~\S\ref{subsec:oldproofs}.

  We define $\Pic(X,x_0)$ for a variety $X/K$ and a basepoint $x_0\in X(K)$ to be the group of isomorphism classes
of line bundles on $X$ rigidified at  $x_0$. These groups are naturally
functorial for basepoint preserving maps. The rigidified and non-rigidified
groups are isomorphic (recall that all our rigidifications are $K$-rational) and
statements translate easily between the two scenarios since any isomorphism
uniquely scales to an isomorphism preserving the rigidifications. In
particular, picking the basepoint $0$ for an abelian variety $A/K$, there is
an analogue of the theorem of the cube where a rigidified $c\in \Pic(A,0)$ gives a quadratic function
$q_c\colon \Mor((X,x_0),(A,0)) \to \Pic(X,x_0)$ on the abelian group of basepoint
preserving maps.

For a variety $X$ over $K$ with base point $x_0\in X(K)$, the isometry
classes of rigidified line bundles with normalized log functions form a group
$\picm(X,x_0)$ under tensor products.
\begin{prop}\label{P:metriccube}
    The theorem of the cube holds for $\picm$ : For any $c\in \picm(A,0)$
    the map
$q_c(\phi)= \phi^\ast c$ is a
quadratic function
$q_c\colon \Mor((X,x_0),(A,0)) \to \picm(X,x_0)$.
\end{prop}
\begin{proof}
It is sufficient to consider the universal
situation~\eqref{E:universalthmcube}, where the result
follows from~\cite[Lemma~B.2]{Bes05} by restricting the map $m_I$ there to
$0\times A^3$.
\end{proof}
\begin{prop}\label{symisgood}
The maps ~\eqref{fg-iso} for a symmetric line bundle with a symmetric
  log function and~\eqref{fg-iso1} for an antisymmetric line bundle with an
  antisymmetric log function on
$A$ are isometries. In other words, these log functions are good.
\end{prop}
\begin{proof}
  In the symmetric case the proof is identical to that of Proposition~\ref{cubeisoms} using $\picm$ and its associated theorem of the cube, Proposition~\ref{P:metriccube}. In the anti-symmetric case one only needs to note that the isometry is obtained after squaring, so also holds without it.
\end{proof}

To proceed, let us discuss in general what might be called $p$-adic log
functions
in ``dynamical'' situations. Suppose $X/K$ is a proper variety, 
$\L$ is a line bundle on $X$, $f\colon X\to X$ is an
endomorphism, $d$ is an integer and 
\begin{equation}\label{E:beta}
\beta\colon f^\ast \L \to \L^d
\end{equation}
is an 
isomorphism. In such a situation we may want
to find a log function on $\L$ making $\beta$ into an
isometry. Classically (see for instance~\cite{Zha95}
or~\cite[Section~9.5]{BG06}) this is done by an averaging technique which
unfortunately does not work over the $p$-adics. We may, however, use the
theory of curvature forms (see~\S\ref{subs32}) to remedy the situation. If a log function on $\L$ has
curvature form $\alpha$, then the log function on $f^\ast \L$ has curvature
form $f^\ast
\alpha$ while the log function on $\L^d$ has curvature form $d \alpha$. Therefore,
Proposition~\ref{P:excurve} and Remark~\ref{R:logfunctionuniqueness}
immediately give the following:
\begin{prop}
    Let $\log_L$ be a log function on $\L$ with curvature form $\alpha$. Then
    $\beta$ in~\eqref{E:beta} is an isometry up to the integral of a holomorphic 
    form with respect to $\log_L$ if and only if $f^\ast \alpha = d \alpha$.
\end{prop}
Suppose that $\log_L$ is already such that $\beta$ is an isometry
up to the integral of a holomorphic form. If we add to the log function $\int
\omega$, for $\omega\in \Omega^1(X)$, then this adds $f^\ast \int \omega$
to the induced log function on $f^\ast \L$, while it adds $d \int \omega$
to the induced log function on $\L^d$.
Thus, our ability to make $\beta$ an isometry up to a constant
depends on the action of $f^\ast -d$ on $\Omega^1(X)$. We record the two
extreme cases.
\begin{prop}\label{P:finaladjustment}
    Suppose $\L$ has a log function with curvature form $\alpha$ such that $f^\ast
    \alpha = d\alpha$, then
\begin{enumerate}[\upshape (a)]
\item If $f^\ast -d$ is invertible on $\Omega^1(X)$, there is a
  unique (up to constant) log function on $\L$ with curvature form $\alpha$ 
    such that $\beta$ is an isometry up to a constant. 
\item If $f^\ast -d =0$ on $\Omega^1(X)$, then for all
    log functions with curvature form $\alpha$, the isomorphism
    $\beta$ will remain an isometry only up to an integral
    of the same holomorphic form  $\eta$.
\end{enumerate}
\end{prop}

Suppose that $d\ne 1$, and that $\log_L$ is a log function such that
$\beta$ in~\eqref{E:beta} is an isometry up to an additive constant. 
Scaling the log function $\log_L$ by an additive constant $c$ adds
$c$ to the induced log function on $f^*L$, but adds $dc$ to the induced log
function on $L^d$.
Hence, by a
similar argument, it is always possible to modify
$\log_L$ such that $\beta$ is an isometry on the nose.
Alternatively, one may use rigidified line bundles and
normalized log functions to achieve this.

In our first and second applications of the above we let $X$ be an abelian
variety $A/K$ and $f=[-1]\colon A\to A$ the multiplication by $-1$ map. In our
first application $\L$ is a symmetric line bundle and $d=1$. Since $[-1]$
acts as $1$ on curvature forms and as $-1$ on $\Omega^1(A)$, we immediately
get from Proposition~\ref{P:finaladjustment} the following.
\begin{prop}\label{P:newsym}
     Let $\L$ be a symmetric line bundle on $A$
and let $\alpha \in \Omega^1(A)\otimes \hdr^1(A)$  such that $\cup\alpha = \ch_1(\L)$.
  Then there exists a unique symmetric log function $\log_{\L}$ such that $\Curve(\log_L) =
  \alpha$. 
\end{prop}	
In our second application $\L$ is antisymmetric and $d=-1$. Now we see that
the only curvature form $\alpha$ for which $[-1]^\ast \alpha = -\alpha$ is
$\alpha=0$.
Since antisymmetric line bundles have trivial first Chern classes,
$\alpha=0$ is a valid curvature form on $L$.
By Proposition~\ref{P:finaladjustment} either all log
functions
on $\L$ are antisymmetric or none of them are. But starting with any log
function on $\L$ with
zero curvature we obtain a log function on $\L \otimes [-1]^\ast \L^{-1}\isom \L^2$ which is clearly antisymmetric, and taking its square root
gives us an antisymmetric log function on $\L$. We therefore get
\begin{prop}\label{P:newantisym}
All log functions with zero
curvature on an antisymmetric line bundle $\L$ are antisymmetric.
\end{prop}
\begin{proof}[Proof of Theorem~\ref{T:symgood}~and~Theorem~\ref{T:antigood}]
  By Proposition~\ref{P:newsym} we have a symmetric log function on a symmetric
  line bundle and by Proposition~\ref{P:newantisym} we have an
  antisymmetric log function on an antisymmetric line bundle with the appropriate conditions. These are good by Proposition~\ref{symisgood}.
\end{proof}

 \begin{rk}
   Alternatively, one can construct a good log function on an anti-symmetric line bundle $\L$
   by writing $\L$ as $t_a^*\M \otimes \M^{-1}$ for some symmetric ample
  line bundle $\M$ and $a \in A$ and showing that, if we start with any good log function
   on $\M$, the induced log function on $t_a^*\M \otimes \M^{-1}$ is good.
 \end{rk}

\begin{rk}\label{R:loguniq}
 For symmetric line bundles, there is a \textit{unique} good log function for
  \textit{any} curvature form for $\L$, whereas for antisymmetric line bundles, good log
  functions exist \textit{only} if the curvature is $0$. Moreover, in this case,
  \textit{every} normalized log function with curvature $0$ is automatically good, so good log
  functions are far from unique for antisymmetric line bundles.
\end{rk}

\subsubsection{Sketches of alternative proofs of
Theorems~\ref{T:symgood}~and~\ref{T:antigood}}\label{subsec:oldproofs}
We now briefly sketch our original approach to proving
Theorems \ref{T:symgood}~and~\ref{T:antigood}. This avoids the use of
Proposition~\ref{P:metriccube} but
in the antisymmetric case has to rely on the theory of canonical log
functions, which uses the Poincar\'e line bundle.
We first let $\L$ be a symmetric line bundle on
$A$. Call a log function on $\L$ $2$-good if the isomorphism~\eqref{m-iso1}
is an isometry (we similarly call a log function on an antisymmetric line
bundle \( 2 \)-good if~\eqref{m-iso3} is an isometry). By the third application
of Proposition~\ref{P:finaladjustment},
this time with $f$ being the multiplication by $2$ map and $d=4$, we easily
see that, similar to Proposition~\ref{P:newsym} there is, for each
curvature form on $\L$, a unique $2$-good log function on $\L$ with
this curvature. Fix such a log function. Since multiplication by $2$ and by
$-1$ commute, it is easy to see that the pullback by $[-1]$ of the log
function is also $2$-good. By uniqueness of $2$-good
log functions we thus see that the same log function is also symmetric. To show that
this log function is also good we compare the log functions on both sides
of~\eqref{m-iso}. An easy computation shows that they have
the same curvature and they therefore differ by the integral of a
holomorphic form on $A\times A$. This is of the form $\pi_1^\ast \omega_1
+\pi_2^\ast \omega_2$ with $\omega_1,\omega_2\in \Omega^1(A)$. We now pull back by the automorphism of $A\times A$ switching both coordinates.
composing with this the sum map $s$ does not change it while the difference
map $d$ is transformed to $-d$. Since the log function on $\L$ is symmetric,
switching gives on both sides of~\eqref{m-iso} isometric log functions. It
follows that $\omega_1=\omega_2$. Pulling back via the diagonal we see that
the two log functions on both sides of~\eqref{m-iso1} differ by the
integral of $2\omega_1$, and since the log function is $2$-good we find $\omega_1=0$ proving
the goodness of the log function.

For an antisymmetric line bundle  $\L$ on $A$
with a log function with trivial curvature, a similar approach does not
quite work. Indeed, comparing the two sides
of~\eqref{m-iso2} they both have trivial log function and thus the isomorphism is an isometry up to the integral of a holomorphic form on $A\times A$.
The log function on both sides of~\eqref{m-iso2} will change the same if we
modify the log function and therefore, just as for antisymmetry, either all
log functions are good, or all of them are not good. One can use the theory
of the canonical log function of Subsection~\ref{S:antisymm} to show that one, hence, all, flat and
normalized log functions are good. For a brief account, see
Remark~\ref{canrelic}.

\subsection{Canonical log functions on antisymmetric line
bundles}\label{S:antisymm}

Recall from Remark~\ref{R:loguniq} that good log functions on antisymmetric line bundles are
not unique.
  We can metrize an antisymmetric line bundle $\L/A$ by
  restricting a log function on the Poincar\'e bundle $\P$  on $A\times \hat{A}$ to the
  fiber corresponding to $\L$. 
  Since $\P$ is symmetric, we can fix a curvature form for $\P$ and obtain, by
  Theorem~\ref{T:symgood},  a unique associated good log function $\log_\P$.
  We will show that for certain curvature forms on $\P$,
  this yields a good log function on $\L$. This construction 
  simultaneously attaches a unique good log function to every
  antisymmetric line bundle.

For notational convenience we denote $A$ by
$A_1$ and $\hat{A}$ by $A_2$  and we write $\pi_i$ for the projection $A_1\times
A_2\to A_i$. There are direct sum decompositions 
\begin{equation*}
  \Omega^1(A \times \hat{A}) = \pi_1^\ast \Omega^1(A)
  \oplus \pi_2^\ast \Omega^1(\hat{A})\;,
  \hdr^1(A \times \hat{A}) = \pi_1^\ast \hdr^1(A)\oplus \pi_2^\ast \hdr^1(\hat{A})
\end{equation*}
resulting in a direct sum decomposition  
\begin{equation*}
  \Omega^1(A\times \hat{A})\otimes \hdr^1(A\times \hat{A})= \oplus_{i,j=1,2}H_{ij}\;,
  H_{ij} = \pi_i^\ast \Omega^1(A_i) \otimes \pi_j^\ast \hdr^1(A_j)
    \isom \Omega^1(A_i) \otimes  \hdr^1(A_j).
\end{equation*}
For an element 
$ \alpha\in \Omega^1(A\times \hat{A})\otimes \hdr^1(A\times \hat{A}) $ we let
$\alpha_{ij}$ be its component in $H_{ij}$.

\begin{defn}\label{D:admcurv}
 An element $\alpha \in \Omega^1(A\times \hat{A})\otimes \hdr^1(A\times \hat{A})$ is called a
  {\em purely mixed curvature form} for the Poincar\'e bundle $\P$  if $\cup\alpha =
  \ch_1(\P)$ and  $\alpha_{11} = \alpha_{22} = 0$.
\end{defn}

We will show in Section~\ref{subsec:CG} that 
purely mixed curvature forms for the Poincar\'e bundle exist.
In Proposition~\ref{P:curvspace}, we prove that purely mixed curvature forms for $\P$ are in one-to-one correspondence with complementary subspaces for the inclusion $\Omega^1(A) \hookrightarrow \hdr^1(A)$. 
First we give an alternative characterization of the good log function on $\P$ with a
given purely mixed curvature form.
\begin{rk}\label{canrig} 
  From now on, we fix a rigidification $r_\P$ of $\P$ at $(0,0)$. This choice induces 
  trivializations $\P_{\{0_A\}\times \hat{A}}\cong \O_{\hat{A}}$ and  $\P_{A\times
  \{0_{\hat{A}}\}}\cong \O_A$ by requiring that $r_\P$ corresponds to $1(0_A)$
  (respectively $1(0_{\hat{A}})$), 
  where $1$ is the canonical section of $\O_A$ (respectively of $\O_{\hat{A}}$).  
  In particular, for $\hat{a}\in \hat{A}$, this choice
  induces a rigidification $r_{\hat{a}}$ on the antisymmetric line bundle $\P_{A\times \{\hat{a}\}}$ on
  $A$, corresponding to $1(\hat{a})$, and similarly for $a\in A$.
\end{rk}

\begin{prop}\label{P:cangood}
 Let $\alpha$ be a purely mixed curvature form for $\P$.
  \begin{enumerate}[\upshape (a)]
    \item\label{P:canalt} There is a unique log function $\log_\P$ on $\P$ with curvature
      $\alpha$ which restricts to the trivial log function on $A\times \{0\}$ and on
      $\{0\}\times \hat{A}$ with respect to the trivializations from
      Remark~\ref{canrig}.
    \item\label{P:canisgood} The log function $\log_\P$ from part~\eqref{P:canalt} is the
      good log function with curvature $\alpha$ with respect to the
      rigidification $r_{\P}$ from Remark~\ref{canrig}.
  \end{enumerate}
\end{prop}

\begin{proof}
\hfill

\begin{enumerate}[\upshape (a)] 
  \item   Since $\alpha$ is purely mixed, the restriction of $\alpha$ to any horizontal or vertical
fiber is trivial. 
    Thus, by Theorem~\ref{T:antigood}, any log function for $\alpha$ induces flat log
    functions on each antisymmetric line bundle on $A$ and by duality a flat log function on
    each antisymmetric line bundle on $\hat{A}$. In particular, the restriction of such a log
    function to $A\times \{0\}$ and $\{0\} \times \hat{A}$ is trivial up to an integral of a
    holomorphic form.  By Remark~\ref{R:logfunctionuniqueness},
    the curvature determines the log function up to the integral of a
    holomorphic 1-form
    $\omega = \pi_1^\ast \omega_1 +
    \pi_2^\ast \omega_2 \in \Omega^1(A \times \hat{A})$, for some $\omega_1 \in \Omega^1(A), \omega_2 \in
    \Omega_1(\hat{A})$. 
    Changing the log function by the integral of $\omega$ changes the log function restricted
    to $A\times \{0\}$ by $\int \omega_1$ and the restriction to $\{0\} \times \hat{A}$ by
    $\int \omega_2$. From this (a) follows easily.

  \item The restriction of the good log
    function on $\P$ with curvature form $\alpha$ to 
    $A\times \{0\}\sim \O_A$ is flat and symmetric, and has curvature~0. 
    Therefore this restriction is the trivial log function, and the same is
    true for the restriction to $\{0\} \times \hat{A}$. The result follows
    from part~\eqref{P:canalt}. \qedhere
\end{enumerate}

\end{proof}

We now prove a result that can be used to show that certain isomorphisms between
symmetric line bundles are isometries. 
\begin{lemma}\label{L:goodness}
  Let $L/A$ be a symmetric line bundle and let $\log_L$ be a good log function with
  respect to the chosen rigidification $r$. 
      \begin{enumerate}[\upshape (a)]
        \item Let $f\colon A'\to A$ be a {homo}morphism
          of abelian varieties. Then $f^*\log_L$ is good.
    \item\label{L:goodsums} Let $M/A$ be another symmetric line bundle with a good log function $\log_M$.
      Then $\log_L+\log_M$ is good. 
    \item 
  Let $\varphi\colon L_1\to L_2$ be 
an isomorphism between line bundles on abelian variety $A'/K$ such that $L_1$ and
          $L_2$ are obtained from $L$ by pullbacks and tensor products as in (a) and (b).
          For $i=1,2$, let  $\log_i$ be the log function on $L_i$ induced by $\log_L$ and let
  $\alpha_i =  \Curve(\log_i)$. Then
  $\varphi$ is an isometry if and only if $\alpha_1=\alpha_2$.
      \end{enumerate}
\end{lemma}
\begin{proof}
  As $f$ is a homomorphism, pullback by $f$ followed by pullback by $s,d,\pi_1,\pi_2$ is the same as first pulling back by $s,d,\pi_1,\pi_2$ followed by pullback by $f \times f$.
  Using Remark~\ref{R:symmfunctorial}, the first two assertions are immediate.  
  For $i=1,2$, there is a unique good log function on 
  $L_i$ with curvature $\alpha_1$.
  By (a) and (b), $\log_1$ and $\log_2$ are good, so (c) follows.
\end{proof}

\begin{prop}\label{P:anticangood}
Let $\alpha$ be a purely mixed curvature form for $\P$ 
and let $\hat{a}\in \hat{A}$.
  Then the log function
  induced by $\log_\P$ on $\P|_{A\times \{\hat{a}\}}$  is good
  for the rigidification $r_{\hat{a}}$
  from Remark~\ref{canrig}. 
\end{prop}
\begin{proof}
  This follows immediately from Theorem~\ref{T:antigood}.
\end{proof}

If $\alpha$ is a purely mixed curvature form for $\P$ and if $\L/A$ is an antisymmetric line bundle with rigidification $r$, then there is a unique
isomorphism $\psi_{L,r}\colon (L,r)\cong (\P|_{A\times \{\hat{a}\}}, r_{\hat{a}})$, where $\hat{a}$ is the class of
$\L$. 
\begin{defn}\label{D:antican}
Let $\L/A$  be an antisymmetric line bundle with rigidification $r$.
 The {\em canonical log function on $\L$} (with respect to $\alpha$) is
  the log function obtained by pulling back the log function from
  Proposition~\ref{P:anticangood} by $\psi_{L,r}$.
\end{defn}
We immediately deduce: 
\begin{cor}\label{C:anticangood}
  Let $\log_L$ be the  canonical log function on an antisymmetric rigidified line bundle
  $(L,r)$. Then $\log_L$ is good. In particular, it is normalized.
\end{cor}

We now show that the  canonical log function on an antisymmetric line bundle has some
desirable properties. 
First, by dualizing the objects in the proof of Proposition~\ref{P:anticangood}, we get the following result.
\begin{prop}\label{P:tensorcan}
The tensor product of the canonical log functions on two antisymmetric line bundles is
  canonical.
\end{prop}
\begin{proof}
This is equivalent to the statement that for any fixed \( a\in A \) the log function
  induced by $\log_\P$ on $\P|_{\hat{A}\times \{a\}}$  is good, and this
  follows from Theorem~\ref{T:antigood} just like
  Proposition~\ref{P:anticangood}.
\end{proof}
\begin{rk}\label{canrelic}
  The characterization of the good log function on \( \P \) provided by
  Proposition~\ref{P:cangood} may be used for proving
  Propositions~\ref{P:anticangood}~and~\ref{P:tensorcan} without the use of
  flatness, hence, ultimately, without the use of the theorem of the cube
  for metrized line bundles, Proposition~\ref{P:metriccube}. This was our
  original proof, as indicated in \ref{subsec:oldproofs}.
  Let us just sketch a proof that the canonical log function is \( 2
  \)-good in the sense of \ref{subsec:oldproofs}. Such a statement for all
  antisymmetric line bundles is equivalent to the statement that the
  canonical isomorphism \( (\id_{\hat{A}}\times [2])^\ast \P \isom \P
  ^{\otimes 2}  \) is an isometry. For this one simply checks that the log
  function on the square root of the left hand side has curvature \(
  \alpha\) and satisfies the conditions of Proposition~\ref{P:cangood}.
\end{rk}

To understand the behavior of canonical log functions on antisymmetric line bundles under translations, we first prove a preliminary result.
\begin{lemma}\label{L:trivcondition} 
Let $(\L,\log_{\L}), (\M,\log_{\M})$ be metrized line bundles on $A \times \hat{A}$, with an
isomorphism of line bundles $\L \cong \M$. Suppose $\Curve(\log_{\L}) = \Curve(\log_{\M})$, and
that the restrictions of $\log_{\L}$ and $\log_{\M}$ to $A \times \{0\}$ and $\{0\} \times
\hat{A}$ are isometric. 
  Then the given isomorphism is an isometry.
\end{lemma}
\begin{proof}
  Write $\log_\M$, by abuse of notation, for the pulled back log function on $\L$.  
 By Remark~\ref{R:logfunctionuniqueness}, the curvature determines the log
  function up to the integral of a holomorphic 1-form $\omega \in \Omega^1(A \times \hat{A})$, so we have
 \[\log_{\L} = \log_{\M} + \pi_1^\ast \int \omega_1 + \pi_2^\ast \int \omega_2\] 
 for some $\omega_1 \in \Omega^1(A), \omega_2 \in \Omega_1(\hat{A})$. 
 Restricting both sides of the equality above to $A \times \{0\}$ and $\{0\} \times
  \hat{A}$ and using our assumption that the restrictions of the log functions to the
  fibers are isometric, it follows that $\int \omega_1 = \int \omega_2 = 0$, and therefore
  $\log_\L=\log_\M$.
\end{proof}

\begin{prop}\label{P:trans}
  Let $\L/A$ be an antisymmetric line bundle with canonical log function $\log_\L$ and let
  $a\in A$. Then, up to an additive constant, $t_a^\ast\log_\L$ is the canonical log
  function on $t_a^*\L$.
\end{prop}
\begin{proof}
Let $\tau = t_{(a,0)}\colon A\times \hat{A} \to A\times \hat{A}$. Then, for $\hat{a}\in
  \hat{A}$, 
since $\P|_{A\times \{\hat{a}\}}$ is translation-invariant,
  we have $$\tau^*\P|_{A\times \{\hat{a}\}}\isom \P|_{A\times \{\hat{a}\}}\isom  
  (\P\otimes \pi_2^*M_a)|_{A\times \{\hat{a}\}},$$ where 
  $M_a$ is the line bundle on $\hat{A}$ whose class is $a$ under 
  duality.
  Moreover, the restriction of $\tau^*\P$ to $\{0\}\times \hat{A}$ is isomorphic to
  $\pi_2^*M_a$, so
  $\tau^*\P |_{\{0\}\times \hat{A}}\isom (\P \otimes\pi_2^*M_a)|_{\{0\}\times \hat{A}},$
  which implies 
  \begin{equation}\label{transisomo}
\tau^*\P\isom \P\otimes \pi_2^*M_a\,.
  \end{equation}
  Let $\log_\P$ be the good log function on $\P$ (with respect to a purely mixed curvature
  form $\alpha$ and the rigidification $r_\P$, as usual) and let $\log_{\M_a}$ be the canonical log function on $M_a$
  induced by $\log_\P$ via restriction. We want to use
  Lemma~\ref{L:trivcondition} to show that~\eqref{transisomo} is an isometry for these log
  functions.
  First note that $\Curve(\log_\P)$ is translation-invariant 
  and that $\pi_2^*\log_{M_a}$ has trivial curvature. Furthermore,~\eqref{transisomo} is tautologically an isometry when restricted to
  $\{0\}\times \hat{A}$. It is also an isometry when restricted to
  $A\times \{0\}$, by Remark~\ref{P:cangood} and since the trivial bundle equipped with the trivial log
  function is translation-invariant.
  Hence~\eqref{transisomo} is indeed an isometry as claimed.

  Finally, let $\hat{a} = [L]\in \Pic^0(A)$. Restricting the isometry~\eqref{transisomo} to the fiber
  $A\times \{\hat{a}\}$, the result follows, since $\pi_2^*\log_{M_a}|_{A\times
  \{\hat{a}\}}$ is constant. 
\end{proof}

\subsection{Canonical log functions on arbitrary line bundles}
Fix a purely mixed curvature form $\alpha$ for $\P$.
Let $\log_{\P}$ be the good log function with curvature
$\alpha$ from Proposition~\ref{P:cangood}. Now we systematically fix canonical (with respect to $\alpha$) log functions on all line bundles on $A$ using $\log_{\P}$.
We have already defined canonical log functions for antisymmetric line bundles in
Definition~\ref{D:antican}.

For a line bundle $L/A$, we set $\L^+\colonequals \L \otimes [-1]^* \L$  and $\L^-
\colonequals \L \otimes
([-1]^* \L)^{-1}$, so that $\L^+$ is symmetric, $\L^-$ is antisymmetric and we have 
\begin{equation}\label{E:oddeven}
\L^{\otimes 2} \cong \L^+\otimes \L^-. 
\end{equation}
Recall the homomorphism $\phi_L$ from Proposition~\ref{phiprop}.

\begin{defn}\label{D:arblog}
  Let $\L$ be a line bundle on $A$ with rigidification $r$.
  The \emph{canonical log function} $\log_{\L}$ for $\L$ (with respect to $\alpha$ and $r$) is
  defined as follows:
  \begin{enumerate}[\upshape (a)]
    \item\label{arbloga}  When $\L$ is antisymmetric, then $\log_\L$ is defined in Definition~\ref{D:antican}.
    \item\label{arblogb}  When $\L$ is symmetric, then $\log_\L$ is the good log
      function with curvature $\frac12(\id\times \phi_L)^\ast\alpha$ (with respect to $r$).
    \item\label{arblogc} In general, we define 
$\log_{\L} = \frac{1}{2}(\log_{\L^+}+\log_{\L^-})$ using the 
  canonical decomposition~\eqref{E:oddeven} and
      Definition~\ref{D:relatedlogfunctions}~\eqref{D:dividinglogfunctions},
       where the log functions on $L^+$ and $L^-$ are the canonical log functions with respect to the
      rigidifications $r^+$ (respectively $r^-$) on $L^+$ (respectively $L^-$) induced by $r$.
  \end{enumerate}
 \end{defn}

\begin{rk}\label{R:canlogformula}
  The definition in part~\eqref{arblogc} above is compatible with parts~\eqref{arbloga}
  and~\eqref{arblogb} by
  Proposition~\ref{P:tensorcan} and Lemma~\ref{L:goodness}~(\ref{L:goodsums}), since
  $\L^+ \cong \L^{\otimes 2}$ and $\L^- \cong \O_A$ for $\L$ symmetric, and analogously
  $\L^+ \cong \O_A$ and $\L^- = \L^{\otimes 2}$ for $\L$ antisymmetric. 
 
\end{rk}
Letting $\hat{a} = [L^-]\in \hat{A}$,
we get the following formula for the canonical log function
  for $(\L,r)$ with respect to $\alpha$ as a sum of (scaled) pullbacks:
  \begin{equation}\label{cangeneral}
    \log_{\L} = \frac{1}{2}\psi_{L^+, r^+}^*(\id\times {\phi}_{\L})^* \log_{\P}+
    \frac{1}{2} \psi_{L^-,r^-}^*\log_{\P}|_{A \times \{
    \hat{a} \}}\,,
  \end{equation}
  where $\psi_{L^+,r^+}\colon L^+\cong L^+$ is the unique isomorphism sending $r^+$ to
  $(\id\times {\phi}_{\L})^*r_\P$.

  \begin{lemma}\label{L:can_log_add}
    Let $L_1, L_2$ be two line bundles on $A$ with respective rigidifications
    $r_1$ and $r_2$. For $i\in \{1,2\}$, let $\log_{L_i}$ be the canonical
    log functions on $L_i$  with respect to $\alpha$ and  $r_i$. Then
    $\log_{L_1}+\log_{L_2}$ is the canonical log function with respect to
    $\alpha$ and
    $r_1\otimes r_2$.
  \end{lemma}
  \begin{proof} This is a consequence of 
Lemma~\ref{L:goodness} and Proposition~\ref{P:tensorcan}. 
  \end{proof}

\begin{rk}\label{R:change_rig2}
  By construction and by Remarks~\ref{R:goodsymmnorm} and~\ref{R:goodantinorm}, the
  canonical log function is normalized.
  In analogy with Remark~\ref{R:change_rig}, 
changing the rigidification $r$ to a rigidification $r' = \lambda r$ changes
  the canonical log function by $\log \lambda$.
\end{rk}

\subsection{Relation with Colmez's theory}\label{subsec:relcol}
For a divisor $D$ on an abelian variety $A/K$ 
Colmez~\cite[Proposition~II.1.19]{Col98} defines Green functions $G_D\colon A(K) \setminus \supp(D)\to
\Q_p$~\footnote{Actually, his Green functions are valued in a polynomial algebra with
variable $\log(p)$; see~\S\ref{S:Colmez}. In this section, we assume that we have
chosen a value for $\log(p)$.}
. He then extends this construction, by pullback, to general nice
varieties (see~\cite[\S II.1.3]{Col98}). The goal of this section
is to relate our canonical log functions to Colmez's Green functions in the
case where $D$ is the theta divisor $\Theta$ on a Jacobian $J$ (see
Theorem~\ref{T:loggreen}). See

\begin{prop}\cite[Proposition~B.3]{Bes05}\label{P:BesLogGreen}
If $\log_L$ is a log function on a line bundle $L$ on a nice variety and
$s$ is a section of $L$, then
$\log_L\circ s$ is a Green function with respect to $\div(s)$. 
\end{prop}

In contrast to our approach, there is no notion of good or canonical Green functions
on general abelian varieties and no notion of curvature forms in Colmez's
setup. For the special case of a Jacobian $J$ of a nice curve $C$ and a
theta divisor $\Theta$ on $J$, Colmez is able to associate to each complementary
subspace to $\Omega^1(C) \subset \hdr^1(C)$ that is isotropic with respect
to the cup product pairing a Green function $G_\Theta$,
with respect to $\Theta$, unique up to an additive constant, as follows. 

By~\cite[Proposition~8.10.23]{BG06}, there is a unique
point $w\in J$ such that the divisor $\Theta$ is equal to
$t_w^*[-1]^*\Theta$. Colmez defines a Green function $G_{\Theta}$ to be {\em symmetric} if it is
symmetric with respect to the map $x\mapsto w-x$ (see~\cite[\S
II.2.2]{Col98}). He then shows that symmetric Green functions
up to an additive constant are in bijection with isotropic complementary
subspaces (see~\cite[Proposition~II.2.4]{Col98}).
  In Proposition~\ref{P:curvspace} we will show that purely mixed curvature forms
  are in bijection with complementary subspaces
  $\Omega^1(J) \subset \hdr^1(J)$. 
\begin{thm}\label{T:loggreen}
  Let $J/K$ be a Jacobian of a nice curve $C$ with theta divisor
  $\Theta=\div(s)$ corresponding to an Abel--Jacobi embedding $\iota\colon
  C\to J$ such that 
$\Theta = t_w^*[-1]^*\Theta$, where $s$ is
  a section of a line bundle $\L$ on $J$ and $w\in J$.  Let $\log_L$ be the  canonical log function on $L$
  corresponding to a purely mixed curvature form $\alpha$ on $\P$ and let
  $G_\Theta$ be a symmetric Green function with respect to the
  complementary subspace corresponding to $\alpha$. Then
  $\log_L\circ s-G_\Theta$ is an additive constant. 
\end{thm}
To prove Theorem~\ref{T:loggreen}, we will need a few definitions and lemmas.
\begin{defn}\label{D:symmetricdiv}
  Let $f \colon X \rightarrow X$ be an isomorphism of a nice
  variety $X/K$. 
  We say that a Vologodsky function $g$ defined on a nonempty Zariski open
  subset of $X$ is \emph{$f$-symmetric} if $g \circ f = g$.
  
   We say that a divisor $D$ on $X$ is \emph{$f$-symmetric} if
   $f^*D = D$. Let $\L/X$ be a line bundle and $s$ a section of $\L$. We
   call the pair $(\L,s)$  \emph{$f$-symmetric} if there is an
   isomorphism $f^* \L \cong_{\bar{K}} \L$ 
   such that $f^*s = s$ under the induced isomorphism on sections.
   Let $\log_{\L}$ be a log function on a line bundle $\L$ satisfying
   $f^*\L\cong_{\bar{K}} \L$ with curvature
   form $\alpha$ satisfying $f^* \alpha = \alpha$. Then we call $\log_{\L}$
   \emph{$f$-symmetric} if the isomorphism $f^* \L \cong_{\bar{K}} \L$ is an isometry.
\end{defn}

\begin{rk}\label{R:twosymm}
 Note that symmetric log functions on symmetric line bundles as defined in
  Definition~\ref{D:symmetriclog} are simply log functions that are
  $[-1]$-symmetric.
\end{rk}

\begin{lemma}\label{L:symmetricdiv}
 Let $f \colon X \rightarrow X$ be an isomorphism of a nice variety $X$. 
  The map $(\L, s) \mapsto \div(s)$ defines a bijection between
  $f$-symmetric pairs $(\L,s)$, where $s$ is a section of a line bundle
  $\L/X$, and $f$-symmetric divisors on $X$.   
\end{lemma}
\begin{proof}
  It is easy to see that if $(\L,s)$ is $f$-symmetric, then $\div(s)$ is
  $f$-symmetric. For the converse, if $D$ is an $f$-symmetric divisor, then the equality $f^*D = D$
  induces an isomorphism $f^*\O(D) \cong \O(D)$ and since $X$ is proper, an
  equality $f^*s = \lambda s$ for some $\lambda \in K^*$. We can then
  rescale the isomorphism $f^*\O(D) \cong \O(D)$ and the section $s$ by $\sqrt{\lambda}$ to
  obtain an $f$-symmetric pair.   
\end{proof}

\begin{lemma}\label{L:symmlog}
 Let $s$ be a section of a line bundle $\L$ on a nice variety $X/K$ such that
  $(\L,s)$ is $f$-symmetric, where $f \colon X
  \rightarrow X$ is an isomorphism. Let $\log_{\L}$ be a log function for
  $\L$ with respect to a curvature form $\alpha$ satisfying $f^* \alpha =
  \alpha$. Then the isomorphism $f^* \L \cong_{\bar{K}} \L$ is an isometry if and
  only if the function $g \colonequals \log_{\L} \circ s$ is 
  $f$-symmetric.
\end{lemma}
\begin{proof}
 If the isomorphism $f^* \L \cong_{\bar{K}} \L$ is an isometry for the log function
  $\log_{\L}$, then since $(\L,s)$ is $f$-symmetric,
  it follows that $g$ is $f$-symmetric. For the converse, by
  Remark~\ref{R:logfunctionuniqueness}, the difference
  $f^*\log_{\L}-\log_{\L}$ is a function that factors through $X$. Hence $f^*\log_{\L}-\log_{\L} = g \circ f-g$ and the lemma follows.
\end{proof}

\begin{lemma}\label{L:changesymm}
  Let $A/K$ be an abelian variety, and let $u,w \in
  A(\overline{K})$ such that $2u+w=0$. Let $\L,\L'$ be two line bundles on
  $J$ such that $\L \cong t_u^* \L'$. Let $\log_{\L'}$ be a log function
  for $\L'$ and let $\log_{\L} = t_u^* \log_{\L'}$. Let $f \colon J
  \rightarrow J$ be the map $x \mapsto w-x$. Then $\log_\L$ is
  $f$-symmetric  if and only if $\log_{\L'}$ is $[-1]$-symmetric.
\end{lemma}
\begin{proof}
 Let $g \colonequals t_u \circ f \circ (t_u)^{-1}$. Then
  \[ g(x) = t_u \circ f \circ (t_u)^{-1}(x) =  t_u \circ f \circ (t_{-u})(x) = t_u \circ f (x-u) = t_u(w-(x-u)) = w-x+u+u = w+2u-x. \]
  Since $w+2u = 0$, it follows that $g(x) = -x$. We have $\L \cong_{\bar{K}} t_u^*
  \L'$, and this implies that $f^*\L \cong_{\bar{K}} \L$ if and only if $f^* (t_u)^*
  \L' \cong_{\bar{K}} t_u^* \L'$, or equivalently that $(t_{-u})^* f^* (t_u)^* \L'
  \cong_{\bar{K}} \L'$  since $t_{u}^{-1} = t_{-u}$. Since $g(x) = -x$, it follows
  that $(t_{-u})^* f^* (t_u)^* \L' \cong_{\bar{K}} \L'$ if and only if $[-1]^*\L'
  \cong_{\bar{K}} \L'$, because
 \[ [-1]^*\L' = g^*\L' = (t_{-u})^* f^* (t_u)^* \L'.\] 
 Pulling back $\log_{\L}$ along all these induced isomorphisms gives the result that we need.
\end{proof}

Our proof of Theorem~\ref{T:loggreen}, we need one more preliminary result.
This will allow us to identify the two complementary subspaces associated to
the Green function $\log_L\circ s$ in our work and in Colmez's.

\begin{prop}\label{P:derivrest}
  Let $X/K$ be a nice variety. Let $\L$ be
  a line bundle on $X$ and let $s$ be a rational section of $\L$. Let
  $U\subset X$ be an affine open on which $s$ is invertible. Fix 
  a basis $\omega_1,\ldots,\omega_g$ for $\Omega^1(X)$ and 
  let $\log_{\L}$ be a log function on $\L$ with
  curvature $\sum x_i \otimes \omega_i$, with $x_i\in
  \hdr^1(X)$. Let $\partial_i$ be the derivation dual to $\omega_i$. Then
  $d \partial_i \log_{\L}(s)$ represents the restriction of $x_i$ to $U$ in
  $\hdr^1(U)$.
\end{prop}
\begin{proof}
The proof is almost the same as the proof of Theorem~7.3 in~\cite{Bes05}.
Let $G=\log_{\L}(s)$. We have on $U$
\begin{equation*}
  d \partial_i G = d\left( (d G)|_{\partial_i} \right)\;.
\end{equation*}
The function $(d G)|_{\partial_i}$ is in $\O_{V,1}(U)$ and
by~\cite[Lemma~2.8]{Bes05} we have
$ \overline{\partial} ((dG)|_{\partial_i})
  =(\overline{\partial} dG)|_{\partial_i},$
where the retraction acts on the second component in $\Omega^1(U) \otimes
\hdr^1(U)$.
But $\overline{\partial} dG$ is just the curvature restricted to $U$, so by
assumption we get
\begin{equation*}
  \overline{\partial} ((dG)|_{\partial_i}) = x_i|_U \otimes 1\;.
\end{equation*}
Since $U$ is affine it follows from Proposition~2.7 in~\cite{Bes05} that
$(dG)|_{\partial_i}= \int \eta_i + f_i$ for some $\eta_i \in \Omega^1(U)$
representing $x_i|_{U}$ and some $f_i \in \O(U)$. Applying $d$ again we get
the result.
\end{proof}

\begin{proof}[Proof of Theorem~\ref{T:loggreen}]
  Denote the map $x\mapsto w-x$ on $J$ by $f_w$.
By Proposition~\ref{P:BesLogGreen}, Lemma~\ref{L:symmetricdiv},
  Lemma~\ref{L:symmlog} and Lemma~\ref{L:changesymm}, a Green function
  for the theta divisor $\Theta$ is $f_w$-symmetric if and only if the induced metric on $\L$ is such that $u^\ast
  \L$ is symmetric for one, hence any, $u$ such that $2u=w$. By
  Theorem~\ref{T:symgood}, any choice of a
curvature form for $\L$ determines a good, hence symmetric, log function on
$u^\ast \L$.  
  
  To complete the proof, it suffices to show that
the complementary subspace $W$ associated by Colmez to $G\colonequals \log_\L\circ s$ is the same as the
  space $W_\alpha$ corresponding to $\alpha$.
Fix a  choice
  $\omega_1,\ldots,\omega_g$ of a basis of $\Omega^1(J)$; by abuse of
  notation, we also denote their pullbacks by $\iota$ by
  $\omega_1,\ldots,\omega_g$.
  By~\cite[Proposition~II.2.4]{Col98}, $W$ is spanned by the classes of
  $d(\partial_1G),\ldots,d(\partial_gG)$, where
  $\partial_i$ is the derivation dual to $\omega_i$.
  The curvature form of $\log_\L$ 
  is given by $\sum_i^g\bar{\omega}_i\otimes\omega_i$, where
  $\bar{\omega}_1,\ldots,\bar{\omega}_g$
  is a basis of $W_\alpha$ that is dual to $\omega_1,\ldots,\omega_g$ with
  respect to the cup product pairing (see the proof of~\cite[Proposition~6.1]{Bes05}, which essentially 
  follows an argument of Faltings).
  Hence the final part of Theorem~\ref{T:loggreen} follows from
  Proposition~\ref{P:derivrest}.
\end{proof}

\begin{rk}\label{R:}
Let $J$ be the Jacobian  of a genus~2 curve  with a
  rational Weierstrass point at infinity and let $\Theta$ be the
  corresponding theta divisor on $J$. 
  In~\cite{Bia23}, Bianchi explicitly constructs symmetric Green
  functions for $\Theta$ as logs of $p$-adic
  sigma functions due to Blakestad~\cite{Bla18} (and generalizations of the
  latter).
  Combined with Bianchi's work, Theorem~\ref{T:loggreen} shows that we can also
  obtain the log of these
  $p$-adic sigma functions from a canonical log function.
\end{rk}

\section{$p$-adic adelic valuations and global $p$-adic heights}\label{sec:global}

Shou-Wu Zhang introduced a theory of real-valued adelic metrics on line bundles on
varieties over
number fields in~\cite{Zha95}.
Such an adelic metric is a family of continuous real-valued metrics, one for every place of the number
field. 
Using the previous three sections, we now develop a $p$-adic version of this theory, a
theory of
adelic line bundles on nice varieties over number fields with values in $\Q_p$.
Via a choice of an id\`{e}le class character, we can use this to develop a fairly general
theory of $p$-adic heights.
We then specialize to the case of abelian varieties and canonical $p$-adic heights,
similar to Zhang's construction of N\'eron--Tate heights using adelic metrics (see
also~\cite{C-L11} and \cite[\S9.5]{BG06} for expositions).

Let $K$ be a number field.  For a place $v$ of $K$, we denote by
  $K_v$ the completion of $K$ at $v$ with ring of integers
  $\O_v$ and 
 uniformizer  $\pi_v$ 
  We also fix an algebraic closure $\bar{K}_v$ of $K_v$.

Let $p$ be a prime number and let
 \begin{equation*}
   \chi = \sum_v \chi_v \colon  \mathbb{A}_K^\times/K^\times \to \Q_p
 \end{equation*}
be a continuous id\`{e}le class character. 
This means that $\chi$ is a continuous homomorphism such that
\begin{itemize}
  \item we have $\chi_\q(\O_{\q}^\times)=0$ for $\q\nmid p$;
\item for every $\p\mid p$, there is a $\Q_p$-linear trace map $t_\p$ such that we can decompose
  \begin{equation}\label{tdefnd}
    \xymatrix{
      {\O_{\p}^\times}  \ar[rr]^{\chi_{\p}} \ar[dr]^{\log_\p} & &   \Q_p.\\
      & K_\p\ar[ur]^{t_\p}
  }
  \end{equation}
\end{itemize}
One can think of $\chi$ as a ``global log''.
We shall assume that $\chi$ is ramified , meaning that $
\chi_{\p}(\O_\p^\times)\ne 0$, at all $\p\mid p$, so that we can
extend~\eqref{tdefnd} to $K_\p^{\times}$, leading to a factorization
\begin{equation}\label{E:chilog}
\chi_\p = t_\p\circ \log_\p\,,
\end{equation}
valid on $K_\p^\times$. Then $\log_\p$ is a branch of the logarithm 
$K_\p^\times\to K_\p$ as in Sections~\ref{sec:background} and~\ref{sec:log}.

Suppose that $X$ is a nice variety over $K$. If $v$ is a finite place of $X$,
then we write $X_v$ for the base change
of $X$ to $K_v$ and $\bar{X}_v$ for the base change to $\bar{K}_v$. Similarly, for a line
bundle $\L$ of $X$, we write $\L_v$ and $\bar{\L}_v$.

\subsection{$p$-adic adelic metrics and heights}\label{subsec:adelic}
The following definition is a $p$-adic analogue of Zhang's adelic metric
introduced in~\cite{Zha95}.
\begin{defn}\label{D:adelic}
  Let $\L$ be a line bundle on $X/K$.
  A {\em $p$-adic adelic metric} on $\L$ consists of the following data:
    \begin{itemize}
      \item for every place $\p\mid p$, a log function
        $\log_{\L,\p}$ on $\L_\p$
        (see Definition~\ref{D:logfunction}),
      \item for every finite place $\q\nmid p$, a $\Q$-valuation
        $v_{\L,\q}$ on $\L_\q$ 
        (see Definition~\ref{D:val}),
   \end{itemize}
        satisfying the following compatibility condition:
        There exists an integral model $\mathcal{X}/\O_K$ of $X$ and an extension
        $\mathcal{\L}$ of $\L$ to $\mathcal{X}$ such that
        for all but finitely many $\q\nmid p$,  the valuation $v_{\L,\q}$ is equal to the
model        valuation $v_{\mathcal{\L}\otimes \O_\q}$ (see
  Example~\ref{E:alg_val}). 

    We call $\overline{L} \colonequals (\L, (\log_{\L,\p})_{\{\p \mid p\}}, (v_{\L,\q})_{\{\q
    \nmid p\}})$ a {\em $p$-adic
    adelically metrized line bundle} on $X$.
\end{defn}
There are obvious notions of tensor products and pullbacks of adelic metrics and adelically
metrized line bundles via the corresponding notions in Sections~\ref{sec:vals}
and~\ref{sec:background}.
\begin{defn}\label{D:ht}
  Let $\overline{L} = (\L, (\log_{\L,\p})_{\{\p \mid p\}}, (v_{\L,\q})_{\{\q
    \nmid p\}})$ be an adelically metrized line bundle on $X/K$ such that for every place
    $\p\mid p$, the branch of the logarithm in Definition~\ref{D:logfunction} is the branch $\log_\p$ induced by $\chi_{\p}$ as in~\eqref{E:chilog}.    
  The {\em $p$-adic height} associated with $\overline{\L}$ and
    $\chi$ is the function $h_{\overline{\L}}\colonequals h_{\overline{L}, \chi}$
    mapping a point $x\in X(K)$ to the value 
    \begin{equation*}
      h_{\overline{\L}}(x) = \sum_{\p\mid p}t_\p(\log_{\L,\p}(u)) + \sum_{\q\nmid p}
      v_{\L,\q}(u)\chi_\q(\pi_\q)\in \Q_p\,,
    \end{equation*}
    where $u\in \L_x(K)\setminus\{0\}$.
\end{defn}
Note that $\log_{\L,\p}(u)\in K_{\p}$ by~\eqref{KtoK} of Remark~\ref{Volthmcom}, so that the formula makes
sense.
The height is independent of the choice of $u$ by the fact that $\chi$ is an id\`{e}le class character. 

\begin{rk}\label{R:gensubvar}
  It would be interesting to generalize our construction of the $p$-adic
  height of a point to the $p$-adic height of a subvariety of $X$.
\end{rk}

\subsection{Canonical $p$-adic heights on abelian varieties}
Let $A/K$ be an abelian variety.
As in previous sections, we equip all line bundles $\L/A$ with a rigidification $r$ at 0,
inducing compatible rigidifications at 0 on all $\L_v$.
\begin{defn}\label{D:goodadelicmetric}
  Let $\L/A$ be a line bundle.
  An adelic metric on $\L$ is called {\em good}  if  the log function at $\p$
    is good for all $\p\mid p$ and if
the valuations at all $\q\nmid p$ are canonical. 
\end{defn}

\begin{prop}\label{P:goodheight}
  Let $\overline{\L}$ be a line bundle on $A$ endowed with a good adelic metric. Then the
  $p$-adic height $h_{\overline{\L}}$ is a quadratic function on $A(K)$. It is a quadratic (respectively linear) form if 
    $\L$ is symmetric (respectively antisymmetric).
\end{prop}
\begin{proof} 
By Definition~\ref{D:arblog}\eqref{arblogc} and
  Proposition~\ref{P:good}\eqref{P:goodadd}, it suffices to show that $h_{\overline{\L}}$ is a quadratic
  (respectively linear) form for $\L$ symmetric (respectively
  antisymmetric).
  But this is immediate from the definition of good
  log functions (see~\S\ref{S:good}) and canonical valuations (see the 
  paragraph above Proposition~\ref{P:good}) for $\L$ symmetric or
  antisymmetric. 
\end{proof}
Recall that for all $\q\nmid p$, there is a unique good $\Q$-valuation on 
  the line bundle $\L_\q$ with respect to a rigidification $r$. To define canonical adelic
  metrics, we fix a
  purely mixed curvature form $\alpha_p$ of the 
Poincar\'e bundle $\P_\p$ for all $\p\mid p$.

\begin{defn}\label{D:canadelicmetric}
  Let $\L/A$ be a line bundle, rigidified at~0. For each $\q\nmid p$, let $v_{\L,\q}$ be the canonical valuation
  and for each $\p\mid p$ let $\log_{\L,\p}$ be the canonical log function. Then we call
  $((\log_{\L,\p})_{\{\p \mid p\}}, (v_{\L,\q})_{\{\q
    \nmid p\}}))$ the {\em canonical adelic metric} on $\L$.  
\end{defn}

\begin{lemma}\label{L:canadmmet}
  The canonical adelic metric is good.
\end{lemma}
\begin{proof}
The proof follows from Definition~\ref{D:arblog} and Corollary~\ref{C:anticangood}.
\end{proof}

Let $\overline{L} \colonequals (\L, (\log_{\L,\p})_{\{\p \mid p\}}, (v_{\L,\q})_{\{\q
    \nmid p\}}))$  be a canonical adelic metrized line bundle.
We write 
\begin{equation}\label{E:hhatbundle}
\hat{h}_{\L}\colonequals h_{\overline{\L}} = h_{\overline{L}, \chi}\,,
\end{equation}
and we call $\hat{h}_L$ the \textit{canonical $p$-adic height on $A(K)$ with respect
to $L$}.
\begin{cor}\label{L:canht}
  The height 
\begin{equation}\label{E:canht}
  \hat{h} \colonequals \hat{h}_{\P}\colon A(K)\times \hat{A}(K)\to \Q_p
\end{equation}
with respect to the Poincar\'e bundle
defines a bilinear pairing.
\end{cor}
\begin{proof}
  This follows at once from Proposition~\ref{P:goodheight}. Alternatively, 
  note that the adelic metric on $\P$ restricts over $\{a\}\times \hat{A}$ and over $A\times \{a'\}$
  to good adelic metrics on antisymmetric line bundles by Proposition~\ref{P:cangood}. The
  induced heights are thus linear by Proposition~\ref{P:goodheight}, which implies
  bilinearity.
\end{proof}
We call $\hat{h}$ the \textit{canonical $p$-adic height
pairing} 
on $A(K)\times \hat{A}(K)$. It depends both on $\ba$ and on $\chi$, but not on the
  rigidification of $\P$ by Remarks~\ref{R:change_rig} and~\ref{R:change_rig2}.

\begin{cor}\label{C:can_ht}
  Let $\L$ be a line bundle on $A$ and let $\hat{a}=[L^-]\in \hat{A}$ be the point
  corresponding to $L^- = L\otimes ([-1]^\ast L)^{-1}$.
  \begin{enumerate}[\upshape (a)]
    \item\label{C:canhtpart1} For all $a\in A(K)$, we have
  \begin{equation*}
    \hat{h}_{\L}(a) = \frac12\hat{h}(a,\phi_{\L}(a)+\hat{a})\,.
    \end{equation*}
    More precisely, the local contributions at all places to both sides induced by any compatible choices of rigidification of $\P$ and $\L$ are equal.
  \item The bilinear form associated with $\hat{h}_{\L}$ is
    \begin{equation*}
      (a,b)\mapsto \frac{1}{4}\left(\hat{h}(a,\phi_{\L}(b))+\hat{h}(b,\phi_{\L}(a))\right).
    \end{equation*}
  \end{enumerate}
  \end{cor}
\begin{proof}
  For $v\nmid p$, the equality of local contributions follows immediately from Proposition~\ref{P:good}. For $\p\mid p$, it is implied by the construction of the canonical log function and by~\eqref{cangeneral}.
  Part~(b) is obvious.
\end{proof}
The following property is a consequence of
Definition~\eqref{E:hhatbundle}, of Proposition~\ref{P:good}\eqref{P:goodadd},
and of Lemma~\ref{L:can_log_add}:
\begin{cor}\label{C:ht_bund_lin}
  The map $L\mapsto \hat{h}_L$ is linear.
\end{cor}

\section{Comparison with Mazur--Tate and Coleman--Gross heights}\label{sec:comparison}
In this section, we first show that the canonical heights constructed in
Section~\ref{sec:global} induce height pairings in the sense of Mazur and Tate~\cite{MT83}.
We then restrict to the Jacobian of a nice curve. In this case, the canonical
height pairing with respect to a theta divisor is the same as the height pairing due to
Coleman and Gross~\cite{CG89}, with appropriate choices.
By~\cite{Bes04, Bes17}, this pairing is the same as the geometric pairing 
constructed by Nekov\'a\v{r}~\cite{Nek93}, if the curve has semistable
reduction at all places above $p$. 
In future work we will prove a more general result dealing with arbitrary abelian
varieties and the Zarhin height pairing~\cite{Zar90} (which is equivalent
to Nekov\'a\v{r}'s pairing, at least for good reduction).

In this section, $K$ denotes a number field and $A/K$ an abelian variety.
We choose a continuous id\`{e}le class character $\chi\colon
A^\times_K/K^\times \to \Q_p$.
For $\p\mid p$, let $\log_\p$ be the branch of the logarithm induced by
$\chi_\p$ as in~\eqref{tdefnd}.

\subsection{Mazur--Tate}\label{subsec:MT}

In~\cite{MT83}, Mazur and Tate construct global height pairings on $A$
using biextensions of $A$ and $\hat{A}$ by $\G_m$.
These global pairings are sums of local pairings defined in~\cite[Section~2]{MT83}.
For all non-archimedean places $v$ of $K$ we now define local pairings 
$\langle\cdot, \cdot\rangle_v$ via canonical valuations. We then show
that they satisfy the conditions of~\cite[\S2.2]{MT83}. In other words, the pairing at
a place $v$ is induced from a $\chi_v$-splitting in the sense of~\cite[Section~1]{MT83}.
In this section, we write $A_v$ for $A\otimes K_v$, and similarly for line
bundles and divisors.
\subsubsection{Local height pairings away from $p$}
First let $\q$ be a non-archimedean place of $K$ such that $\q\nmid p$. We
now 
define a pairing
between zero cycles
$\frka = \sum_x n_xx\in Z^0_0(A_\q)$ of degree~0 and divisors
$D\in\Div(A_\q)$  with disjoint support as follows:
Let $s$ be an (algebraic) meromorphic section of a line bundle $\L$ on $A_\q$ such that $D=\div(s)$ and let $r$
be a rigidification of $\L$. Let $v_{\L,\q}$ be the canonical valuation on $\L$ (with
respect to $r$). We set
\begin{equation}\label{pairingaway}
    \langle \frka, D\rangle_\q \colonequals \sum_x n_xv_{\L,\q}(s(x))\chi_{\q}(\pi_\q),
\end{equation}
where $\pi_\q$ is a uniformizer at $\q$.
Since $\deg(\frka)=0$, this is independent of the choice of $r$.
\begin{lemma}\label{L:localMTq}
  The pairing~\eqref{pairingaway} satisfies the following properties:
  \begin{enumerate}[\upshape (a)]
    \item $\langle \cdot,\cdot\rangle_\q$ is biadditive,
    \item $\langle \cdot,\cdot\rangle_\q$ is translation-invariant,
    \item if $D=\div(f)$, then $\langle \frka, D\rangle_\q =
      \chi_\q(f(\frka)) $.
  \end{enumerate}
\end{lemma}
\begin{proof}
  The pairing~\eqref{pairingaway} is the same, up to a constant, 
as the real-valued pairing defined by
\begin{equation}\label{pairingawayR}
  \langle \frka, D\rangle_{\q, \R} \colonequals \sum_x
  n_xv_{\L,\q}(s(x))\log_\R\mathrm{Nm}(\q)\,.
\end{equation}
The pairing~\eqref{pairingawayR}
  satisfies (a)--(c) by  
  the proof of~\cite[Theorem~9.5.11]{BG06}. 
\end{proof}
\begin{rk}\label{localneron}
  The pairing~\eqref{pairingawayR} is the classical local N\'eron symbol at $\q$, see~\cite[Theorem~9.5.11]{BG06}. 
  It is characterized uniquely by (a)--(c) and by local boundedness of the function   
$ x \mapsto \langle D, (x) - (b)\rangle_\q$ 
  for every fixed $b\in A(\bar{K}_\q)\setminus\supp(D)$.
  By~\cite[Proposition~2.3.1]{MT83}, the pairing~\eqref{pairingawayR}
  is equal to the canonical $\chi_\q$-pairing defined
  in~\cite[\S2.3]{MT83} with respect to the canonical $\chi_\q$-splitting
  from~\cite[Theorem~1.5]{MT83}. 
\end{rk}

\subsubsection{Local height pairings above $p$}
Let $\p$ be a prime above $p$. Let $\Div_a(A_\p)$ be the subgroup of
$\Div(A_\p)$ consisting of
  divisors algebraically equivalent to~0.
 We fix a purely mixed curvature form $\alpha_\p$ on $\P_\p$ for
every $\p\mid p$. 
Let
$\frka = \sum_x n_xx\in Z^0_0(A_\p)$ be a zero-cycle of degree~0 and let
$D\in\Div_a(A_\p)$ have disjoint support from $\frka$.
Suppose that $s$ an (algebraic) meromorphic section of a line bundle $\L$ on $A_\p$  such that $D=\div(s)$ and let
$\log_{\L}$ be
the canonical log function on $\L$ induced by $\alpha_\p$ with respect to an arbitrary
choice of rigidification.
Then we define
\begin{equation}\label{pairingp}
  \langle \frka, D\rangle_\p\colonequals t_\p(\sum_x
  n_x\log_{\L}(s(x)))\,,
\end{equation}
where $\chi_\p = t_\p\circ \log_\p$.
\begin{lemma}\label{L:localMTp}
  The pairing~\eqref{pairingp} satisfies
  \begin{enumerate}[\upshape (a)]
    \item $\langle \cdot,\cdot\rangle_\p$ is biadditive,
    \item $\langle \cdot,\cdot\rangle_\p$ is translation-invariant,
    \item if $D=\div(f)$, then $\langle \frka, D\rangle_\p =
      \chi_\p(f(\frka))$.
  \end{enumerate}
  In particular, it satisfies the conditions of~\cite[\S2.2]{MT83}. 
\end{lemma}
\begin{proof}
  The first property follows from Proposition~\ref{P:tensorcan}.
Now let $\frka, D$ and $\L$ be as above and let $a\in A_\p$. By Proposition~\ref{P:trans},
  the canonical log function $\log_{t_a^*\L}$ on
  $t_a^*\L$  is the same as $\log_{\L}$ up to an
  additive constant. Since $\deg(\frka)=0$, we obtain 
  \begin{equation}\label{}
    \langle t_a^*(\frka), t_a^*(D)\rangle_\p =
    t_\p(\log_{t_a^*\L}(t_a^*(s(\frka)))) = 
    t_\p( \log_{\L}((t_a)_*\circ t_a^*(s(\frka)))) =
    t_\p(\log_{\L}(s(\frka))) = \langle \frka, D\rangle_\p
  \end{equation}
  as in the proof of~\cite[Theorem~9.5.11]{BG06}.

  For (c), let $D=\div(f)$ be principal and consider $f$ as a section of $\O_A$. Let
  $\log'$ be the canonical log function on $\O_A$. Then we have
  $\langle \frka, D\rangle_\p = t_\p(\log'(f(\frka)))$. But by
    Proposition~\ref{P:cangood}, $\log'$ is the
    trivial log function. 
\end{proof}
\begin{rk}\label{R:chipsplitting}
  By~\cite[\S2.2]{MT83}, the pairing $\langle\cdot,\cdot\rangle_\p$ is induced from a
  $\chi_\p$-splitting in the sense of~\cite[Section~1]{MT83}. Hence the choice of
  curvature form $\alpha_\p$ induces a $\chi_p$-splitting.
\end{rk}
\begin{rk}
In contrast to Lemma~\ref{L:localMTq}, we cannot hope for the pairing to be uniquely determined by these properties, since the canonical log function depends on the choice $\alpha_\p$.
\end{rk}

\subsubsection{Global height pairings}
Let 
$$\hat{h} = h_{\P,\ba, \chi}\colon A(K)\times \hat{A}(K)\to \Q_p$$ 
be the canonical height pairing~\eqref{E:canht} relative to the choices $\ba$ and $\chi$. 
It is explained in~\cite[(3.1.1)]{MT83} that the sum of the local
pairings $\langle\cdot,\cdot \rangle_v$ defines another global pairing 
$$\langle \cdot, \cdot\rangle\colon A(K)\times \hat{A}(K)\to \Q_p\,.$$
We now review the construction of the latter pairing and show that both pairings are
equal.

For $\frka \in Z^0_0(A)$ and 
$D\in\Div_a(A)$ with disjoint support, we define 
$$
\langle \frka, D\rangle \colonequals \sum_v\langle \frka_v, D_v\rangle_v\,.
$$
\begin{lemma}\label{L:MTcomp}
 The pairing $\langle \cdot, \cdot\rangle$
  induces a well-defined pairing on $A(K)\times \hat{A}(K)$.
\end{lemma}
\begin{proof}
  Let $\frka \in Z^0_0(A)$ and 
$D\in\Div_a(A)$ have disjoint support.
  It is obvious that $\langle \frka, D\rangle$ does not change if we replace $D$ by
  another divisor $D'$ such that $\O(D)\simeq \O(D')$. 
Let $S\colon Z^0_0(A) \to A$ be the summation map;
its kernel is generated by
  cycles of the form $t_x^*Z-Z$. Since the class of $D$ is
  antisymmetric, $t_{-x}^*D - D = \div(f)$ is principal for any $x \in A$. We find that \begin{equation}\label{kerS}
    \langle t_x^*Z-Z, D\rangle_v =  \langle Z, t_{-x}^*D - D\rangle_v = \chi_v(f(Z))
\end{equation}
for all $v$, and all $Z$ and $x$ such that the arguments have disjoint support. 
Hence $\sum_v\langle\cdot,\cdot\rangle_v$ factors through 
  $Z^0_0(A)/\ker(S)\times \Pic^0(A) = A\times \hat{A}$.
\end{proof}

\begin{cor}\label{C:MTcomp}
  Fix a continuous id\`{e}le class character $\chi\colon
  A^\times_K/K^\times \to \Q_p$ and, for
every $\p\mid p$, a purely mixed curvature form $\alpha_\p$ for $\P_\p$.
  Let $\hat{h} = \hat{h}_{\P, \chi, (\alpha_\p)_{\p\mid p}}$ denote the corresponding
  $p$-adic height.
Let 
$\langle\cdot,\cdot \rangle$ denote the Mazur--Tate height pairing with respect to the
  following
$\chi$-splitting: For $\q\nmid p$, the splitting is the canonical $\chi_\q$-splitting, and
  for $\p\mid p$ it is the $\chi_\p$-splitting induced by $\alpha_\p$ as in
  Remark~\ref{R:chipsplitting}.

 Then we have 
$$\hat{h}(a,\hat{a}) = \langle a, \hat{a}\rangle$$
for every 
$a\in A(K)$ and $\hat{a}\in \hat{A}(K)$. 
\end{cor}
\begin{proof}
  The definitions of $\hat{h}$ and $\langle\cdot, \cdot\rangle$ imply that 
\begin{align*}
  \hat{h}(a,\hat{a}) &=   \hat{h}(a,\hat{a}) - \hat{h}(0,\hat{a}) = 
   \langle (a)- (0), D\rangle,
\end{align*}
  for any $D$ such that $[\O(D)]=\hat{a}$ and such that $a, 0\notin \supp(D)$.
\end{proof}

\begin{rk}
  One can define a notion of $p$-adic N\'eron functions based on our log functions
  and $\Q$-valuations and use it to give a local decomposition of (Mazur--Tate)
  height \textit{functions}, analogous to the classical decomposition of
  N\'eron--Tate height functions as sums of real-valued N\'eron functions.
  See~\cite{BKM23}.
\end{rk}

\subsubsection{Jacobians}\label{subsec:jacs}
Let $J$ be the Jacobian of a nice curve $C/K$. We assume for simplicity that there is an
Abel--Jacobi map $\iota\colon C \to J$ defined over $K$.
Let $\Theta$ denote the theta divisor with respect to $\iota$. 
Then $J$ is self-dual via
the principal polarization $\phi_\Theta\colonequals \phi_{\O(\Theta)}\colon J\to \hat{J}$.
More precisely, by~\cite[Proposition~8.10.20]{BG06}, we have
$$
(\id\times \phi_{\Theta})^*\P \isom s^*\Theta\otimes (\pi_1^*\Theta)^{-1}\otimes
(\pi_2^*\Theta)^{-1} \equalscolon \delta
$$
with notation as in~\S\ref{sec:app}, and $\Delta^* \delta \isom \O(\Theta)\otimes [-1]^*\O(\Theta)$.
Therefore, for all $a\in A(K)$, we have
$$h_{\P}(a,\phi_{\Theta}(a)) = \hat{h}_{\delta}(a,a) =
\hat{h}_{\O(\Theta)\otimes [-1]^*(\O(\Theta))}(a)$$
where the heights are defined with respect to $\chi$ and the curvatures induced by $\ba$.
We obtain a bilinear pairing 
\begin{align*}
  J(K)\times J(K)&\to \Q_p\\
  (a,b)&\mapsto \hat{h}(a, \phi_{\Theta}(b)) \,.
\end{align*}
We can express this global height pairing as a sum of
local pairings on the curve.
Let $v$ be a non-archimedean place of $K$ and let $a,b \in J(K_v)$. Let 
$w\in
\Div^0(C_v)$ be a representative  of $a$.  Write $w=\div(s)$, where $s$ is a meromorphic section of a line bundle $L\in
\Pic^0(C_v)$. Let $M\in \Pic^0(J_v)$ such that $\L =\iota^*M$. Pick a 
representative $z=\sum_x n_x(x) \in \Div^0(C_v)$ of $b$ such that $w$ and $z$ have disjoint
support. If
$v=\q\nmid p$ with uniformizer $\pi_\q$, then we 
define 
\begin{equation}\label{pairingcurveaway}
    \langle z,w\rangle_{\q} \colonequals \sum_x
    n_xv_{\L,\q}(s(x))\chi_{\q}(\pi_{\q})\,,
    \end{equation}
where $v_{\L,\q} = \iota^*v_{M,\q}$ and $v_{M,\q}$ is the canonical valuation on $M$.
By Galois-equivariance, we have $\langle z,w\rangle_{\q} \in \Q_p$.
Then, up to a constant factor,
$\langle\cdot, \cdot\rangle_\q$ is the local N\'eron symbol on the curve,
see~\cite[Theorem~9.5.17]{BG06}.

If $v=\p\mid p$, let $\log_{M}$ be the canonical log function on $M$ and  let $\log_{L}=\iota^*\log_{M}$.
Then we define 
\begin{equation}\label{pairingcurvep}
  \langle z,w\rangle_{\p} \colonequals t_\p\left(\sum_x
  n_x\log_{\L}(s(x))\right)\,.
\end{equation}
 As above, Galois-equivariance of Vologodsky functions implies that $\langle z,w\rangle_{\p} \in \Q_p$.
Suppose that $w=w_1-w_2$,
where $w_1$ and $w_2$ are effective divisors of degree equal to the genus
of $C$ whose Riemann--Roch space has dimension~1.
For $b_i = \iota(w_i)\in J(K_v)$ we have
by~\cite[Theorem~5.5.8]{Lan83}, $w_i = \iota^*\Theta^-_{b_i}$, where
$\Theta^-_{b_i}= t_{b_i}^*\Theta^-$ and  $\Theta^- = [-1]^*\Theta$.
Moreover, write $z = \sum_i (c_i) - \sum_i(d_i)$, with $c_i,d_i \in C_v$ and let $\frka = \iota_*z =
\sum_i (\iota(c_i)) - \sum_i(\iota(d_i))\in Z_0^0(J_v)$. 
We immediately find that
\begin{equation}\label{E:mt_cg_local}
  \langle \frka,\Theta^-_{b_1} - \Theta^-_{b_2}\rangle_v = \langle z,w\rangle_v
\end{equation}
for all $v$, where the left hand side is defined in~\eqref{pairingaway}
(respectively~\eqref{pairingp}) and the right hand side is defined in 
~\eqref{pairingcurveaway}
(respectively~\eqref{pairingcurvep}) if $v\nmid p$ (respectively $v\mid p$).

\begin{prop}\label{P:MThtglobalpb}
  Let $a, b \in J(K)$ and let $z,w\in \Div^0(C)$ be representatives with disjoint support of $a$ and $b$,
  respectively. 
  Then 
  \begin{equation}\label{Mhtglob}
   \hat{h}(a,\phi_{\Theta}(b))  = -\sum_v\langle z_v, w_v\rangle_v\,.
  \end{equation}
\end{prop}
\begin{proof}
  First suppose that $w=w_1-w_2$, where the $w_i\in \Div^g(C)$ are effective and
  non-special.
  By construction, we have $a=S(\iota_*z) = S(\frka)$, where $S\colon Z^0_0(J) \to J$ is the
  summation map, and $\phi_\Theta(b) = [\Theta_{b_1} - \Theta_{b_2}]$.
  Using Corollary~\ref{C:MTcomp} and~\eqref{E:mt_cg_local}, we obtain
$$
  \hat{h}(a,\phi_{\Theta}(b)) = - \langle \frka,\Theta^-_{b_1} - \Theta^-_{b_2}\rangle
  = -\sum_v \langle \frka,\Theta^-_{b_1} - \Theta^-_{b_2}\rangle_v = -\sum_v
\langle z_v,w_v\rangle_v,
$$
  because $[-1]$ acts on $\Pic^0(A)$ as multiplication by $-1$.

  There is a finite extension $K'/K$ such that $a$ has a
  representative $w=w_1-w_2 \in \Div^0(C_{K'})$, where the $w_i$ are as above. Both sides
  of~\eqref{Mhtglob} can be extended to $K'$.
  Since the sum $\sum_v\langle z,w\rangle_v$ is independent of the choices of $z$ and $w$,
  we deduce~\eqref{Mhtglob}.
\end{proof}

\subsection{Coleman--Gross}\label{subsec:CG}
As in~\S\ref{subsec:jacs}, let $C/K$ be a nice
curve with Jacobian $J$ and let $\iota\colon C \to J$ be an Abel--Jacobi map defined over
$K$. Let $\Theta$ be the associated theta divisor. 
We now show that for appropriate choices, our height pairing on $J$ is the same as the
one constructed by Coleman and Gross~\cite{CG89}.
In fact, to be slightly more general and to
avoid the need to translate between Vologodsky and Coleman integration, we will compare our
height pairing with the one constructed in~\cite{Bes17}, which is the same
as the Coleman--Gross
height pairing with Coleman integration replaced by Vologodsky integration, so that it applies
to curves with bad reduction above $p$ as well. We will nevertheless continue to call this the
Coleman--Gross height pairing. Note that, for general
reduction, this pairing was
actually constructed earlier by Colmez in~\cite{Col98},
see~\S\ref{S:Colmez} below.

The Coleman--Gross $p$-adic height pairing 
$$
h^{CG}\colon J(K)\times J(K)\to \Q_p
$$
depends, in addition
to $\chi$, on the choice, for every $\p |p$, 
of a splitting of the Hodge filtration on $\hdr^1(C_{\p})$; in other words, a
subspace $W_{\p}$ of $\hdr^1(C_\p)$ complementary to the image of the holomorphic differentials.

The goal of this section is to prove the following comparison result.
\begin{thm}\label{T:CGcomp}
    For $\p \mid p$,  let $\alpha_{\p}$ be a choice of curvature form of $\P_\p$ and let
    $W_{\p}$ be the complementary subspace associated to $\alpha_{\p}$ in
    Proposition~\ref{P:curvspace} below. Write 
$\ba = (\alpha_{\p})_{\p \mid p}$ and
$\underline{W} = (W_{\p})_{\p \mid p}$. 
   Relative to these choices, let $\hat{h} = \hat{h}_{\P, \ba, \chi}$ be the canonical height and let
  $h^{CG} = h^{CG}_{\underline{W}, \chi}$ be the Coleman--Gross height pairing.
  Then we have
\begin{equation}\label{E:globaloursCGequal}
  \hat{h}(a,\phi_\Theta(b)) = -h^{CG}(a,b)
\end{equation}
  for all $a,b\in J(K)$. 
\end{thm}

\begin{rk}
  In particular, Coleman--Gross heights are (up to sign) Mazur--Tate
  heights, with appropriate choices. To the best of our knowledge, this has
  not been shown in the literature in this level of generality. For genus~2
  curves, an explicit comparison result is due to
  Bianchi~\cite[Corollary~5.35]{Bia23}. Coleman showed in~\cite{Col91} that
  if $C$ has good ordinary reduction at all places above $p$, then the
  Coleman--Gross height with respect to the unit root splitting is the
  Mazur--Tate height with respect to the canonical splitting constructed
  in~\cite[\S1.9]{MT83}. See~\cite{BKM23} for a local comparison result in
  genus~2, extending the local comparison for elliptic curves
  in~\cite{BB15}.
\end{rk}

\subsubsection{Local decomposition of the Coleman--Gross height pairing}\label{S:CGlocaldecomp}

The Coleman--Gross height is defined on a pair of divisors $z,w\in
\Div^0(C)$ with disjoint support, as the finite sum 
\begin{equation*}
  h^{CG}(z,w) =\sum_{v} h^{CG}_{v}(z,w)\;,
\end{equation*}
over all finite primes $v$ of $K$, of local height pairings. For every such $v$ the local
height pairing $h^{CG}_{v}(z,w)$ depends only on the respective images $z_v, w_v$ of $z,w$ in
$\Div^0(C_v)$.

\begin{rk}\label{R:localequalaway}
For $\q\nmid p$, the local Coleman--Gross pairing between divisors $z,w\in
\Div^0(C_{\q})$ with disjoint support is 
the unique pairing satisfying the conditions of Lemma~\ref{L:localMTq}
(see~\cite[Section~2]{CG89}), so it is immediately seen to be equal to our pairing $\langle z, w\rangle_\q$.
\end{rk}

It remains to consider places $\p\mid p$. The local Coleman--Gross pairing $\langle z,w\rangle_{\p,}$ for $z,w\in
\Div^0(C_{\p})$ with disjoint support is given in this case
as follows: There is a map 
\begin{equation*}
  \Psi\colon \Omega_{K_{\p}(C_{\p})}^1 \to \hdr^1(C_{\p})\;,
\end{equation*}
where the source of the map is the space of meromorphic 1-forms on $C_{\p}$.
This map, constructed in~\cite[Definition~3.9]{Bes05} (there it is called $\Psi'$ ) is an
extension of the map $\Psi$ considered in~\cite{CG89},
which is the logarithm of the universal vectorial extension of the Jacobian of $C$. It has the following two properties:
\begin{equation}\label{E:psipropa}
  \Psi (\eta) =[\eta] \text{ where $[\eta]$ is the cohomology class of the form of the second
  kind $\eta$; }
\end{equation}
  \begin{equation}\label{E:psiprop}
  \Psi d\log(f) =0 \text{ for any rational function $f$}\;.
\end{equation}

It easily follows that corresponding to the divisor $w$ there exists a
unique meromorphic form $\omega_w$ with the properties
\begin{enumerate}[(i)]
\item the form $\omega_w$ has log singularities and its residue divisor is $w$;
\item we have $\Psi(\omega_w)\in W_{\p}$.
\end{enumerate}
The local height pairing is given by 
\begin{equation}\label{eq:cgp}
  h^{CG}_{\p}(z,w) = t_{\p}(\int_z \omega_w)\;,
\end{equation}
where $t_{\p}$ and the branch of the logarithm used for integration are obtained from~\eqref{tdefnd}.

For the same
$z,w\in \Div^0(C_\p)$ as above, our local pairing $\langle z,w\rangle_{\p}$ defined
in~\eqref{pairingcurvep}  depends on the choice of a curvature form
$\alpha_\p$. In order to compare this with the local Coleman--Gross height pairing $h^{CG}_{\p}(z,w)$
we first have to compare this choice with the choice of a splitting of the Hodge filtration on
$\hdr^1(C_\p)$ needed for the Coleman--Gross pairing. This is what we will do in the next
subsection.

\subsubsection{Curvature forms and complementary subspaces}
For notational ease, let $F = K_{\p}$. Let $A$ be an abelian variety over $F$. Recall the
notion of purely mixed curvature forms from Definition~\ref{D:admcurv} and the discussion
preceding it. It is well known (see for example~\cite[p.~1380]{Col98}) that there is a natural duality
between $\hdr^1(A)$ and $\hdr^1(\hat{A})$.
By~\cite[Lemma~5.1.4]{BeBrMe82}, $\Omega^1(\hat{A})$ is   
precisely the annihilator of $\Omega^1(A)$ under the pairing.
In the next result, we use the notation introduced before Definition~\ref{D:admcurv}.
\begin{prop}\label{P:curvspace}
    There is a one to one correspondence between purely mixed curvature 
    forms $\alpha$ on $\P$ and complementary
    subspaces to $\Omega^1(A)$ in $\hdr^1(A)$
    which is given as follows:
    \begin{enumerate}
      \item  For $\alpha$ as above, the element 
     $\alpha_{21}\in H_{21}\isom \Omega^1(\hat{A})\otimes \hdr^1(A)$ gives
     a linear map from $\Hom(\Omega^1(\hat{A}),F)$ to $\hdr^1(A) $ and the corresponding
     subspace $W_\alpha$ is its image.
   \item
     Given a complementary subspace $W$, let $W'$ be its annihilator in $\hdr^1(\hat{A})$. Then
     $W$ is dual to $\Omega^1(\hat{A})$ while $W'$ is dual to $\Omega^1(A)$. The dualities
     provide elements
     \begin{align*}
       \alpha_{21} &\in \Hom(\Hom(\Omega^1(\hat{A}),F),W)\isom \Omega^1(\hat{A})\otimes W \subset
       \Omega^1(\hat{A})\otimes \hdr^1(A) \isom H_{21} \;,\;\\
       \alpha_{12} &\in \Hom(\Hom(\Omega^1(A),F),W')\isom \Omega^1(A)\otimes W' \subset
       \Omega^1(A)\otimes \hdr^1(\hat{A}) \isom H_{12} \;,
     \end{align*}
     and the corresponding curvature form is $\alpha_W=\alpha_{12}+\alpha_{21}$.
   \end{enumerate}
   Explicitly, the form $\alpha_W$ may be written as follows:
Choose bases
\begin{equation}\label{basis}
  \omega_k^1,~ \omega_k^2,~ k=1,\ldots, g
\end{equation}
  for $\Omega^1(A)$ and $\Omega^1(\hat{A})$, respectively. Take a basis $\eta_k^1$
  ($k=1,\ldots,g$) of $W$ which is
dual to the basis $\omega_k^2$ of $\Omega^1(\hat{A})$ and a basis $\eta_k^2$ of $W'$ which is
dual to the basis $\omega_k^1$. In this way, the $\omega_k^1$ and $\eta_k^1$ form a basis for
$\hdr^1(A )$ and the dual basis for $\hdr^1(\hat{A} )$ is provided by the $\eta_k^2$ and
$\omega_k^2$. The curvature form corresponding to $W$ is then given by
\begin{equation}\label{E:explicitalpha}
  \alpha_W = \sum_{k} \pi_1^\ast \omega_k^1 \otimes \pi_2^\ast \eta_k^2 
  - \pi_2^\ast \omega_k^2 \otimes \pi_1^\ast \eta_k^1\;.
\end{equation}
\end{prop}
\begin{proof}
It is easy to see that~\eqref{E:explicitalpha} is just an explicit form of the description of
$\alpha_W$ in terms of $W$.
We recall that the relation between $\P$ and the duality is that $\ch_1(\P)$ 
lies in $\hdr^1(A)\otimes \hdr^1(\hat{A})$, embedded in 
$\hdr^2(A\times \hat{A})$ by the K\"unneth formula and this 
class gives the required duality. In concrete terms, this means that with 
respect to any choice of dual bases for $\hdr^1(A ) $ and $\hdr^1(\hat{A})$, in
particular the bases of $\omega$'s and $\eta$'s we chose before, we have
\begin{equation*}
  \ch_1(\P)= \sum_k (\pi_1^\ast \omega_k^1 \cup \pi_2^\ast  \eta_k^2+ \pi_1^\ast \eta_k^1 \cup \pi_2^\ast
  \omega_k^2)\;.
\end{equation*}
This immediately shows that the curvature form $\alpha_W$ we have associated with $W$ indeed cups
to $\ch_1(\P)$. 
  
  Now fix one complementary subspace $W_0$ and choose bases
corresponding to $W_0$ as above, and write, in terms of this basis, 
  a general curvature form satisfying our two
  conditions as 
  \begin{equation}\label{alphageneral}
  \alpha = \sum_{i,j,k,l} a_{ij}^{kl} \pi_i^\ast \omega_k^i \otimes \pi_j^\ast \omega_l^j
  +\sum_{i,j,k,l} b_{ij}^{kl} \pi_i^\ast \omega_k^i \otimes \pi_j^\ast
  \eta_l^j\;\;\text{such that}\;\, 
  \alpha_{11}=\alpha_{22}=0\;\,\text{ and }\;\, \cup\alpha=\ch_1(\P)
\;,
\end{equation}
  where the $a^{kl}_{ij}$ and $b^{kl}_{ij}$ are constants for which we now need to find
  restrictions. The fact that $\alpha$ cups to $\ch_1(\P)$ implies that $b_{12}^{kk}=1$ and
  $b_{21}^{kk}=-1$ for all $k$ and all other $b^{kl}_{ij}$
vanish and that $a_{ij}^{kl}=a_{ji}^{lk}$ for all possible indices. The condition
$\alpha_{11}=\alpha_{22}=0$ implies that  $a_{11}^{kl}=a_{22}^{kl}=b_{11}^{kl}=b_{22}^{kl}=0$
  for all $k,l$. For an $\alpha$ as in~\eqref{alphageneral}, let us now compute the corresponding $W_\alpha$. We can write
\begin{equation*}
  \alpha_{21}=\sum_k  \pi_2^* \omega_k^2 \otimes  \pi_1^* \overline{\omega}_k 
\end{equation*}
for appropriate classes  
\begin{equation*}
  \overline{\omega}_k = -  \eta_k^1+ \sum_{l} a_{21}^{kl}  \omega_l^1 \in \hdr^1(A)
\end{equation*}
and $W_\alpha=\Span(\overline{\omega}_1,\ldots,\overline{\omega}_g)$.
As the spanning vectors for $W_\alpha$ are congruent modulo $\Omega^1(A)$ to the basis elements
$\eta_k^1$, it is clear that $W_\alpha$ is indeed complementary. 

  To complete the proof, it
remains to show that the two constructions are inverses of one another. It is immediate that
$W_{\alpha_W}= W$. To see that also $\alpha_{W_\alpha}=\alpha$ we now repeat the computation
above of a general curvature form, but this time we start, instead of the basis provided by
$W_0$, with the basis provided by $W=W_\alpha$ itself. With this choice of basis it is now
  immediate that all $a^{kl}_{ij}$ are $0$ and therefore $\alpha=\alpha_W$.
\end{proof}

\subsubsection{Forms of third kind with prescribed residues from log functions}
Continuing with the notation of the previous subsection, suppose now that the abelian variety
is $J=\mathrm{Jac}(C)$, the Jacobian
of a nice curve $C$ over $F=K_\p$. There is an isomorphism $\hdr^1(C) \isom \hdr^1(J)$, compatible with the Hodge filtration.
Hence the choice of a complementary subspace $W \in \hdr^1(C)$, as required for the construction of the local
Coleman--Gross pairing,
 determines a complementary
subspace in $\hdr^1(J)$. We denote this subspace by $W$. By
Proposition~\ref{P:curvspace}, this  choice determines a curvature form
$\alpha$ on the Poincar\'e line bundle on $J\times \hat{J}$.
Let $w\in \Div^0(C)$ and suppose that $\L$ is a line bundle on $C$ and $s$ a
rational section of $\L$ such that $\div(s) = w$. Then 
$\alpha$ induces a canonical log function $\log_\L$ and 
a local height pairing  $\langle \cdot,\cdot\rangle_\p$ as in~\eqref{pairingcurvep}.
The function $\log_{\L}(s)$
is locally analytic outside the support of $w$ and $d \log_{\L}(s)$ is independent of the
ambiguity in $\log_L$. 

\begin{prop}\label{P:divfromlog}
  In the situation above we have $d\log_{\L}(s)= \omega_w $.
\end{prop}
By the construction of the Coleman--Gross height pairing in~\S\ref{S:CGlocaldecomp} and our
pairing in \eqref{pairingp} it is clear that Proposition~\ref{P:divfromlog} implies:
\begin{cor}\label{C:pairingspequal}
  We have $h^{CG}_\p(z,w) = \langle z,w\rangle_\p$.
\end{cor}

\begin{proof}[Proof of Proposition~\ref{P:divfromlog}]
  The metrized line bundle $(\L,\log_{\L})$ is flat, because $\log_\L$ is the pullback of a log function on a
line bundle on $J$, which is flat by Theorem~\ref{T:antigood} and
  Corollary~\ref{C:anticangood} because $\alpha$ is purely mixed. 
  Thus, 
the 1-form $d\log_{\L}(s)$ is meromorphic, and it has simple poles with residue divisor exactly
$w$ by Lemma~\ref{L:iflogexists}.
Proving that $d \log_{\L}(s) = \omega_w$ is therefore equivalent to showing that 
  $\Psi(d\log_L(s))\in W$. 
For a given line bundle $\L$ of degree $0$, the class
$\Psi(d\log_L(s))\in \hdr^1(C)$ is independent of the choice of the section $s$,
by~\eqref{E:psiprop}, because for another $s$ the form will
differ by a $d\log$ of a rational function. The association
\begin{equation*}
  \L  \mapsto \Psi(d\log_L(s))
\end{equation*}
clearly maps tensor products to sums and therefore gives a
homomorphism 
\begin{equation*}
  r\colon J(F) \to \hdr^1(C )\;.
\end{equation*}
  \begin{lemma}\label{L:rlocana}
    The map $r$ is locally analytic.
\end{lemma}
\begin{proof} 
Fix a degree $g$ divisor $D$ on $C$, and consider the associated map 
\begin{align}\label{E:deffD}
 f_D\colon C^g &\to J=\Pic^0(C) \\
    \underline{P} &\mapsto  \left[\sum_i P_i-D \right] \,.
\end{align}
Let $U$ be the open subset of $C^g$ where no two coordinates are equal. Since $J$ is covered by the image of $f_D|_U$ as we vary $D$, and the topology on $J$ is induced from the topology on $C^g$, it suffices to show the local analyticity of $r_D \colonequals r \circ f_D|_U$ for every $D$. Furthermore, since we get coordinates on the target $\hdr^1(C)$ as a topological $F$-vector space by cupping with the cohomology classes of forms of the second kind on $C$, it suffices to pick one such form, $\eta$, and show the local analyticity of the map $U \rightarrow F$ given by $\underline{P} \mapsto \eta \cup r_D(\underline{P})$. (To ease notation, we will henceforth drop the restriction of $r_D$ to the open subset $U$ of $C^g$ and talk about analyticity on $C^g$, keeping in mind that we are only interested in proving analyticity at points where no two coordinates are equal.)

By definition \( r_D(\underline{P}) \) is \(   \Psi(d\log_L(s))\), where \(
L= \O(\sum_i P_i -D) \) and \( s \) is any rational section of this line bundle. To
prove local analyticity we will pick the canonical section \( s=1 \) of $L$. Let \( \omega_{\underline{P}} \colonequals d \log_{\O(\sum P_i -D)} 1 \).  We now need to prove that \( \underline{P} \mapsto \eta \cup \Psi(\omega_{\underline{P}}) \) is
locally analytic in \( \underline{P} \).

To prove this, let $\pi_C \colon C^g \times C \rightarrow C$ be the projection onto the
second factor. We will first show that the forms \( \omega_{\underline{P}} \)
patch together to a locally meromorphic section of
 $
\pi_C^*(\Omega^1_C)$, locally analytic outside the support of
\[
  S\colonequals \{(P_1,\ldots ,P_g,P )\;,\quad P=P_i \text{ for some } i \}\;.
\]  
Let \( \mathcal{L} \colonequals \O(S - \pi_C^\ast D) \), a line bundle on \( C^g \times C \).
The fiber of \( \mathcal{L} \) over \( \underline{P}\times C \) is \(
\O(\sum_i P_i - D )\).
Consider the map 
$f_D \times
\id\colon C^g \times C \rightarrow J \times C$ where $f_D$ is as in \eqref{E:deffD}. Then 
by the universal property
of the Poincar\'e bundle $\P_C$, we have a canonical isomorphism of the fiber
$\left( f_D \times \id \right)^* \P_C|_{\underline{P} \times C} \cong
L$. Let \( \log_{\mathcal{L}} \) be the pulled
back log function. Since the isomorphism $\psi_{L,r}: (L,r) \cong \P|_{J \times \hat{a},r_{\hat{a}}}$  appearing in the definition of $\log_\L$ factors through $(f_D \times \id)^*$, the restriction of \( \log_{\mathcal{L}}
\) to \( \underline{P}\times C \) is \( \log_L \). 
The line bundle \( \mathcal{L} \) has a canonical rational section \( 1 \),
regular outside the support of \( S - \pi_C^\ast D\) and 
restricting to the section \( 1 \) of \( L \) over each \(\underline{P} 
\). Consequently, we get
\begin{equation*}
  \omega_{\underline{P}}= 
  d \log_{\O(\sum P_i -D)} 1=
  d ((\log_{\mathcal{L}} 1)|_{\underline{P}\times C})
  = (\pi_{\Omega_C^1}(d \log_{\mathcal{L}} 1))|_{\underline{P}\times C}
\end{equation*}
where \( \pi_{\Omega_C^1} \) indicates the projection corresponding to the
direct sum decomposition \( \Omega_{C^g\times C}^1= \Omega_C^1 \oplus
\Omega_{C^g}^1 \).
Since $d \log_{\mathcal{L}} 1
$ is locally meromorphic, we get that the \(
\omega_{\underline{P}} \) patch together to a locally meromorphic section of
 $
\pi_C^*(\Omega^1_C)$, locally analytic outside the support of $S$ as claimed before.

Now, by~\cite[Corollary~3.14]{Bes05} we have
  \begin{equation}\label{}
    \eta \cup \omega_{\underline{P}} = \sum_x \res_x \left(\omega_{\underline{P}}\int
  \eta\right) \;,
  \end{equation}
The right hand side is a 
finite sum of residues at the $P_i$, points in the support of $D$ and
  singular points of $\eta$. We may assume that the singular points of \(
  \eta \) are disjoint from the other points. Since \(
  \omega_{\underline{P}} \) has simple poles by our assumption that $\underline{P}$ is in the set $U$ where no two $P_i$s are equal, and its residue divisor is
  exactly \( \sum P_i -D \), by Lemma~\ref{L:iflogexists}, we get
  \begin{equation}\label{}
    \eta \cup \omega_{\underline{P}} = \left(\int \eta \right) (\sum P_i -D)+ \sum_{x'}
    \res_{x'} \left(\omega_{\underline{P}}\int
  \eta\right) \;,
  \end{equation}
  where the sum now is only over the singular points \( x' \) of \( \eta
  \). Since \(
  \int \eta(\sum P_i -D) \) is locally analytic in \( \underline{P} \), it suffices to prove local analyticity of the sum. Let \( z \) be a local parameter near one such point \( x'=x_0 \). Write the Laurent expansion \( \int  \eta = \sum_i a_i z^i \) in a (analytic) neighborhood of $x_0$, and the expansion
   \( \omega_{\underline{P}} = \sum_j
  b_j(\underline{P}) z^j dz \) around the point \(( (\underline{P}_0), x_0) \). Note
  that each \( b_j \) is locally analytic in $\underline{P}$. Then, \[
    \res_{x_{0}} \omega_{\underline{P}} \int \eta = \sum_{i+j=-1}a_i
    b_j(\underline{P})
  \] 
  (a finite sum) is locally analytic near \( \underline{P}_0 \) by local analyticity of $b_j$. This finishes the proof.
\end{proof}

Proposition~\ref{P:divfromlog} follows if we can show that the image of $r$ is contained in $W$. 

We will follow closely the proof of Theorem~7.3 in~\cite{Bes05}. 
By Lemma~\ref{L:rlocana}, it suffices to show that the
derivative of $r$ with respect to any vector field on $J$ at one, hence at any, point is in $W$.
The computation is essentially the same as in the proof of Theorem
  7.3 in~\cite{Bes05}. Let $\P_C$ be the pullback of $\P$ to  $C\times J$ via the map $C\times J \to J\times J$.
Let $\alpha_C$ be the curvature of $\P_C$ induced by pullback of $\alpha$. It has exactly the same formula
as~\eqref{E:explicitalpha}
under the identification of the cohomologies and forms of $C$ and $J$. To compute the
derivative of $r$ with respect to a vector field $T$ on $J$, we pick a rational section $s$
defined on an open $U\subset C\times J$. Locally, in the analytic topology on $J$, the relative
form $d \log_{\P_C}(s)$ is a family of forms of the third kind on $C$ 
  in the sense of~\cite[Definition~A.1]{Bes05} (called
there relative forms of the third kind). We need to apply
$\Psi$ fiber by fiber and then differentiate with respect to $T$. However,
because the derivative of a family of forms of the third kind is a form of the second
kind by Definition~A.3 in \cite{Bes05} and the following discussion, and these are sent by $\Psi$ to their associated cohomology classes
according to~\eqref{E:psipropa},
we can instead differentiate with respect to $T$ and then take the cohomology class fiber by
fiber. We now repeat the computation from the proof of~\cite[Theorem~7.3]{Bes05},
adjusting notation. We need to compute $\partial_T d\log_{\P_C}(s)$. This is the same as
$d(d\log_{\P_C}(s)|_{\partial_T})$, where the last notation means retraction in the direction
of $T$. Then, just as in the above reference, we notice that 
$\overline{\partial}(d\log_{\P_C}(s)|_{\partial_T})=\overline{\partial}(d\log_{\P_C}(s))|_{\partial_T}$
with the retraction acting on the first coordinate of the tensor product. The term
  $\overline{\partial}(d\log_{\P_C}(s))$ is precisely the curvature $\alpha_C$ restricted to $U$. Just as
in~\cite{Bes05} after retraction and restriction to the fiber of the projection to $J$, the only term
that survives is 
\begin{equation*}
  \sum \omega_k^2|_{\partial_T}\overline{\omega}_k\,.
\end{equation*}
Again as in~\cite{Bes05}, this shows that the required cohomology class lies in the subspace
spanned by the $\overline{\omega}_k$, which is $W$.
\end{proof}

We finally prove the main theorem of this section.
\begin{proof}[Proof of Theorem~\ref{T:CGcomp}]
It suffices, by Proposition~\ref{P:MThtglobalpb}, to show
that $h^{CG}_{v}(z,w)=\langle z,w\rangle_v$, where the latter is defined in~\eqref{pairingcurveaway}
(respectively~\eqref{pairingcurvep}) if $v\nmid p$ (respectively $v\mid p$).
The result now follows from Remark~\ref{R:localequalaway} and Corollary~\ref{C:pairingspequal}.
\end{proof}

\begin{rk}\label{R:localnotequal}
The proof strategy above for Theorem~\ref{T:CGcomp} does not show that the local
  contributions to both sides of \eqref{E:globaloursCGequal} (with respect to suitable
  choices) are equal. In fact, as stated, there is no natural local decomposition of the
  left hand side of \eqref{E:globaloursCGequal} (which is also the left hand side of the
  equality \eqref{Mhtglob} in Proposition~\ref{P:MThtglobalpb}). The proof
  does show that the local Coleman--Gross height pairing coincides with the local pairing defined in~\eqref{pairingaway} and ~\eqref{pairingp}.
\end{rk}

\section{Quadratic Chabauty}\label{sec:qc}
Let $C/\Q$ be nice curve of
genus $g >1$ such that $C(\Q)\ne \varnothing$ and let $p$ be a prime number.
Our goal in this section is to use the theory of (canonical) $p$-adic heights developed above (in particular,
the properties of the local terms) to construct a nonconstant locally analytic function on $C(\Q_p)$
whose values on $C(\Q)$ can be controlled.
Fixing a base point $b \in
C(\Q)$, we let $\iota=\iota_b\colon C\into J$ denote the corresponding
Abel--Jacobi map, inducing a morphism $\NS(J)\to
\NS(C)$, where $J$ is the Jacobian of $C$.
We fix a purely mixed curvature form $\alpha$ on $\P_p$. 

We assume that  $\ker(\NS(J)
\to\NS(C))$ has positive rank. 
 This assumption,  combined with the fact that the restriction of $\iota^*$
 to the subgroup $\Pic^0(J)$ of the Picard group $\Pic(J)$ is an
 isomorphism onto $\Pic^0(C)$, guarantees that there is some line bundle $\L$ on $J$ such that
  $\iota^*\L= \O_C$  and $L\notin \Pic^0(J)$. Henceforth we fix such a line bundle $\L$.
Since $\iota^*\L= \O_C$, we get a map on (the complement of the zero section of) total spaces of line bundles $\tilde{\iota} \colon C \times \mathbb{G}_m \rightarrow \L^{\times}$. We let $1$ denote the section $x \mapsto (x,1)$ of the natural projection $C \times \mathbb{G}_m \rightarrow C$
We endow $\L$ with the rigidification $r_b \colonequals\tilde{\iota}(1(b))$;
the corresponding canonical adelic metric $(\log_{\L,p}, \{v_{\L, q}\}_{q})$ then
satisfies 
 \begin{equation}\label{rigidq}
   v_{\L,{q}}(r_b) =0 
 \end{equation}
 for all $q\ne p$ and 
 \begin{equation}\label{rigidp}
   \log_{\L,{p}}(r_b) =0\,. 
 \end{equation}
We write $\hat{h}_{\L}$ for the corresponding $p$-adic height. 
By pullback, we obtain an adelic metric on 
$\O_C$: 
\begin{equation}\label{vq}
(\log_{p}, \{v_{q}\}_{q}) \colonequals \iota^*(\log_{\L,p}, \{v_{\L, q}\}_{q})\,.
\end{equation}
  For a prime $q\ne p$, the valuation $v_{q}$ is a $\Q$-valuation (see
  Remark~\ref{R:Qval}) and we define a local height function
  \begin{equation}\label{eq:lambdadef}
  \lambda_{q}\colon C(\Q_{q})\to \Q\,;\quad x \mapsto v_{q}(1(x))\,.
  \end{equation}

\begin{lemma}\label{L:fin_many}
  Let $q\ne p$ be a prime. Then $\lambda_q$
  \begin{enumerate}
    \item\label{fin} takes only finitely many values;
    \item\label{id0} is identically 0 if $C$ has potentially good reduction at
  $q$.
  \end{enumerate}
\end{lemma}
\begin{proof}
Since $v_{\L, q}$ is a canonical valuation on $\L$, Proposition~\ref{P:good} implies that 
  $\lambda_q$ is a locally constant $\Q$-valued function on $C(\Q_q)$, implying the first
  statement. For the second, we may assume that $J$ has good reduction at $q$ by passing
  to an extension. Then a canonical valuation on $\L$ is a model valuation on the N\'eron
  model~$\mathcal{J}$ by Remark~\ref{R:good_good}.
  The closure $\mathcal{L}$ of $\L$ on $\mathcal{J}$ pulls back to an extension of
  $\O_C$ to a smooth model $\mathcal{C}$ of $C$, so $v_q$ is constant, hence
  identically~0 by~\eqref{rigidq}.
\end{proof}
We will give  a more precise version of~\eqref{fin} below in
Proposition~\ref{P:factor}.

We now combine Lemma~\ref{L:fin_many} with the quadraticity of $\hat{h}_L$
and the fact that $\log_p\circ 1$ is a Vologodsky function to explain our version of Quadratic Chabauty.

\subsection{Quadratic Chabauty for rank $=$ genus}\label{S:reqg}
Suppose, in addition to $\rank\NS(J)>1$, that $\rank J(\Q) =g$. We also assume that $p$ is a prime number such that the $p$-adic closure of $J(\Q)$ has finite index in $J(\Q_p)$.
Fix a basis $(\omega_0,\ldots,\omega_{g-1})$ of $H^0(J,\Omega^1)$.
By abuse of notation, we also write $\omega_i$ for $\iota^*\omega_i$.
Our assumptions on $J$ imply that $(J(\Q)\otimes \Q_p)^\vee$ is generated by $f_0,\ldots,
f_{g-1}$ , where $f_i(x) = \log(x)(\omega_i) = \int^x_0\omega_i$. 
 Let $\chi\colon \A^\times_{\Q}/\Q^\times\to \Q_p$ be a nontrivial continuous id\`ele class
character. Then $\chi$ is (a scalar multiple of) the cyclotomic id\`ele
class character (i.e. the id\`ele class character corresponding to the cyclotomic character by
class field theory), and satisfies
$\chi_p(p)=0$.
By Proposition~\ref{P:goodheight}, the height $\hat{h}_{\L}$ with respect to $\alpha$ and
$\chi$ is a quadratic polynomial in the $f_i$ with constant term 0, 
say 
\begin{equation}\label{globalhtints}
  \hat{h}_{\L} = \sum a_{ij} f_if_j + \sum b_kf_k\,, \quad
  a_{ij},b_k\in \Q_p\,.
\end{equation}
The idea of the Quadratic Chabauty method
is to solve for the constants $a_{ij}$ and $b_k$ and to use
that for $x\in C(\Q)$, Equation~\eqref{globalhtints} implies
\begin{equation}\label{}
  \sum a_{ij} (\int_b^x\omega_i)( \int_b^x\omega_j)+
  \sum b_k\int_b^x\omega_k-\log_p\circ 1(x) = \sum_{q\ne p} \lambda_{q}(x)\chi_q(q)\,.
\end{equation}
\begin{thm}\label{T:QC}
  Suppose that $\rank\NS(J)>1$, $\rank J(\Q) =g$ and that  the
$p$-adic closure of $J(\Q)$ has finite index in $J(\Q_p)$.
  Let $L$ be a line bundle on $J$ whose class in $\NS(J)$ is nontrivial and which
  satisfies 
  $\iota^*L=\O_C$. With notation as above, 
  the function  
$$
  F\colon C(\Q_p)\to \Q_p\,;\quad x\mapsto \sum a_{ij}
  (\int_b^x\omega_i)(\int_b^x\omega_j)+
  \sum b_k\int_b^x\omega_k-\log_p\circ 1(x)
$$
is a Vologodsky function. It
takes values on $C(\Q)$ in the finite
  set $T = \{\sum_{q\ne p}l_q\cdot \chi_q(q)\}$, where $l_q$ runs through the values that the
  function 
  $\lambda_{q}$ takes on $C(\Q_q)$.
    For every $t\in T$ there are only finitely many points $x\in C(\Q_p)$ such that
  $F(x) = t$.
\end{thm}
\begin{proof}\label{P:}
   Since $\log_p$ is a Vologodsky function on $\O_C^\times$ and since $1\colon C(\Q_p)\to
   \O_C^\times(\Q_p)$ has no poles, the function $\log_p\circ 1$ is indeed a Vologodsky function on all of $C(\Q_p)$. Hence the same is
   true for $F$.
  
  If $x\in C(\Q)$, then we have
  $$F(x) = \hat{h}_L(\iota(x)) - \log_p\circ 1(x)= \sum_{q\ne p} \lambda_{q}(x)\chi_q(q)\,.$$ 
  The finiteness of $T$ follows from Lemma~\ref{L:fin_many}.

  To prove the final claim, note that a Vologodsky function
is locally analytic. If it assumes a value an infinite number of
times on a residue disc, then it must be constant on that disc, hence by the identity principle for
  Vologodsky functions~\cite[Lemma~2.5]{Bes05} it must be constant. It is therefore sufficient to prove that $F$ is not
  constant. The exact sequences~\eqref{dexact} and~\eqref{delbarexact} imply that it is sufficient to prove that 
  \begin{equation}\label{toshow}
    \overline{\partial} d(F) \ne 0.
  \end{equation}
  
  First note that  $\overline{\partial} d \log_p\circ 1 = \Curve(\log_p)$ as in the
  proof of~\cite[Proposition~4.4]{Bes05}.
But $F - \log_p\circ 1$  is a product of integrals of holomorphic forms and thus its
  $\overline{\partial} d$ resides in $\Omega^1(C)\otimes
\Omega^1(C)\subset \Omega^1(C) \otimes \hdr^1(C ) $. 
  It suffices then to prove that $\Curve(\log_p)$ is not in this subspace. As
  $\Curve(\log_p) = \iota^*\Curve(\log_L)$,
  it suffices to show  that 
  \begin{equation}\label{Curvnotinsubsp}
  \Curve(\log_L) \notin \Omega^1(J) \otimes \Omega^1(J)\,.
  \end{equation}
  To show this, we use the Hodge filtration over $\C$.
  It is well-known that over $\C$, the Chern class of $\L$ 
  is anti-invariant with respect to complex conjugation.
  Hence, since $\L$ is nontrivial in $\NS(J)$ by assumption, the Chern class of $\L$ cannot be in $F^2 \hdr^2(J)$,  because $F^2 \cap \overline{F^2}=0$. 
  This proves~\eqref{Curvnotinsubsp}, and hence
  completes the proof of the theorem.
\end{proof}

\begin{rk}\label{R:equiv1}
  Theorem~\ref{T:QC} is the analogue of \cite[Theorem~1.2]{BD18} using our
  setup. In fact, Theorem~\ref{T:localcomparison} below shows that
  the two results are equivalent in the good reduction setting (while~\cite{BD18}
  assumes good reduction throughout, we do not).
\end{rk}
\begin{rk}\label{R:moregeneralA}
 More generally, suppose that we have an isogeny quotient $\pi \colon J \rightarrow A$ such that
  $\rank A(\Q)  = \dim(A)$, the $p$-adic closure of $A(\Q)$ has finite index in $A(\Q_p)$,
    and a line bundle $\L$ on $A$ whose class in $\NS(A)$ is nontrivial and whose
  pullback to the curve is the trivial bundle. Then we can analogously use the
  canonical $p$-adic height $\hat{h}_L$ associated to $\L$ to write down a nonconstant
  Vologodsky function that takes on finitely many values on the rational
  points $C(\Q)$. Indeed, the class of $\pi^*\L$ is nontrivial in $\NS(J)$, so the same
  Hodge-theoretic argument as in Theorem~\ref{T:QC} shows that the
  difference of the global height and the local height at $p$ is a
  nonconstant Vologodsky function on $C(\Q_p)$. The assumption
  $\rank A(\Q)  = \dim(A)$ guarantees that this function, when pulled back
  to the curve, can analogously be expanded using single and double
  Coleman integrals as before, and takes values on $C(\Q)$ in a finite set
  coming from the local heights away from $p$. See also~\cite{DLF19} for an
  extension of the finiteness statements of Balakrishnan--Dogra~\cite{BD18}
  to this setting.
\end{rk}

 \begin{rk}\label{R:numberfields}
   We have chosen to restrict to the base field $\Q$ for simplicity.
   Our approach can be extended to nice curves $C/K$, where $K$ is a number field, and the
   Jacobian $J/K$ of $C$
   satisfies 
   \begin{equation}\label{rk_nf}
     \rank\NS(J)>1\;\quad\text{ and }\quad\rank J(K)+\rank \O_K^\times\le
     [K:\Q]g\,.
   \end{equation}
 This works
   exactly like the extension of
   the Quadratic Chabauty techniques from~\cite{BBM16} for integral points on hyperelliptic curves over
   $\Q$ presented in~\cite{BBBM21}. More
   precisely, the obvious analogue of Lemma~\ref{L:fin_many} continues to hold. Every
 nontrivial continuous $\Q_p$-valued id\`{e}le class character $\chi$ of $K$ gives us a nontrivial
   Vologodsky function $F^{\chi}$, inducing a relation satisfied by $C(K)$ as in
   Theorem~\ref{T:QC}.  We need
   $[K:\Q]$ relations to have any chance to cut out a finite set, which leads to the
   requirement~\eqref{rk_nf}. As in~\cite{BBBM21}, we cannot hope that this rank condition suffices
   to guarantee finiteness, for instance when $C$ is the base-change of a curve over $\Q$
   that does not satisfy the rank conditions of this section.
   \end{rk}

   \subsection{Quadratic Chabauty for rank $>$ genus}\label{S:nonconstant}
  In \cite[\S5.3]{BD18},  Balakrishnan--Dogra generalize their Theorem~1.2 to
  the setting
  $$\rank J(\Q)< g+\rank \NS(J)- 1\,.$$
 Our Theorem~\ref{T:QC} may also be generalized, without any assumption on
 the reduction, to this case, as we now explain.
  
  Recall that we have fixed a purely mixed
  curvature form on $\P_p$ and a nontrivial continuous id\`{e}le class
  character on $\Q$.  
  This gives a canonical $p$-adic height function associated with any line bundle on $J$, which is
  linear in the line bundle. The association  
  \begin{equation*}
    V'\colonequals \ker(\iota^* \colon \NS(J) \rightarrow \NS(C)) \mapsto \Q_p^{J(\Q)}
  \end{equation*}
  given by sending a class in the kernel to the canonical $p$-adic height associated with the unique line
  bundle in its class that pulls back to the trivial line bundle on $C$, is
  linear by Corollary~\ref{C:ht_bund_lin} and therefore
we may tensor it with $\Q_p$. By also evaluating on $J(\Q)$ the integrals of elements of
$H^0(C_p,\Omega^1)$ we get a map
  \begin{equation}\label{eq:V}
P \colon J(\Q) \rightarrow V \colonequals H^0(C_p,\Omega^1)^{\vee} \oplus
    { V'}^{\vee} \otimes {\Q_p} \cong \mathbb{Q}_p^{g+\nu}\,;\quad \nu =
    \rank \NS(J)-1\,,
  \end{equation}
  where the maps $J(\Q) \rightarrow H^0(C_p,\Omega^1)^{\vee}$ is linear, and $J(\Q) \rightarrow { V'}^{\vee} \otimes {\Q_p}$ is quadratic, since by Proposition~\ref{P:goodheight}, we know that the canonical  $p$-adic height $\hat{h}_{\L}$ associated to any line bundle $\L$
  is a   quadratic function. Therefore, the map $P$ extends to a polynomial map  
$P \colon J(\Q) \otimes \Q_p \rightarrow V$. By our assumption that $$\rank
J(\mathbb{Q}) = \dim J(\mathbb{Q}) \otimes \mathbb{Q}_p <
g+\rank\NS(J)-1\,,$$ 
the image of $P$ is a proper Zariski closed subvariety. Hence we obtain:
  \begin{lemma}\label{L:noncons}
   There is a nonconstant polynomial function $R$ on $V$ such that $R\circ P=0$ on $J(\Q)$.
  \end{lemma}

To get a result about rational points on $C$ we combine this with the following.
   \begin{thm}\label{T:nonconstant}
   For any nonconstant polynomial function $R$ on $V$ there is a finite collection
   $\{f_t\}$ of nonzero Vologodsky functions on  $C(\Q_p)$ such that for any $x\in
   C(\Q)$ there exists some $t$ such that 
   \begin{equation*}
     R(P(\iota(x)))= f_t(x)
   \end{equation*}
  \end{thm}
  Since a nonzero Vologodsky function has finitely many zeroes in $C(\Q_p)$, we get the following corollary.
\begin{cor}\label{C:nonconstantapplication} Assume that $\rank J(\Q) < g + \rank \NS(J)-1$. 
 Then there is a finite collection
   $\{f_t\}$ of nonzero Vologodsky functions on  $C(\Q_p)$ such that
   \begin{equation*}
     C(\Q)\subset \cup_t \{x\in C(\Q_p),\; f_t(x)=0 \}.
   \end{equation*}
   Furthermore each $f_t$ has finitely many zeroes in $C(\Q_p)$.
\end{cor}
\begin{proof}[{Proof of Theorem~\ref{T:nonconstant}}]
We pick a basis $\omega_1,\ldots,\omega_g$ for $H^0(C_p,\Omega^1)$  and
we let $l_1,\ldots,l_g \in H^0(C_p,\Omega^1)^\vee$ the corresponding dual basis (note that the restriction of each $l_i$ to the points of $C(\Q)$ may be computed by a Vologodsky integral of the corresponding $\omega_i$). 
We also pick line
bundles $\L_1,\ldots,\L_\nu$ on $J$ whose classes form a basis of $V'$ and which
all satisfy $\iota^*\L_j = \O_C$. These give
coordinates on $V$ and we can write any polynomial function $R$ on $V$ as 
\begin{equation*}
  R= E(l_1,\ldots,l_g,\hat{h}_{\L_1},\ldots,\hat{h}_{\L_\nu})
\end{equation*}
for some polynomial $E$ in $g+\nu$ variables. 
For $x\in C(\Q)$ we get 
\begin{equation*}
  R(P(\iota(x))) = E(l_1(x),\ldots,l_g(x),
  \hat{h}_{\L_1}(\iota(x)),\ldots,\hat{h}_{\L_\nu}(\iota(x))).
\end{equation*}
For each index $j$ we have  $\hat{h}_{\L_j}(\iota(x))= h_p^j(x) + \sum_{q \neq p}
h_q^j(x)$,  where $h_p^j = \log^j_p\circ 1$ and $h_q^j =
v_q\circ 1$; here $1$ is a nowhere vanishing section on $\O_C$ and the
functions $\log^j_p$ and $v_q$ are as in~\eqref{vq}. 
By Lemma~\ref{L:fin_many}, 
the functions $h_q^j$ take values on $C(\Q_q)$ in a finite set $T_q^j$,
and we have $T_q^j\colonequals \{0\}$ when $q$ is a prime of potentially
  good reduction. We set $T_j \colonequals 
\{\sum_{q\ne p} t_q\,:\, t_q \in T_q^j\}$, so that
the function $h_f^j\colonequals \sum_{q \neq p} h_q^j$ takes values in
the finite set $T_j$ on  $C(\Q)$.
Let $T\colonequals T_1\times \cdots\times T_\nu$ and for $t \in T$ let 
$$f_t (x)\colonequals E(l_1(x),\ldots,l_g(x),h_p^1(x)+t _1,\ldots,h_p^\nu(x)+t _\nu)$$ 
for $x\in C(\Q_p)$. Since the $t_i$ are constants, we may view $E$ as a
  nonzero polynomial evaluated in $l_1(x),\ldots,l_g(x)$,
  $h_p^1(x),\ldots,h_p^\nu(x)$. As Vologodsky functions are closed under addition and multiplication
    it follows that $f_t $ is a Vologodsky function for each
    $t \in T$.  For each $x\in C(\Q)$ we have 
   \begin{equation*}
     R(P(\iota(x)))= f_t (x)\,,\quad\text{where}\;t
     =(h_f^1(x),\ldots,h_f^\nu(x))\,.
   \end{equation*}
   For simplicity, write $h_j$ for $h_p^j$. It remains to show that all the functions
       $f_t $ are nonzero. This follows from Theorem~\ref{T:nonzeropoly} below since $E$ is a nonzero polynomial.
  \end{proof}
  
  \begin{thm}\label{T:nonzeropoly}
  Let $l_1,\ldots,l_g,h_1,\ldots,h_\nu$ be as in the proof of Theorem~\ref{T:nonconstant}. 
   Suppose that we have a polynomial $E$ in $g+\nu$ variables
  such that  $E(l_1,\ldots,l_g,h_1,\ldots,h_\nu)$ 
  is identically $0$ as a
  Vologodsky function on $C(\Q_p)$.  Then $E$ is identically zero.
  \end{thm}

For proving this theorem, we will use the results of \S\ref{S:higherdelbar}.
We will also need a few preparatory lemmas.
\begin{lemma}
Let $H$ be a finite dimensional vector space over a field $K$ of characteristic $0$ with a decomposition  
$H=F\oplus W$. Consider the map 
\begin{equation*}
  S\colon \Symm^m(F) \otimes \Symm^k(F\otimes W) \to T(H)  
\end{equation*}
whose target $T(H)$ is the shuffle algebra on $H$ (i.e. the tensor algebra
  $T(H)$ on $H$, with (commutative) multiplication coming from the shuffle
  product \eqref{shuffle}) defined as follows. The map is induced by the
  embeddings $F\otimes W
\to H\otimes H$, $F\to H$, and extended to the symmetric powers $\Symm^m(F)$ and $\Symm^k(F\otimes W)$  using the shuffle product~\eqref{shuffle} on the commutative algebra $T(H)$. Then $S$ is injective.
\end{lemma}
\begin{proof}
Since the shuffle product is defined over $\Q$, it is enough to prove the result assuming $K = \Q$. Let $n \colonequals \dim F$, and $l \colonequals \dim W$.  
Any choice of one-forms  $\{\omega_i \ | \  i=1,\ldots,n\}$ and $\{\eta_j \ | \ j=1,\ldots,l\}$ on the interval $[0,1]$ gives a well defined map $T(H) \otimes \R \to \R$ by sending words to the appropriate iterated integral from $0$ to $1$. (Note that one could replace $[0,1]$ with other intervals or other spaces.) The shuffle rule for the product of iterated integrals precisely means that this map is a
homomorphism. Suppose that $S$ is not injective. Then a nontrivial element in the kernel of $S$ gives a nonzero polynomial $u$ in the variables
$\{ y_i \ | \ i=1,\ldots, n \} \cup \{ x_{ij} \ | \  i=1,\ldots,n, j=1,\ldots,l\}$, homogeneous of degree $m$ in the $y$'s and of degree $k$ in the $x$'s, such that 
\begin{equation}\label{E:polyuinker}
  u(\int_0^1 \omega_1,\ldots, \int_0^1 \omega_n, \int_0^1 \omega_1 \circ
  \eta_1,\ldots,
  \int_0^1 \omega_n\circ \eta_l )=0
\end{equation}
for any choice of $\omega$'s and $\eta$'s. We can choose arbitrary differentiable functions
$f_i$ and $g_j$ on $[0,1]$ and set $\omega_i = f_i dx$ and $\eta_j = d
  g_j$. Note that if we choose the $g_j$'s to vanish at $0$, then $\int_0^1 \omega_i \circ
  \eta_j = \int_{0}^1 f_i g_j dx$, which is the standard inner product of
  the differentiable functions $f_i$ and $g_j$ on the interval $[0,1]$. The
  statement about $u$ in~\eqref{E:polyuinker} now translates into $u$ vanishing for any substitution of $\int_0^1 f_i dx$ into the variables $y_i$ and the inner product matrix of the functions $f_i$ and $g_j$ on the interval $[0,1]$ into the variables $x_{ij}$. This is true for any choice of $f$'s and $g$'s with the $g$'s vanishing at $0$. By first picking the $f$'s appropriately and then the $g$'s, we can see that we may obtain any set of values in the variables. Thus, the polynomial $u$ must vanish, which contradicts the assumption that $S$ is not injective. 
\end{proof}
We have a decreasing filtration on $T(H)$ such that $F^r$ is generated by tensors that contain
at least $r$ elements from $F$. 
\begin{cor}\label{modfinj}
  The map $S$ modulo $F^{k+m+1}$ is still injective.
\end{cor}
\begin{proof}
The image of $S$ consists of tensors of degree $2k+m $ that contain at least $k$ entries from
$W$. Thus, it is disjoint from $F^{k+m+1}$.
\end{proof}

We will need one more ingredient from the theory of generalized delbar operators $D_k$ and iterated integrals of Theorem~\ref{hdelbth} for the proof of Theorem~\ref{T:nonzeropoly}. 
The $l_i$ are $1 $-iterated integrals and $D_1 l_i= \omega_i$. On the other hand, Corollary~\ref{C:Dkonkitint} shows  that the $dh_j$ are one-iterated differential forms, so, by \eqref{hd4} of Theorem~\ref{hdelbth} the $h_j$ are
$2$-iterated integrals and we may apply the operator $D_2$ to $h_j$ and
$D_2 h_j= \delbr d h_j $
is simply the curvature of $\iota^{\ast}\log_{\L_j,p}$, an element of
$\Omega^1(C) \otimes \hdr^1(C)\subset \hdr^1(C)^{\otimes 2}$. 
Here $\log_{\L_j,p}$ is the canonical log function on $L_j\otimes \Q_p$
with respect to the curvature $\alpha$.
\begin{lemma}\label{dtindep}
  The elements $D_2 h_j$ are independent modulo $\Omega^1(C)^{\otimes 2}$.
\end{lemma}
\begin{proof}
It suffices to show this after applying the map 
\begin{equation*}
  \Omega^1(C) \otimes \hdr^1(C) / \Omega^1(C)^{\otimes 2}
  \xrightarrow{(\iota^{\ast})^{-1}} 
  \Omega^1(J) \otimes \hdr^1(J) / \Omega^1(J)^{\otimes 2}\xrightarrow{\cup}
  \hdr^2 / F^2.
\end{equation*}
The image of $D_2 h_j$ under this map is the Chern class of $\L_j$ modulo $F^2$. As
the Chern classes of line bundles are rational and $C$ is defined over $\Q$, it suffices to prove the linear independence of Chern classes of $\L_j$ modulo $F^2$ over $\Q$. 
Clearing any denominators in a nontrivial dependence relation between the
  images of $D_2h_j$ gives rise to a nontrivial line bundle whose Chern
  class is in $F^2$. But such line bundles do not exist, by the same
  argument we used at the end of the proof of Theorem~\ref{T:QC}.  
\end{proof}

\begin{proof}[Proof of Theorem~\ref{T:nonzeropoly}]
Note that by the identity principle, a Vologodsky function is identically $0$ as soon as this function is identically $0$ in a neighborhood of one point.
We will first show that the 
$h_j$ do not show up in $E$. We define two different
  degrees $\deg_1$ and $\deg_2$ on the monomials $M =
  M(l_1,\ldots,l_g,h_1,\ldots,h_\nu)$ as follows. Let
  $\deg_1$ satisfy $\deg_1(h_j) =2$ and
  $\deg_1(l_i) =1$ and let $\deg_2$ satisfy $\deg_2(h_j) = \deg_2(l_i)=1$.
  Clearly, a monomial $M=\prod_{i=1}^g l_i^{n_i}\prod_{j=1}^{\nu}h^{m_j}_j$
  satisfies $\deg_1(M)=\deg_2(M)$ if and only of all $m_j$ are zero.
  
  Let us assume that $E$ is not a polynomial in the $l_i$'s only. By the
  above there is a non-empty collection of
monomials that have maximal $\deg_1$ 
 which we use to form a polynomial $E_1$, and of these, a non empty
  collection of monomials that have  minimal $\deg_2$, which we use to form
  a polynomial $G$. We may assume that the maximal $\deg_1$ is $2k+m$ and
  the minimal $\deg_2$ is $k+m$,
respectively, so that $G$ is homogeneous of degree $k>0$ in the $h_j$'s and of degree
$m$ in the $l_i$'s.

Let $H=\hdr^1(C)$ and let $F\subset H$ be the space of holomorphic differentials. Pick a
complementary subspace $W$ to $F$ in $H$. The $l_i$ are $1 $-iterated integrals
and $D_1 l_i\in F$ while the $h_j$ are $2$-iterated integrals and we have $D_2 h_j
\in F\otimes H$. By \eqref{hd1} of Theorem~\ref{hdelbth} all monomials $M$ in $E(l_1,\ldots,l_g,
  h_1,\ldots,h_\nu)$ give $s=\deg_1(M)\le 2k+m$-iterated integrals.
  We apply the operator $D_{2k+m}$ to $E(l_1,\ldots,l_g,
h_1,\ldots,h_\nu)$ 
and take a further quotient to land in $H^{\otimes 2k+m}/F^{k+m+1}$. By
assumption we get $0$. By \eqref{hd3} of Theorem~\ref{hdelbth} the operator
  $D_{2k+m}$ kills all monomials $M$ such that 
  $\deg_1(M)<2k+m$ and by \eqref{hd2} of Theorem~\ref{hdelbth} we see
that 
$$D_{2k+m}E(l_1,\ldots,l_g, h_1,\ldots,h_\nu) = E_1(D_1 l_1,\ldots,D_1 l_g,
  D_2(h_1),\ldots,D_2(h_\nu))$$
where the computation of $E_1$ uses the shuffle product~\eqref{shuffle}. 

Since this product is just a sum of
products of possibly permuted coordinates, we easily see that all monomials
  $M$ in $E_1$ 
  such that $\deg_2(M)>k+m$ land in $F^{k+m+1}$  when applied to the $D_1 l_i$'s and $D_2 h_j$'s and therefore
die in the quotient. Putting it all together we get
\begin{equation*}
  0 \equiv D_{2k+m} G(l_1,\ldots,l_g, h_1,\ldots,h_\nu) 
  \equiv G(D_1 l_1,\ldots,D_1 l_g, D_2(h_1),\ldots,D_2(h_\nu))\pmod{F^{k+m+1}}\;.
\end{equation*}
The $D_2(h_j)$ belong to $F\otimes H$. Replacing them with their projections $z_j$ in $F\otimes W$ along $F\otimes F$ changes nothing 
once we quotient out by $F^{k+m+1}$, so we finally get that 
$$ G(D_1 l_1,\ldots,D_1 l_g, z_1,\ldots,z_\nu)\equiv 0 \pmod{F^{k+m+1}}\,.$$ But Corollary~\ref{modfinj} plus the fact
that the $D_1 l_i$ are independent and the $z_j$ are independent (because the
$D_2(h_j)$ are
independent modulo $F\otimes F$ by Lemma~\ref{dtindep}) tells us that 
$G=0$. This gives a contradiction. Thus, the polynomial $E$ contains only $l_i$'s.

  Now we assume that $E$ is a nonzero polynomial such that
$E(l_1,\ldots,l_g)=0$. Let $E_1$ be the homogeneous component with the
largest degree of $E$, and let this degree be $m$. Applying $D_m$ to the
equality we obtain as before $E_1(\omega_1,\ldots,\omega_g)=0$ in
$T\left(\hdr^1(C) \right) $. We now claim that the induced map $\Symm^{m}\hdr^1(C)
\to T(\hdr^1(C))$ is injective. One can verify using induction and the definition of the shuffle product that this induced map sends 
\[ v_1 \otimes v_2 \otimes \cdots \otimes v_n \mapsto \sum_{\sigma \in S_n} v_{\sigma(1)} \otimes v_{\sigma(2)} \otimes \cdots \otimes v_{\sigma(n)}, \] 
in other words to the associated symmetric tensor. This map is clearly injective, and we conclude that $E_1=0$ which gives a contradiction.
\end{proof}

\begin{rk}\label{R:construct_ft}
  While  Corollary~\ref{C:nonconstantapplication} only guarantees the
  existence of the functions $f_t$, the proof of
  Theorem~\ref{T:nonconstant} shows how these functions may be constructed.
\end{rk}

\begin{rk}\label{R:higher_BD}
Following comments by the referees, we want to clarify the relation of
Theorem~\ref{T:nonconstant} and Theorem~\ref{T:nonzeropoly}
with Zariski density arguments. 
  Balakrishnan--Dogra consider the $p$-adic Albanese map 
  \begin{equation}\label{E:padicAJ}
  C(\Q_p) \to \selmervariety\,,
\end{equation} 
where $U_Z$ is a certain nonabelian quotient of the $\Q_p$-\'etale unipotent fundamental group of
  $C_{\overline{\Q}}$ depending on a cycle $Z$ as in our setup and
  $\selmervariety$ is Kim's local Selmer variety (see~\cite{Kim05, Kim09}).

Zariski density is the statement that the
  the image of each residue disc under~\eqref{E:padicAJ} is Zariski dense in
the local Selmer variety $\selmervariety$. As a
consequence, a nonzero polynomial function on $\selmervariety$  does not
  vanish identically on the image of the residue disc.
Therefore, if there is a nonzero polynomial function on $\selmervariety$
that vanishes on the image of the global Selmer variety $\globalselmer$,
then it produces by pullback a nonzero function on this residue disc
vanishing on all rational points. Such a function is Coleman analytic, hence we get  finiteness of rational points on that disc by 
Coleman's uniqueness principle
  (see~\cite[Lemma~2.4.4]{CdS88} and
\cite[Corollary~4.13]{Bes02}), which says that a nonzero Coleman function cannot vanish identically on a residue disc.
Such a function exists when
the dimension of $\globalselmer$ is smaller than the dimension
of $\selmervariety$ by the commutativity of the following diagram:
 \begin{equation}\label{}
  \xymatrix{
   C(\Q)\ar[d] \ar[r]& \globalselmer\ar[d] \\
      C(\Q_p) \ar[r] & \selmervariety. \\
  }
    \end{equation}
We note that Zariski density of the image
  of~\eqref{E:padicAJ} can be deduced by combining
the injectivity of the map from polynomial functions on $\selmervariety$ to Coleman functions
on $C$ given in~\cite[Proposition~1.6.9]{Bes10a} with 
Coleman's uniqueness principle.

But Zariski density on its own does not tell us that a specific Coleman
function constructed to vanish on all rational points will be nonzero -- knowing that such a function is nonzero is necessary for constraining the set of rational points to a finite set. Since our Coleman function is a polynomial in logarithms and local height
functions, we need an algebraic independence statement to show that no
nontrivial polynomial may vanish when evaluated on these functions. This
is the content of Theorem~\ref{T:nonzeropoly}.

Viewing these $g+\nu$ Coleman functions as coordinates on $\selmervariety$, an
alternative approach is to show that the induced map to
$\mathbb{A}^{g+\nu}$ is surjective (or at least has dense image). This is
stated in~\cite[Remark 5.7]{BD18}. The proof is not given but Dogra indicated (personal
communication) that it follows in a straightforward way
from the isomorphism $\selmervariety \cong
  H^1_f(G_p,T_p(J)) \oplus \Q_p$ (see~\cite[Lemma~3.10]{BD21}).

  Via Theorem~\ref{T:localcomparison} and Corollary~\ref{R:localcomparisonp} below,
  our version of Quadratic Chabauty when $\rank J(\Q) < g + \rank
  \NS(J)-1$ is  essentially equivalent to that developed by
  Balakrishnan--Dogra in~\cite[\S5.3]{BD18}. In particular, we may
  recover~\cite[Proposition~5.9]{BD18}.
\end{rk}

\subsection{Computing rational points using Quadratic Chabauty}
Let $C/\Q$ be a nice curve satisfying the conditions of Theorem~\ref{T:QC}. 
In order to compute the rational points on $C$, 
we need to be able to do the following (to suitable finite precision):
  \begin{enumerate}[(I)]
    \item\label{lambdap} compute $\lambda_p(x)$ for $x\in C(\Q_p)$
      (see~\S\ref{subsec:lambdap});
    \item\label{lambdaq} find all possible values of $\lambda_q(x)$ for bad primes $q\ne p$  and $x\in
      C(\Q_q)$ (see~\S\ref{subsec:lambdaq});
    \item\label{htsolve} solve for the coefficients $a_{ij}$ and $b_k$ (see~\S\ref{subsec:solve})
    \item\label{findsols} compute the set 
     $\mathcal{Z} \colonequals \{z \in C(\Q_p)\,:\, F(z)\in T\}\,;$
    \item\label{ratinpadic} identify $C(\Q)$ inside $\mathcal{Z}$.
  \end{enumerate}
  These tasks are the same as in~\cite{BD18, BDMTV19}.
  The final two steps are standard: Step~\eqref{findsols} is possible in practice using the
  Weierstrass preparation theorem. If $\mathrm{rk}\mathrm{NS}(J)>2$, we can repeat
  Steps~\eqref{lambdap}--\eqref{findsols} for another line bundle $L'$ on $J$ such that
  $\iota^*L'\cong \O_C$ and such that the classes of $L$ and $L'$ in $\mathrm{NS}(J)$ are
  independent. We expect that the sets $\mathcal{Z}$ for $L$ and $L'$ will usually only have
  $C(\Q)$ as common roots, unless there is a geometric reason for further
  common roots (see~\cite{Bia20} for a discussion of geometric reasons for
  multiple roots in the context of Quadratic Chabauty). Alternatively, we can combine $\mathcal{Z}$ with information at
  primes $q\ne p$ via the Mordell--Weil sieve~\cite{BruinStoll, BBM17}. 
  
  We briefly discuss Steps~\eqref{lambdap},~\eqref{lambdaq} and~\eqref{htsolve} below, referring
  to future work for more details.

\subsubsection{Computing $\lambda_p$ }\label{subsec:lambdap}
For Step~\eqref{lambdap}, we need to expand the function $\lambda_p = \log_p\circ 1$ into
a convergent power series on every residue disk of $C(\Q_p)$. 
Recall that $\log_p=\iota^*\log_{\L,p}$ has curvature $\iota^*\alpha_\L$, where 
$$\alpha_\L = \frac{1}{2}\alpha_{\L^+} = \frac{1}{2}(\id\times \phi_\L)^*\alpha\,.$$
However, $\alpha_\L$ does not determine $\log_{\L,p}$; in fact, it is not straightforward
to construct \textit{any} good log function on $L_p$. 
Fortunately, this is not necessary. 
Let ${\log'_p}$  be \textit{some} log function on $\O_C$ with 
curvature $\iota^*\alpha_\L$ and set
$\lambda'_p \colonequals \log'_p\circ 1$.
By Proposition~\ref{P:excurve},
we have $\log'_p = \log_p +\int \theta$ for some unknown
holomorphic form $\theta\in \Omega^1({C_p})$ which induces via $\iota^*$ a linear form 
$$\ell\colon J(\Q_p)\to \Q_p\,\quad a \mapsto \int_0^a\theta\,.$$ 
We set $$h'\colonequals \hat{h}_{\L} + \ell\,.$$
Then Theorem~\ref{T:QC} remains true if we replace $\hat{h}_{\L}$ by $h'$ and $\lambda_p$
by $\lambda'_p$, without changing $\lambda_q$ for $q\ne p$. Hence we can work with
$\log'_p$ in place of $\log_p$ and there is no need to compute $\theta$.
An explicit construction of a log function
$\log'_p$ on $\O_C$ with curvature form $i^*\alpha_\L$ as an iterated integral
is given in \S\ref{subsec:logcurves}. It thus remains to compute the curvature form
$\alpha_\L$; we will describe in future work how this can be done in terms of a
complementary subspace $W$ and how this can be used to compute $\log'_p$ in practice. 

\subsubsection{Computing (all possible values of) $\lambda_{q}$}\label{subsec:lambdaq}

Let $q\ne p$ be a prime of bad reduction for $C$. We would like to compute a finite set
$T_q\subset \Q$ such that $\lambda_q(C(\Q_q))\subset T_q$.
We first strengthen Lemma~\ref{L:fin_many}. 

\begin{prop}\label{P:factor}
Let $K/\Q_q$ be a finite extension such that $C$ has a semistable regular model $\mathcal{C}$ over $\O_K$. Then $\lambda_q(P)$ only depends on the
  component of the special fiber $\mathcal{C}_q$ that $P$ reduces to.  
\end{prop}
\begin{proof}
 Suppose that $x$ and $y$ are points in $C(K)$
  such that $\iota(x)$ and $\iota(y)$ reduce to the same component
of the special fiber $\mathcal{J}_q$ of the N\'eron model of $J$ over $\O_K$.  
It suffices to show that we have
  $\lambda_q(x) = \lambda_q(y)$.

  Let $s$ be a meromorphic section of $\L$ such that $\iota(x)$ and $\iota(y)$ are not in
  the support of $D\colonequals\div(s)$. By~\cite[Theorem~11.5.1]{Lan83}, 
  we have
  \begin{equation}\label{langneron}
    v_{L,q}(s(a)) = i(D,a) + \gamma_s
  \end{equation}
for all $a\in J(K)\setminus \supp(D)$,
  where $i(D,a)$ is the intersection multiplicity of the extensions of $D$
  and $a$ to
  $\mathcal{J}$ defined before~\cite[Theorem~11.5.1]{Lan83}
  and where $\gamma_s$ is a constant that only depends on the component of $\mathcal{J}_q$ that $a$
  reduces to.  
  Since $\iota^*\L = \O_C$, the section $s$ pulls back to a rational function $f$ on $C$.
  We then have
      \begin{align*}
        \lambda_q(x) - \lambda_q(y) &= v_{\L, q}(s(\iota(x)) - v_{\L, q}(s(\iota(y)) -(\ord_q(f(x))-
        \ord_q(f(y)))\\
        &= i(D, \iota(x))- i(D, \iota(y))-(\ord_q(f(x))- \ord_q(f(y)))\,.
      \end{align*}
      To show that this vanishes, we use the fact that $\iota$ induces a morphism 
      $$
  \bar{\iota}\colon \mathcal{C}^{sm} \to \mathcal{J}\,,
      $$
      where $\mathcal{C}^{sm}$ is the smooth locus of $\mathcal{C}$. 
      Denoting the closure of $D$ on $\mathcal{J}$ by $\mathcal{D}$ and the closure of
      $D'\colonequals\div(f)$ on $\mathcal{C}$ by $\mathcal{D}'$, we then have 
      $\bar{\iota}^*\mathcal{D} = \mathcal{D}' + V'_s$ for some vertical divisor $V'_s$ on
      $\mathcal{C}$. We extend $V'_s$ to a vertical divisor $V_s$ on $\mathcal{C}$.
      If $P\in C(K)$ with closure $\mathcal{P}$ on $\mathcal{C}$, then the
      projection formula implies
  $$i(D, \iota(P)) = (\mathcal{D}\,.\,\bar{\iota}(P)) = (\mathcal{D'}\,.\,\mathcal{P}) +
  (V_s\,.\,\mathcal{P}) = \ord_q(f(P)) + (V_s\,.\,\mathcal{P})\,. $$
Since  the intersection multiplicity $(V_s\,.\,\mathcal{P})$ only depends on the component
  of $\mathcal{C}_q$ that
$P$ reduces to, we deduce that $\lambda_q(x) = \lambda_q(y)$.
\end{proof}

By the proof of Proposition~\ref{P:factor}, it suffices to find, for every component
of $\mathcal{J}_q$, one point $x\in J(K)$ reducing to it and to
compute $v_{\L,q}(s(x))$ for some section $s$ of $\L$ such that $x\notin \div(s)$.
However, this is a difficult problem in general. Even for $\L=\Theta$, it is only known
how to compute the canonical valuation explicitly for hyperelliptic curves of genus~$\le
3$ (see~\cite{Flynn-Smart, Stoll:H2, MS16, Sto17}).

In Section~\ref{subsec:VarVol} we give a new construction of the canonical
valuation using iterated Vologodsky integrals. In future work, we will show
how this may be used to compute $\lambda_q$ in practice, by pulling back
the canonical valuation $v_{\L,q}$ to the curve and working with iterated
Vologodsky integrals on the curve.

\subsubsection{Solving for the height}\label{subsec:solve}

We briefly discuss possible approaches to the computation of the constants $a_{ij}$ and $b_k$ in Theorem~\ref{T:QC}.
The line bundle $\L$ induces an endomorphism $E_{\L} = \phi_{\Theta}^{-1}\circ\phi_\L$ on
$J$, and hence on $\hdr^1(C)\cong \hdr^1(J)$.
We set $J_p \colonequals J(\Q)\otimes \Q_p$. If there are enough points in $C(\Q)$ such
that their images generate $(J_p\otimes J_p)\oplus J_p$ under the embedding $(\iota\otimes E_\L\circ
\iota + [\L^-], \iota)$, then we can compute the constants $a_{ij}$ and $b_k$ by
computing $\hat{h}_\L$ (or $h'$) and the functions $f_i$ and $f_if_j$ in these images. 

Alternatively, note that 
the bilinear pairing associated to $\hat{h}_{\L}$ and
$h'$ is the same, say 
$$
B(x,y) \colonequals \frac{1}{2}(h'(x+y) -h'(x)-h'(y))=
\frac{1}{2}(\hat{h}_{\L}(x+y) -\hat{h}_{\L}(x)-\hat{h}_{\L}(y))\,.
$$
We want to write $B$ in terms of the basis 
of the symmetric part of $(J(\Q)\otimes J(\Q)\otimes\Q_p)^*$ given by
$$(g_{ij})_{0\le
i\le j<g}\,,\quad g_{ij}(x,y)= \frac{1}{2}(f_i(x)f_j(y)+f_j(x)f_i(y)))\,,$$
by evaluating $B$ and the $g_{ij}$ in sufficiently many points. 
We conclude that
\begin{equation}\label{}
  B(x,y) =  \hat{h}_{\Theta}(x, E_{\L}(y))\,.
\end{equation}
If we evaluate $\hat{h}_{\Theta}$ in sufficiently many points, we get an expression 
$$\hat{h}_{\Theta}(x,y) = \sum_{i,j}c_{ij}g_{ij}(x,y)$$ 
and we obtain the constants $a_{ij}$ by evaluating, using
$f_i(E_\L(y)) =\int_0^{E_\L(y)} = \int_0^y E_\L^*(\omega_i)$. 
Since $B(x,x) = \hat{h}_{\L^+}(x)$ is the quadratic term of both $\hat{h}_{\L}$ and $h'$,
the linear part of $h'$ is
$h'(x) - B(x,x)$, and we can compute it by evaluating
in sufficiently many points. This gives us the constants $b_k$.

\section{Comparison with Balakrishnan and Dogra's approach to Quadratic Chabauty}\label{sec:BDcomparison}

In this section we clarify the relation between our Quadratic Chabauty construction and
the original one of Balakrishnan and Dogra~\cite{BD18}. 
Using the equivalence between the height pairings of Nekov\'a\v{r} and of
Coleman--Gross, their
local contributions are of the form $$h_{v}^{CG} (z-b,D(b,z)),$$ where $D(b,z)$ is a divisor on $C$
constructed using the line bundle $\L$ equipped with a section.
It turns out that our contribution and theirs are the same up to a constant
multiple (see Theorem~\ref{T:localcomparison}).

We first spell out the construction of the divisor $D(b,z) \in \Div(C)$ starting from a
divisor in the class of the line bundle $\L \in \Pic(J)$. For this, we modify the construction of the divisor $D_Z(b)$ in \cite[Section~2.2]{DLF19}. The divisor $D_Z(b)$ of \cite{DLF19} corresponds to the diagonal cycle $D(b,b)$ in our notation.

\subsection{Construction of the divisor $D(b,z)$}\label{S:construction}
Let $C$ be a nice curve of
genus $g >1$ over a field $K$ such that $C(K)\ne \varnothing$. Fix a base point $b \in
C(K)$. Let $z \in C(\overline{K})$. 
We then have natural maps
$i_{1,b},i_{2,z},\Delta\colon C \to C\times C$, defined as follows.
We denote by  $\Delta$ the
diagonal embedding. 
Let $i_{1,b} \colon C \rightarrow C \times C$ be the map defined by $i_{1,b}(x)
\colonequals (x,b)$. Similarly, we define $i_{2,z} \colon C
\rightarrow C \times C$ by $i_{2,z}(x) = (z,x)$. 

Let $\iota = \iota_b \colon C \rightarrow J$ be the Abel--Jacobi map with respect to the base point
$b$ and let $m \colon J \times J \rightarrow J$ denote the group law on the Jacobian. We
define $\iota^{(2)} \colon C \times C \rightarrow J$ to be the composition $m \circ
(\iota,\iota)$.
Let $\phi_z \colon \Div(C \times C) \rightarrow \Div(C)$ be the map $\phi_z \colonequals \Delta^* - i_{1,b}^*-i_{2,z}^*$, where $\Delta^* \colon \Div(C \times C) \rightarrow \Div(C)$ is the pullback map induced by $\Delta \colon C \rightarrow C \times C$, etc.

\begin{defn}\label{D:divcon}
 Let $\L \in \Pic(J)$ and let $s$ be a section of $\L$. Let $Z \colonequals (\iota^{(2)})^*(\L, s) \in \Div(C \times C)$. Define $D(b,z)$ to be the divisor
  corresponding to the image of $(\L,s)$ under the composition
 \[ \Div(J) \xrightarrow{(\iota^{(2)})^*} \Div(C \times C) \xrightarrow{\phi_z} \Div(C). \]
   We will denote the composite map $\theta_{C,b,z}$. 
\end{defn}

Balakrishnan and Dogra have an intersection-theoretic condition on the
divisor $Z$ as above \cite[Definition~6.2, \S6.3]{BD18}. 
They use this condition to justify that certain mixed extensions of Galois
representations constructed out of $D$ are isomorphic (see the end of
\cite[\S6.3]{BD18} for more details).
We now prove that their condition is equivalent to the condition that $\deg(\iota^*\L) =
0$, which is key for running the approach to Quadratic Chabauty in this paper. 
\begin{lemma} 
 The condition $Z.(\Delta-C\times P_1-P_2\times C)=0\quad \text{for all }P_1,P_2\in C$ is
  equivalent to the condition that $\deg(\iota^* \L) = 0$. 
\end{lemma}
\begin{proof}
We may assume that $P_1 = P_2 = b$ without any loss of generality. Using the projection formula
  in the first line, $Z = (\iota^{(2)})^*(\L, s)$ in the second line, the identities 
  \[ \iota^{(2)} \circ i_\Delta = [2] \circ \iota, \quad\quad \iota^{(2)} \circ i_{1,b} =
  \iota^{(2)} \circ i_{2,b} = \iota \] 
in the third line, and the isomorphisms from Equations~\eqref{E:oddeven}, ~\eqref{m-iso1}
  and ~\eqref{m-iso3} 
\[ [2]^*(\L^{\otimes 2}) = [2]^*(\L^{+}) \otimes [2]^*(\L^{-}) = (\L^+)^{\otimes 4} \otimes {\L^-}^{\otimes 2} = (\L^+)^{\otimes 2} \otimes {\L}^{\otimes 2}\] 
in the fourth line, we get
\begin{align*}
 Z.(\Delta-C\times \{b\}-\{b\}\times C) &= \deg({\Delta}^*(Z))-\deg( i_{1,b}^*(Z) ) -\deg( i_{2,b}^*(Z) ) \\
  &= \deg(({\Delta}^* {\iota^{(2)}}^*\L) \otimes (i_{1,b}^* {\iota^{(2)}}^*\L)^{-1}
  \otimes (i_{2,b}^* {\iota^{(2)}}^*\L)^{-1}) \\
  &= \deg(\iota^*([2]^*\L \otimes \L^{\otimes -2})) \\
  &= \frac{\deg(\iota^*\L^+)}{2}.
\end{align*}
We are allowed to divide by $2$, because $\NS(C) \cong \Z$ is torsion-free.
  Since we have
  $\deg(\iota^* \L) = \deg(\iota^* ([-1]^* \L))$, it follows that $\deg(\iota^* \L^+) = 2
  \deg(\iota^* \L)$ and we are done.
\end{proof}

\subsection{Endomorphisms of $J$ and $D(b,z)$}
We keep the notation of the previous subsection and
 we assume, in addition, that
$\iota^*\L \cong \O_C$. 
Next, we modify the arguments in \cite[Section~2.1]{DLF19} to reinterpret the divisor class of
$D(b,z)$ in Definition~\ref{D:divcon} in terms of the action of the endomorphism of $J$
corresponding to the line bundle $\L$ acting on the divisor class $\iota(z)$ of $z-b$
(Lemma~\ref{L:commdiag}~\eqref{L:finalformendo}). Along the way, we show
that we can identify the Chow--Heegner point/Diagonal cycle
$D_Z(b)$ of \cite[Section~2.2]{DLF19} with the pullback of $\L^-$ by $\iota$
(Lemma~\ref{L:commdiag}~\eqref{L:extractingDbb}).
We will then use Lemma~\ref{L:commdiag} to prove Proposition~\ref{P:CGcomp}, which
is a comparison between the global Coleman--Gross height pairing between the divisors $z-b$
and $D(b,z)$, and our canonical height for the line bundle $L$ at $\iota(z)$. 

We first need some more notation. Let $\pi_1,\pi_2$ denote the standard projections $C \times C
\rightarrow C$ to the first and second factors respectively. Let \[\psi \colon \Pic(C
\times C) \rightarrow \End(J)\] denote the usual action on $J$ of a correspondence on $C
\times C$, where we first pull back a degree $0$ divisor on $C$ by $\pi_1$, then intersect it with the given class in $\Pic(C \times C)$ and then push forward this intersection by $\pi_2$. The map $\psi$ is surjective, with kernel the fibral divisors corresponding to $\pi_1,\pi_2$.

We now describe a natural splitting of the exact sequence
\begin{equation}\label{E:exact_seq} 0 \rightarrow \pi_1^*(\Pic(C)) \oplus \pi_2^*(\Pic(C)) \rightarrow \Pic(C \times C) \xrightarrow{\psi} \End(J) \rightarrow 0 \end{equation}
using the maps $i_{1,b},i_{2,z}$. This splitting will then be used to show
that the map $\theta_{C,b,z}$ from Definition~\ref{D:divcon} factors
through $\End(J)$. The splitting is straightforward -- since $\pi_2 \cdot
i_{1,b}$ and $\pi_1 \cdot i_{2,z}$ are constant maps, we see that
$\ker(i_{1,b}^*) \cap \ker(i_{2,z}^*)$ is a natural complement to the
image of $\pi_1^*(\Pic(C)) \oplus \pi_2^*(\Pic(C))$ in $\Pic(C \times C)$.
This gives a natural isomorphism, which by similar abuse of notation as in
\cite{DLF19}, we call $\psi_z^{-1}$:
\[ \psi_{z}^{-1} \colon \End(J) \xrightarrow{\cong} \ker(i_{1,b}^*) \cap \ker(i_{2,z}^*). \]

\begin{rk}
 Using the decompositions
 \[ \Pic^0(C \times C) = \pi_1^*(\Pic^0(C)) \oplus \pi_2^*(\Pic^0(C)),\]
 and
 \[ \Pic(C \times C) = \pi_1^*(\Pic(C)) \oplus \pi_2^*(\Pic(C)) \oplus \ker(i_{1,b}^*) \cap \ker(i_{2,z}^*),\]
 which in turn induce the decomposition
 \[ \NS(C \times C) = \pi_1^*(\NS(C)) \oplus \pi_2^*(\NS(C)) \oplus \ker(i_{1,b}^*) \cap \ker(i_{2,z}^*),\]
 we may view $\ker(i_{1,b}^*) \cap \ker(i_{2,z}^*)$ as a subspace of $\NS(C \times C)$, if we wish to do so.
\end{rk}
 We recall some of the notation introduced in Section~\ref{sec:app}. For any $a \in J$, recall that $t_a \colon J \rightarrow J$ denotes the translation by $a$ map. Recall that $[-1]$ denotes the inversion map on $J$. For any $\L \in \Pic(J)$, let $\L^+\colonequals \L \otimes [-1]^* \L$ and let $\L^- \colonequals \L \otimes
([-1]^* \L)^{-1}$. Given $\L \in \Pic(J)$, recall that we defined $\phi_{\L} \colon J \rightarrow
\hat{J}$ to be the map $a \mapsto t_a^*(\L) \otimes \L^{-1}$.
If $\Theta$ is the theta divisor on $J$ with respect to $\iota$, then $
\phi_{\Theta}$ denotes the corresponding principal polarization. 
\begin{lemma}\label{L:commdiag}
 Let $\tilde{\theta}_{C,b,z} \colon \End(J) \rightarrow \Pic(C)$ be the map $\tilde{\theta}_{C,b,z} \colonequals \phi_z \cdot \psi_z^{-1}$. 
 Let $\tilde{\phi} \colon \Pic(J) \rightarrow \End(J)$ be the map
  $\tilde{\phi}(\L) \colonequals \phi_{\Theta}^{-1} \cdot \phi_\L$. Then 
 \begin{enumerate}[\upshape (a)]
   \item\label{L:alt_defn_phi} $\psi \cdot (\iota^{(2)})^* =   -\tilde{\phi}$.
  \item\label{L:keycommute} $-\theta_{C,b,z} =  \tilde{\theta}_{C,b,z} \cdot \tilde{\phi}$. 
  \item\label{L:chasingEL} 
  $\tilde{\theta}_{C,b,z}(\tilde{\phi}(\L)) = -[D(b,z)]$.
  \item\label{L:antisymmetry} 
  $2\tilde{\phi}(\L) = \tilde{\phi}(\L^+)$. 
\item\label{L:extractingDbb} $[D(b,b)] = \phi_{\Theta}^{-1} (\L^-)$.
\item\label{L:endoaction} $ [D(b,z)]-[D(b,b)] = -\iota^*(\phi_{\L}([z-b])) = \tilde{\phi}(\L)([z-b]).$ 
  \item\label{L:finalformendo} $[D(b,z)] = \phi_{\Theta}^{-1}(\L^-)+ \tilde{\phi}(\L)([z-b])$ 
 \end{enumerate} 
\end{lemma}
\begin{proof} \hfill
\begin{enumerate}[\upshape (a)]
  \item The inverse of $\phi_{\Theta}$ is $-\iota^*$,
    and after unwinding definitions (see \cite[Section~2.1]{DLF19} for more details), this in turn implies that 
    \[ \psi \cdot (\iota^{(2)})^* =   -\tilde{\phi}.\]

  \item Let $\L \in \Pic(J)$ and let $[Z] \colonequals (\iota^{(2)})^*(\L) \in \Pic(C \times
  C)$. Then, since $\psi_z$ splits the exact sequence~\eqref{E:exact_seq}, there is a fibral
    divisor $[F] \in \pi_1^*(\Pic(C)) \oplus \pi_2^*(\Pic(C))$ such that $\psi_z^{-1}
    \cdot \psi([Z]) = [Z]+[F]$. Since  $\phi_z$ is trivial on fibral divisors, and in
    particular on $[F]$, combining this with~\eqref{L:alt_defn_phi}, it follows that
    \[ -\theta_{C,b,z}(\L) = -\phi_z \cdot (\iota^{(2)})^*(\L) = -\phi_z \cdot \psi_z^{-1} \cdot
    \psi \cdot (\iota^{(2)})^*(\L) = \phi_z \cdot \psi_z^{-1} \cdot \tilde{\phi}(\L) =
    \tilde{\theta}_{C,b,z} \cdot \tilde{\phi}(\L). \]

\item Part~\eqref{L:chasingEL} follows from Part~\eqref{L:keycommute} and Definition~\ref{D:divcon}.
 
 \item Note that $\L^-$ is antisymmetric, and antisymmetric line bundles form precisely the
   kernel of the natural map $\Pic(J) \rightarrow \NS(J)$. Since the map $\tilde{\phi} \colon \Pic(J) \rightarrow \End(J)$ factors as
\[ \tilde{\phi} \colon \Pic(J) \rightarrow \NS(J) \rightarrow \End(J), \] 
combining this with the previous sentence, it follows that $\tilde{\phi}(\L^-) = 0$.
    Applying the homomorphism $\tilde{\phi}$ to the identity $\L^{\otimes 2} = \L^+ \otimes \L^-$ and using the previous line, we get
\[ 2\tilde{\phi}(\L) = \tilde{\phi}(\L^{\otimes 2}) = \tilde{\phi}(\L^+ \otimes \L^-) =
    \tilde{\phi}(\L^+) + \tilde{\phi}(\L^-) = \tilde{\phi}(\L^+). \]

\item

  Let $[2] \colon J \rightarrow J$ denote the multiplication by $2$ map on $J$. Since
    $\iota^{(2)} \cdot i_{1,b} = \iota^{(2)} \cdot i_{2,b} = \iota$ and $\iota^{(2)} \cdot
    \Delta = [2] \cdot \iota$, a direct computation with Definition~\ref{D:divcon} and the identity 
\[ [2]^*\L = \L^{\otimes 3} \otimes [-1]^* \L\] shows that
    \begin{equation}\label{E:Dbb} [D(b,b)] = \iota^*([2]^* \L \otimes \L^{\otimes -2}) =
    \iota^*(\L^+). \end{equation}

    Since $\iota^*(\L) = 0$ by choice of $\L \in \Pic(J)$, it follows that
    \[ 0 = (\iota^*\L)^{\otimes 2} = \iota^*(\L^{\otimes 2}) = \iota^*(\L^+) + \iota^*(\L^-), \]
and therefore
    \begin{equation}\label{E:L+L-} \iota^*(\L^+) = - \iota^*(\L^-).\end{equation}

      Combining Equations~\eqref{E:Dbb} and \eqref{E:L+L-} with the equality $-\iota^* = \phi_{\Theta}^{-1}$, we get
    \[ [D(b,b)] = -\iota^*(\L^-) = \phi_{\Theta}^{-1} (\L^-) .\]

  \item Combining Definition~\ref{D:divcon} with the identities $\iota^{(2)} \cdot i_{2,b}
    = \iota$ and $\iota^{(2)} \cdot i_{2,z} = t_{\iota(z)} \cdot \iota$ and once again using the
    equality $-\iota^* = \phi_{\Theta}^{-1}$, we get

    \[ [D(b,z)]-[D(b,b)] = (i_{2,b}^*-i_{2,z}^*)((\iota^{(2)})^*\L) = -\iota^*(
    t_{\iota(z)}^*\L \otimes
    \L^{-1}) =-\iota^*(\phi_{\L}(\iota(z))) = \tilde{\phi}(\L)(\iota(z)).\]
\item Add the equations in Parts~\eqref{L:extractingDbb} and \eqref{L:endoaction}.\qedhere
\end{enumerate}
\end{proof}

\subsection{Comparison of global heights}
We now compare the global height in our approach to Quadratic Chabauty with the one used
in~\cite{BD18}. 
We keep the notation of the previous subsection.  In addition, we assume that $K$ is a number field
for the rest of this section.
\begin{prop}\label{P:CGcomp}
  Fix a purely mixed curvature form $\alpha_\p$ on $\P_\p$ for every prime
  $\p\mid p$ and a continuous id\`{e}le class character $\chi\colon
  \A^\times_{K}/K^\times\to
  \Q_p$.
  Let $\hat{h}_{\L}$ be the canonical $p$-adic height with respect to these
  choices and let $h^{CG} \colon \Div^0(C) \times \Div^0(C) \rightarrow
  \mathbb{Q}_p$ denote the global Coleman--Gross height pairing between degree $0$ divisors
  relative to $\chi$ and the 
  complementary subspaces $W_\p$ corresponding to $\alpha_\p$ as  in Proposition~\ref{P:curvspace}.
  Then we have for all $z\in C(\Q)$  
  \begin{equation}\label{E:CGmain} 2 \hat{h}_\L(\iota(z)) =  -h^{CG}(z-b,D(b,z))\end{equation} 
\end{prop}
\begin{proof} We will use the notation of Lemma~\ref{L:commdiag}. Let $E_{\L} \colonequals
  \tilde{\phi}(\L) = \phi_{\Theta}^{-1} \cdot \phi_{\L} \in \End(J)$. Let $\hat{a}$ be the point in $\Pic^0(J)$ corresponding to the antisymmetric  line bundle $\L^-$.
  By Lemma~\ref{L:commdiag}~\eqref{L:finalformendo}, we have 
  \[  [D(b,z)] = \phi_{\Theta}^{-1} (\hat{a}) + E_{\L}(\iota(z)). \]
  Hence Theorem~\ref{T:CGcomp} implies that
  \[ h^{CG}(z-b,D(b,z)) =  -\hat{h}_{\P}(\iota(z), \hat{a}+\phi_{\L}(\iota(z))).\]
  Now, by Corollary~\ref{C:can_ht}~\eqref{C:canhtpart1}, it follows that 
  \begin{equation*}
   \hat{h}_{\P}(\iota(z), \hat{a}+\phi_{\L}(\iota(z))) = 2 \hat{h}_{\L}(\iota(z))\,.
    \qedhere
  \end{equation*}
\end{proof}

\begin{rk}\label{R:cantrefine}
The left hand side of~\eqref{E:CGmain} admits a local decomposition coming from the choice
  of rigidification of the line bundle $\L$ in~\eqref{rigidq} and choices of curvature form
  for $\L_\p$ at all $\p\mid p$. We may further assume that the curvature form is chosen compatibly with
  the complementary subspace needed for the local decomposition of the
  Coleman--Gross
  height pairing as in Theorem~\ref{T:CGcomp}. However, the proof of
  Proposition~\ref{P:CGcomp} for equality of global heights given above  cannot directly
  be extended to show that the two sides of~\eqref{E:CGmain} agree locally. This is because the proof uses Theorem~\ref{T:CGcomp}, which is inherently global.     
\end{rk}

\subsection{Comparison of local heights}\label{S:locht}
In the rest of this section, we refine the comparison result of the global heights in
Proposition~\ref{P:CGcomp} to a local comparison result, see
Theorem~\ref{T:localcomparison}. 

For the rest of this section, we work locally at a prime $\q\nmid p$ of our number field $K$. We will
mostly suppress $\q$ from the notation. 
In Definition~\ref{D:vcan}, we first recall the standard construction of a canonical
valuation on $D(b,z)$ that appears in the Coleman--Gross height pairing. This is done by
restricting the canonical valuation on the Poincar\'e bundle to the fiber
corresponding to the divisor $D(b,z)$. We make the maps appearing in this
construction explicit in Lemma~\ref{L:psi}. We then give an alternate
pullback construction of the canonical valuation on $D(b,z)$ in
Proposition~\ref{P:Dbzotherpullback}, which factors through $C \times C$
instead, and is more favorable for proving a comparison result with local
contributions to our canonical $p$-adic height. The intermediate pull-back
line bundle $\L_{\beta}$ on $C \times C$ has the property that it equals
the class of the divisor $[D(b,z)]$ when restricted to $C \times
\{z\}$.  In Lemma~\ref{L:trivial}, we write down an explicit section of a
line bundle $\L_{\beta}$ on $C \times C$ that fiber by fiber equals
the divisor $D(b,z)$. We finally combine
Proposition~\ref{P:Dbzotherpullback} and Lemma~\ref{L:trivial} to prove
Theorem~\ref{T:localcomparison}. 

\subsubsection{The canonical valuation on $D(b,z)$ by alternate pull-backs
from the Poincar\'e bundle}

\begin{rk}\label{R:rignopick}
In the following, we will have to consider various valuations arising as pull-backs of the canonical valuation on the Poincar\'e bundle. To make this precise, we would have
  to be careful about choices of rigidifications and induced choices of isomorphisms that
  are isometries. Since changing the rigidification only changes the associated valuation by a constant,  to simplify the exposition, we will talk about pull-back valuations without choosing rigidifications on the relevant bundles first --  we will instead write
  $v\equiv v'$ if $v$ and $v'$ are valuations on the same line bundle that are the same up
  to an additive constant. This actually suffices, since we ultimately only care about the valuations in
  Theorem~\ref{T:localcomparison} and these are determined uniquely by
  their value at $1(b)$. 
\end{rk}

 For $z\in C$, let $L_z=\O(D(b,z))$.
We first recall the standard construction of the canonical valuation
$v_z \colonequals v^\mathrm{can}_{L_z}$ on $L_z$ from the proof of
\cite[Theorem~9.5.16]{BG06}.  
Let $v^{\mathrm{can}}_{\mathcal{P}}$ denote the canonical valuation on
$\mathcal{P}$ with respect to a choice of rigidification above $(0,0)$ as in Remark~\ref{canrig}. 

Let $\mathcal{P}_C$ denote the Poincar\'e bundle on $C \times J$ such that $\P_{C,b}$ is
trivial, and $\mathcal{P}_C^t$ the pull-back of $\mathcal{P}_C$ under the natural map $J
\times C \rightarrow C \times J$ that swaps the two factors. Let $\mu \colon C \rightarrow
\Pic^0(J)$ be the unique map such that $(\mathrm{id}_J \times \mu)^* \mathcal{P} =
\mathcal{P}_C^t$. 
For any $a\in J$, let $i_{2,a} \colon \Pic^0(J) \rightarrow J \times \Pic^0(J)$ be the map
$c \mapsto (a,c)$.

\begin{defn}\label{D:vcan}\cite[Theorem~9.5.16]{BG06}
We define  $v_z 
  \colonequals \mu^* i_{2,c_z}^* v^{\mathrm{can}}_{\mathcal{P}}$, where  $c_z\in
  J$ is the class of $D(b,z)$.
\end{defn}

We want to relate the valuation $v_z$ to the valuation $v_\q$ on $\O_C$
that appears in our Quadratic Chabauty result. The latter is also obtained by pullback from
$\P$, see~\eqref{vq}. (In~\eqref{vq}, we only define $v_q$ for a prime number
$q\ne p$, since we assume that $C$ is defined over $\Q$. However, an
extension to primes $\q$ of number fields $K$ is immediate.)

\begin{lemma}\label{L:psi} Let $\phi_{\Theta}$ be as in Lemma~\ref{L:commdiag}. Let $t \colon J \times J \rightarrow J \times J$ be the map that swaps the two factors. 
\begin{enumerate}[\upshape (a)]
\item\label{L:deltatranspose}\label{L:rigisom} Let $\delta \colonequals (\mathrm{id}_J \times \phi_{\Theta})^*\mathcal{P}$. Then $\delta^t \colonequals t^*\delta = \delta$ and therefore $(\mathrm{id} \times \phi_{\Theta}^{-1})^* t^* (\mathrm{id} \times
  \phi_{\Theta})^*  {\mathcal{P}} = {\mathcal{P}}$. Furthermore,
    \begin{equation}\label{vequiv}
     (\mathrm{id} \times \phi_{\Theta}^{-1})^* t^* (\mathrm{id} \times
  \phi_{\Theta})^*  v^{\mathrm{can}}_{\mathcal{P}} \equiv v^{\mathrm{can}}_{\mathcal{P}}
    \end{equation}
  \item\label{L:Ppullbacks} $(\iota \times \mathrm{id}_J)^* \delta = -\mathcal{P}_C.$ 
  \item\label{L:mualt} $\mu = -\phi_{\Theta} \cdot \iota $. 
\end{enumerate}
\end{lemma}
\begin{proof} \hfill
\begin{enumerate}[\upshape (a)]
 \item The isomorphism $t^*\delta \cong \delta$ follows from
   \cite[Proposition~8.10.20]{BG06}. Since $t$ and $\mathrm{id}_J \times \phi_{\Theta}$
    are homomorphisms of abelian varieties, Prop~\ref{P:good}~\eqref{P:valfunc} 
   implies~\eqref{vequiv}.     
 \item This is \cite[Proposition~8.10.16]{BG06}.
  \item By the definition of $\mu$, it suffices to show that $(\mathrm{id}_J \times
   -\phi_{\Theta} \cdot \iota)^* \mathcal{P} = \mathcal{P}_C^t$.
 Now, using parts~\eqref{L:deltatranspose} and \eqref{L:Ppullbacks}, we get
 \[ (\mathrm{id}_J \times -\phi_{\Theta} \cdot \iota)^* \mathcal{P} = -(\mathrm{id}_J \times
    \iota)^* (\mathrm{id}_J \times \phi_{\Theta})^* \mathcal{P} = -(\mathrm{id}_J \times
    \iota)^* \delta = -(\mathrm{id}_J \times \iota)^* \delta^t = \mathcal{P}_C^t .  \qedhere\] 
\end{enumerate}   
\end{proof}

The following construction is inspired by~\cite[Section~7]{EL21}.
Let $\beta_1,\beta_2 \colon J \rightarrow J$ be the maps $a \mapsto \tilde{\phi}(\L)(a)$
and $a \mapsto \phi_{\Theta}^{-1}(\L^-)$ respectively. Let $D_1 \colonequals
D(b,z)-D(b,b)$ and $D_2 \colonequals D(b,b)$. For $i\in \{1,2\}$, let $L_{\beta_i}$ be the line bundle on
$C\times C$ defined by $(\iota\times \phi_{\Theta}\cdot {\beta_i} \cdot \iota)^*\P$ and define
\begin{equation}\label{Lbeta}
  L_\beta\colonequals L_{\beta_1}\otimes L_{\beta_2}\,.
\end{equation}
Let $v_i \colonequals (\iota \times \phi_{\Theta} \cdot \beta_i \cdot
\iota)^* v^{\mathrm{can}}_{\mathcal{P}} $ and define a valuation $v_\beta$
on $L_\beta$ by $v_{\beta}\colonequals v_1+v_2$. 
Recall that $i_{1,z}(x) = (x,z)$ for any $x \in C$ and that $\Delta \colon C \rightarrow C
\times C$ denotes
the diagonal map. 
Recall the definition of the trivial
valuation on $\O_J$ from Proposition~\ref{P:good}~\eqref{P:TrivVal}. We will also call the
pullback of the trivial valuation on $\O_J$ to $C$ by the Abel--Jacobi map $\iota$ the trivial
valuation on $\O_C$ and denote it by $v^{\mathrm{triv}}$. 
\begin{prop}\label{P:Dbzotherpullback}\hfill
\begin{enumerate}[\upshape (a)]
\item\label{candecomp} 
We have $
  v_z\equiv \mu^* i_{2,[D_1]}^* v^{\mathrm{can}}_{\mathcal{P}} + \mu^* i_{2,[D_2]}^*
      v^{\mathrm{can}}_{\mathcal{P}}$.
  \item\label{Lalphai1z} We have $i_{1,z}^*L_{\beta}  \cong O(D(b,z))$
   and $v_z\equiv -i_{1,z}^* (v_\beta )$. 
  \item\label{LalphavsL} We have $\Delta^* L_{\beta} \cong \O_C$ and $\Delta^*(v_\beta
    )\equiv 2v_{\q}$.
  \item\label{Lalphai2b} We have $i_{2,b}^* L_{\beta} \cong
    \O_C$ and $i_{2,b}^*(v_\beta )\equiv 2v^{\mathrm{triv}}$.
\end{enumerate}

\end{prop}
\begin{proof}\hfill
\begin{enumerate}[\upshape (a)]
\item Since $D(b,z) = D_1+D_2$, the result follows from \cite[Proposition~9.5.16(b)]{BG06} (the analogue of Proposition~\ref{P:good}~\ref{P:goodadd} for more general varieties). 
 \item Lemma~\ref{L:commdiag}~(\ref{L:endoaction},\ref{L:extractingDbb}) implies that
   $\beta_i \cdot \iota(z) = [D_i]$. This in turn implies that 
 \begin{equation}\label{E:alphai} (\mathrm{id} \times \beta_i \cdot i) \cdot i_{1,z} = i_{1,[D_i]},\end{equation}
 where $i_{1,[D_i]} \colon C \rightarrow C \times J$ is the map $x \mapsto (x,[D_i])$. Let $i_{2,[D_i]} \colon C \rightarrow J \times C$ be the map $x \mapsto ([D_i],x)$ and $t \colon J \times J \rightarrow J \times J$ the map that swaps the two factors. We have the following commutative diagram.

$$\xymatrix{
C \times C \ar[r]^{\iota \times \beta_i \cdot \iota} &  J \times J \ar[r]^{t}  & J \times J  
\ar[r]^{\mathrm{id} \times \phi_{\Theta}}  & J \times \Pic^0(J) \ar@{=}[d] \\
C 
\ar[rr]^{i_{2,[D_i]}} \ar[u]^{i_{1,z}} \ar[d]_{\mu} & & J \times C 
\ar[r]^{\mathrm{id} \times \mu } \ar[u]_{\mathrm{id} \times \iota} \ar[d]^{\mathrm{id} \times \mu } & J \times \Pic^0(J) \ar@{=}[d] \\
\Pic^0(J) \ar[rr]^{i_{2,[D_i]}}  & & J \times \Pic^0(J) \ar@{=}[r] & J \times \Pic^0(J)
}$$

  Combining Equation~\eqref{E:alphai} with the factorization
 \[ \iota \times \phi_{\Theta} \cdot \beta_i \cdot \iota = (\mathrm{id} \times
    \phi_{\Theta}) \cdot (\iota \times \mathrm{id}) \cdot (\mathrm{id} \times \beta_i
    \cdot \iota)\, ,\] 
  we get
 \begin{align*}
  i_{1,z}^* \cdot (\iota \times \phi_{\Theta} \cdot \beta_i \cdot \iota)^*
   v^{\mathrm{can}}_{\mathcal{P}} &= i_{1,z}^* (\mathrm{id} \times \beta_i \cdot \iota)^*
   (\iota
   \times \mathrm{id})^* (\mathrm{id} \times \phi_{\Theta})^*  v^{\mathrm{can}}_{\mathcal{P}} \\
  &= i_{1,[D_i]}^* (\iota \times \mathrm{id})^* (\mathrm{id} \times \phi_{\Theta})^*  v^{\mathrm{can}}_{\mathcal{P}} \\
  &= i_{2,[D_i]}^* (\mathrm{id} \times \iota)^* t^* (\mathrm{id} \times \phi_{\Theta})^*  v^{\mathrm{can}}_{\mathcal{P}} \\
  &= i_{2,[D_i]}^* (\mathrm{id} \times \iota)^* (\mathrm{id} \times \phi_{\Theta})^*
   (\mathrm{id} \times \phi_{\Theta}^{-1})^* t^* (\mathrm{id} \times \phi_{\Theta})^*  v^{\mathrm{can}}_{\mathcal{P}} \\
 &= -i_{2,[D_i]}^* (\mathrm{id} \times \mu)^*  (\mathrm{id} \times
   \phi_{\Theta}^{-1})^* t^* (\mathrm{id} \times \phi_{\Theta})^*  v^{\mathrm{can}}_{\mathcal{P}} \\
 &=  -\mu^* i_{2,[D_i]}^* (\mathrm{id} \times \phi_{\Theta}^{-1})^* t^* (\mathrm{id}
   \times \phi_{\Theta})^*  v^{\mathrm{can}}_{\mathcal{P}}\,. 
  \end{align*}
 In the third to last line, we used Lemma~\ref{L:psi} \eqref{L:mualt} to
    deduce $\mu^* = -\iota^*
  \phi_{\Theta}^*$.  The $i_{2,[D_i]}$ in the second to last line is the map $\Pic^0(J)
    \rightarrow J \times \Pic^0(J)$ defined before Definition~\ref{D:vcan}. Now,
    Lemma~\ref{L:psi}~\eqref{L:rigisom} implies that
  $$ (\mathrm{id} \times \phi_{\Theta}^{-1})^* t^* (\mathrm{id} \times
  \phi_{\Theta})^*  v^{\mathrm{can}}_{\mathcal{P}} \equiv
    v^{\mathrm{can}}_{\mathcal{P}}\,$$
    and therefore the two valuations continue to differ by a constant
    after further pullback by $\mu^* i_{2,[D_i]}^*$. The result now follows from
    Definition~\ref{D:vcan} and part~\eqref{candecomp}.
\item  We have a commutative diagram 
$$\xymatrix{
  C \times C \ar[r]^{\id\times \beta_i\cdot \iota} &  C \times J    \ar[r]^{\iota \times
\phi_{\Theta}}
 & J \times \Pic^0(J)
 \\
C 
\ar[r]^{\iota} \ar[u]_{{\Delta}} & J \ar[r]^{\id\times \beta_i}& J \times J \ar[u]_{\id \times
\phi_{\Theta}}
 \\
}$$
Since 
\[ (\id\times (\phi_{\Theta}\cdot \beta_1) )^*\mathcal{P} \cong (\id\times
\phi_\L )^*\mathcal{P}  \cong \L^+, \qquad (\id\times (\phi_{\Theta}\cdot
\beta_2) )^*\mathcal{P} \cong (i_{1,[\L^-]})^*\mathcal{P} \cong \L^- \]
and $\L^{\otimes 2} \cong \L^+ \otimes \L^-$,
the commutative diagram above implies 
\[ {\Delta}^*L_{\beta} = {\Delta}^*{(\id\times (\phi_{\Theta}\cdot \beta_1) )^*\mathcal{P}
  \otimes (\id\times (\phi_{\Theta}\cdot \beta_2) )^*\mathcal{P} } \simeq \iota^*(L^+ \otimes
  L^-) \simeq \iota^*L^{\otimes 2}\simeq \O_C^{\otimes 2} \simeq \O_C, \]
and $\Delta^* v_i \equiv \iota^*(\id\times (\phi_{\Theta}\cdot \beta_i) )^*
  v^{\mathrm{can}}_{\mathcal{P}}$.
Furthermore, \cite[Proposition~9.5.13, Theorem~9.5.16]{BG06} imply $(\id\times
  (\phi_{\Theta}\cdot \beta_1) )^* v^{\mathrm{can}}_{\mathcal{P}} \equiv
  v^{\mathrm{can}}_{\L^+}$, and similarly, $(\id\times (\phi_{\Theta}\cdot \beta_2) )^*
  v^{\mathrm{can}}_{\mathcal{P}} \equiv v^{\mathrm{can}}_{\L^-}$. By further pullback, it
  follows that 
  \begin{equation*}
  \iota^*(\id\times (\phi_{\Theta}\cdot \beta_1) )^*
    v^{\mathrm{can}}_{\mathcal{P}}+\iota^*(\id\times (\phi_{\Theta}\cdot \beta_2) )^*
    v^{\mathrm{can}}_{\mathcal{P}} \equiv v^{\mathrm{can}}_{\L^+}+v^{\mathrm{can}}_{\L^-}\,,
  \end{equation*}
and the result follows from Proposition~\ref{P:good}~\eqref{P:goodadd}.
\item Using a similar diagram as in the previous part with $i_{2,b}$ in place of $\Delta$, and $0 \times \beta_i, \{0\} \times J$ in place of $\id\times \beta_i, J \times J$ , the universal property of the Poincar\'e bundle (in particular, that $\delta_J|_{0 \times J} \cong O_J$) we find
\[ i_{2,b}^* (L_{\beta_1} \otimes L_{\beta_2}) \simeq \O_C^{\otimes 2}  \simeq \O_C\,.\]
Since $v^{\mathrm{can}}_{\mathcal{P}}$ is the trivial valuation on
$\mathcal{O}_{J}$ ($\mathcal{O}_{\Pic^0(J)}$, respectively) 
 when restricted to $J \times \{0\}$ ($\{0\} \times
\Pic^0(J)$, respectively) by Proposition~\ref{P:good}~(\ref{P:TrivVal},\ref{P:valfunc}), the result follows.\qedhere  
\end{enumerate}

\end{proof}

\subsubsection{Comparison of the two height functions}\label{S:comparison}
Recall from Definition~\ref{D:divcon} that we picked $\L \in \Pic(J)$, and
$s$ a section of $\L$. By assumption $i^*\L \simeq \O_C$. We defined $Z
\colonequals (\iota^{(2)})^*(\L, s) \in \Div(C \times C)$ and $D(b,z) =
(\Delta^*-i_{1,b}^*-i_{2,z}^*)Z$ for a divisor $Z$ on $C \times C$. Define
a divisor $D'$ on $C \times C$ by $D' \colonequals \pi_2^*\Delta^*
Z-\pi_2^* i_{1,b}^* Z-Z$, and let $s'$ be the corresponding section of
$\O(D')$. Let $t \colon C \times C \rightarrow C \times C$ denote the map
that swaps the two factors.
\begin{lemma}\label{L:trivial}\hfill
\begin{enumerate}[\upshape(a)]
 \item We have $\L_{\beta} = t^* \O(D')$. Furthermore, if $s_{\beta} \colonequals t^*s'$,
   then the section $s_z \colonequals i_{1,z}^*s_\beta$ satisfies $\div(s_z) = D(b,z)$ for all $z \in C$.
 \item\label{L:trivsec} The section ${\Delta}^*(s_\beta)\otimes i_{2,b}^*(s_\beta)^{-1}$ has trivial divisor.
\end{enumerate}
\end{lemma}

\begin{proof}\hfill
  \begin{enumerate}[\upshape (a)]
  \item By Proposition~\ref{P:Dbzotherpullback}~(\ref{Lalphai1z},\ref{Lalphai2b}) and the
    seesaw principle \cite[Corollary~8.4.4]{BG06}, it suffices to show that $i_{2,z}^*D' =
      D(b,z)$ for every $z \in C$ and that $i_{1,b}^*\O(D') \simeq \O_C$. The first claim
      follows since $\pi_2 \cdot i_{2,z}$ is the identity map on $C$. Now we prove the
      second claim. Since $\pi_2 \cdot i_{1,b}$ is the constant map that sends every point
      to $b$, and since $\iota^{(2)} \cdot i_{1,b} = \iota$, it follows that 
  \[ i_{1,b}^*D' = -i_{1,b}^*Z = -i_{1,b}^*((\iota^{(2)})^*(\L, s)) = -\iota^*(L,s). \]
  By our choice of $\L$, it follows that $\iota^*L \simeq \O_C$ and the result follows.
   \item It suffices to show that ${\Delta}^*(D') - i_{1,b}^*(D')=0$. Recall from Definition~\ref{D:divcon} that $D(b,z) = (\Delta^*-i_{1,b}^*-i_{2,z}^*)Z$ for a divisor $Z$ on $C \times C$.
    Since $\pi_2 \cdot {\Delta} = \id_C$, we have
    \[ {\Delta}^*(D') = ({\Delta}^* \pi_2^*) {\Delta}^* Z -({\Delta}^* \pi_2^*)
    i_{1,b}^* (Z) - {\Delta}^* (Z) = {\Delta}^* Z - i_{1,b}^* (Z) - {\Delta}^* (Z) = -i_{1,b}^* (Z). \]
Since $\pi_2 \cdot i_{1,b} \colon C \rightarrow C$ is the constant function $b$, we have
    \[ i_{1,b}^*(D') = (i_{1,b}^* \pi_2^*) {\Delta}^* Z -(i_{1,b}^* \pi_2^*) i_{1,b}^*
    (Z) - i_{1,b}^* (Z) =0-0 -i_{1,b}^* (Z) = -i_{1,b}^* (Z), \]
and this implies the result. \qedhere
  \end{enumerate}  
  \end{proof}

Let $L_{\beta_i},v_i$ be as in Proposition~\ref{P:Dbzotherpullback}. 
Recall that $v_{\q}$ is a valuation on the trivial bundle $\O_C$ defined as in \eqref{vq} and
its value on the chosen nowhere vanishing section $1$ is the local contribution to the height in our
approach to Quadratic Chabauty. By our choice of rigidification
in \eqref{rigidq}, it follows that $v_{\q}(1(b)) = 0$ for any choice of nowhere vanishing section $1$ of $\O_C$. 

The following result shows that our local contribution at $\q$ coincides with
the one of Balakrishnan and Dogra in~\cite{BD18} up to a constant multiple.

\begin{thm}\label{T:localcomparison}  
Let $1$ denote a nowhere vanishing section of the trivial bundle $\O_C$,
  and $v_\q$ the corresponding valuation, normalized so that
$v_{\q}(1(b)) = 0$.  Then we have, for all $z\in C(K_\q)$,
\begin{equation*} h_{\q}^{CG} (z-b,D(b,z)) = - 2v_\q(1(z))\chi_{\q}(\pi_{\q})\,, \end{equation*}
where $h_{\q}^{CG}$ is the local Coleman--Gross height at $\q$ with respect
  to the local component of a continuous id\`{e}le class character
  $\chi\colon \A_{K}^\times/K^\times\to \Q_p$ and $\pi_\q$ is a uniformizer at $\q$.      
      \end{thm}
      \begin{proof} Let $s_\beta, s_z$ be as in Lemma~\ref{L:trivial}. By
        the definition of the local Coleman--Gross height, we have
 \begin{equation}\label{CGsz}
   h_{\q}^{CG} (z-b,D(b,z)) = (v_z(s_z(z)) -
   v_z(s_z(b)))\chi_{\q}(\pi_{\q})\,.
 \end{equation}

   Since $s_z = i_{1,z}^* s_{\beta}$ and $$z-b = i_{1,z}^*({\Delta}(z) -i_{2,b}(z)),$$ 
  we find that
  \begin{equation}\label{E:section_pullback}
    s_z(z)-s_z(b) = s_{\beta}({\Delta}(z) -i_{2,b}(z)) = {\Delta}^*s_\beta(z) -i_{2,b}^*s_\beta(z).
  \end{equation}
     Define a valuation $v'$ on 
        $${\Delta}^*\L_{\beta} \otimes i_{2,b}^*(\L_{\beta} )^{-1} \simeq
        \O_C$$ by
        \begin{equation}\label{E:defv'}
    v' \colonequals {\Delta}^*v_\beta - i_{2,b}^*v_\beta\,.
  \end{equation}
By Lemma~\ref{L:trivial}~\eqref{L:trivsec}, it follows that the section $1\colonequals
        {\Delta}^*s_\beta\otimes i_{2,b}^*s_\beta^{-1}$  of $\O_C$ is nowhere vanishing.
Therefore,
    \begin{align*}
(i_{1,z}^*v_{\beta}) (i_{1,z}^* s_{\beta}(z))  -  (i_{1,z}^*v_{\beta}) (i_{1,z}^* s_{\beta}(b)) &= v_{\beta}(s_{\beta}({\Delta}(z))) -
      v_{\beta}(s_{\beta}(i_{2,b}(z))) \\ 
    &= v'(({\Delta}^*s_\beta\otimes i_{2,b}^*s_\beta^{-1})(z)\\
      &= v'(1(z))\\
      &= v'(1(z)\otimes 1(z))\\
      &= {\Delta}^*(v_\beta)(1(z))- i_{2,b}^*(v_\beta)(1(z))\,.
    \end{align*}
      In Proposition~\ref{P:Dbzotherpullback}~(\ref{candecomp},\ref{Lalphai1z}), we showed
        that $v_{z}\equiv -  i_{1,z}^*v_\beta$. By
        Proposition~\ref{P:Dbzotherpullback}~\eqref{LalphavsL} we have that
        $\Delta^*v_\beta\equiv 2v_q$. Since $i_{2,b}^*v_\beta \circ 1 \equiv
        2v^{\mathrm{triv}}(1) =
        0$, we get that $h_{\q}^{CG}
        (z-b,D(b,z)) + 2v_\q(1(z))\chi_{\q}(\pi_{\q})$ is constant as we vary over $z$ in $C$, by combining
        the previous two sentences with~\eqref{CGsz}. By plugging in $z=b$, we see that
        this constant is $0$ by our choice of rigidification in~\eqref{rigidq} and by
        biadditivity of 
        $h_{\q}^{CG}$.
      \end{proof}
\begin{rk}\label{R:localcomparisonp}  
By combining Theorem~\ref{T:localcomparison} with
Proposition~\ref{P:CGcomp} one sees that the sum of
the contributions above $p$ also have to be the same. In particular, when
  $K=\Q$ or, more generally, when $p$ is inert, all contributions are the
  same. Another case where we can draw this conclusion is when $F$ is
  imaginary quadratic and $p=\p\bar{\p}$ splits in $\O_F$, since there is 
a nontrivial continuous id\`{e}le class character $\chi$ with $\chi_\p=0$,
and similarly for $\bar{\p}$. Applying   Theorem~\ref{T:localcomparison}
  and Proposition~\ref{P:CGcomp} to the heights associated to these two
  characters then suffices to prove equality of all local contributions for
  all characters of $F$. A different idea would be needed to prove, for
  instance, a similar statement for $F$ real quadratic and $p$ split.
\end{rk}

\section{A unified theory of $p$-adic adelic metrics: a purely $p$-adic approach}\label{sec:VolVal}

In~\cite[\S9.5]{BG06}, the existence of the canonical real-valued metric on a line bundle on an abelian
variety is proved using a dynamical approach, which does not carry over to
the $p$-adic setting. We now show that there is an alternative
method to construct valuations that uses neither models nor dynamics, but
is instead based on ($p$-adic) log functions, and that this also yields a
canonical valuation on an abelian variety in the sense of Section~\ref{sec:vals}. 
In this way, we get a unified construction of
``canonical $p$-adic metrics'' for line bundles on abelian varieties over number fields at
all places.

Rather than fixing the branch of the logarithm when constructing log
functions, as we did up to this point, we instead allow the branch to vary
and our metric is based on the coefficient of $\log(p)$ in these log
functions. To carry this out we first have to develop the theory of
Vologodsky integration and log functions with variable branch. 
The idea is similar to Colmez's approach in~\cite{Col98} who develops a theory of
single $p$-adic integrals with values in the polynomial algebra in the
variable $\log(p)$ over a finite
extensions $K$ of $\Q_p$. He uses this to give a purely analytic
adelic construction of the  
Coleman--Gross height on a curve. 
In what follows, we develop a theory of Vologodsky integrals that takes
values in a polynomial algebra over $K$ in the formal variable $\log(p)$.
For many applications, such as the restriction to curves of the theory of
log functions on abelian varieties, it suffices to consider Vologodsky
integrals valued in $K\oplus K\log(p)$, see Remark~\ref{R:linearlogp}
below. For a more detailed comparison, see Section~\ref{S:Colmez} below. 

\subsection{Vologodsky integration with a variable branch of log}\label{subsec:VarVol}
Let $K/\Q_p$ be finite and fix an embedding $K \rightarrow \bar{K}$ into an algebraic
closure of $K$. We denote by
$\kst$ the polynomial algebra over $K$ in the formal variable $\log(p)$.
In this section, we let $X/K$ be a smooth, geometrically connected
algebraic variety.
As discussed in Section~\ref{sec:background}, underlying the theory of Vologodsky functions is the Vologodsky canonical path in the
fundamental groupoid of the category of unipotent flat connections on $X$. When $(V,\nabla) $ is
such a connection and $x,y$ are two $K$-rational points of $X$, the theory provides a canonical
$\kst$-linear isomorphism $\gamma_{x,y}\colon V_x \otimes \kst \to V_y \otimes \kst$. By choosing a
branch of the $p$-adic logarithm we can specialize to get a $K$-linear isomorphism and the
theory developed in \cite{Bes05} uses that isomorphism.

We are interested in studying the dependency of
the theory on the branch of the $p$-adic logarithm. For this, we first take the
theory of Vologodsky functions with respect to a fixed branch of the logarithm and transform it
into a $\kst$-valued theory.
Let us give a few details. For simplicity, suppose $X$ has a
$K$-rational point $x$ (this will always be the case for us. See Definition~2.2 in
\cite{Bes05} for how to
deal with the general case). An {\em abstract Vologodsky function with values in a locally free
sheaf $\mathcal{F}$} is a quadruple
$(V,\nabla ,v,s)$ where $(V,\nabla )$ is a unipotent flat connection as before, $v\in
V_{x}\otimes \kst$
and $s \in \Hom(V,\mathcal{F})$ is a morphism of vector bundles (not respecting the connection). Given
this data one obtains an associated $\bar{K}\otimes \kst$-valued function $F$ by setting 
\begin{equation*}
  F(y)=s(\gamma_{x,y}(v))\,.
\end{equation*}

Here, for $y$ an $L$-valued point, with $L\supset K$, the path
$\gamma_{x,y}$ is the $L_{st}$-valued path obtained by viewing $x$ as an
$L$-valued point. Functoriality with respect to field extension guarantees
that this is independent of the choice of $L$.

The theory is independent of the choice of the $K$-rational point $x$.
Choosing a different point $x'$ we may simply parallel transport the vector
$v$ to $x'$. Indeed, in~\cite{Bes02}  we used a system of compatible vectors with
respect to parallel transport.

There is an obvious notion of morphism of abstract Vologodsky functions and an actual
Vologodsky function is just a connected component in the category of abstract Vologodsky
functions. It is easy to see that the associated function is well defined. By changing
$\mathcal{F}$ we get Vologodsky differential forms and functions, $\ocols^i(X)$.  We note that as
discussed in the proof of Theorem~\ref{T:volsum}, in
this theory the integral of $\frac{dz}{z}$ is the universal logarithm defined by letting $ \log(p)$
be its name sake in $\kst$.  We routinely check
the following.
\begin{prop}
    The resulting theory satisfies Theorem~\ref{T:volsum} with $K$ replaced with $\kst$ and the branch
    of the logarithm replaced by the universal logarithm $\log$ as in
    equation~\eqref{E:loguniv}. 
\end{prop}
\begin{proof}
This is mostly a laborious exercise in checking that all definitions and proofs 
  carry over to this setting. One
  point to look at more closely is the identity principle~\cite[Lemma~2.5]{Bes05} which
  states that a Vologodsky
function vanishing locally must vanish globally as well. The proof in the ``classical" case
  (which is given in the proof of~\cite[Proposition~4.12]{Bes02}) follows from the
  fact that if the function corresponding to $(V,\nabla ,v,s)$ vanishes on an open $U$, then $v$ is a
section of the maximal integrable subconnection contained in the kernel of $s$, and that by
functoriality this remains true for all  $\gamma_{x,y}(v)$. In the present setting $v$ is
replaced by a formal combination $\sum v_i \log^i(p)$ and one simply argues the same for each
$v_i$. The analogue of property~\eqref{t6} in Theorem~\ref{T:volsum} 
  states that
there is a short exact sequence, right exact as well in the affine case,
    \begin{equation*}
    0\to \Omega^1(X)\otimes \kst \to \ocola^1\xrightarrow{\delbr} \Omega^1(X) \otimes
    \hdr^1(X/K)\otimes \kst\;.
  \end{equation*}
  The proof is essentially the same as the proof of Theorem~\ref{T:volsum}~\eqref{t6}.
\end{proof}

Iterated integrals as discussed in~\S\ref{331} are very easily described in this theory. Indeed, $
\int_x^z \omega_1\circ \cdots \circ \omega_k$ corresponds to the
connection~\eqref{basicconnection}, the vector
$v=(0,0,\ldots,0,1)$ and the projection $s$ on the first component. Similarly, if we start with
an arbitrary vector $v=(a_k,a_{k-1},\ldots,a_1,1)$ we get what might be called an iterated
integral with constants of integration $a_1,\ldots,a_k$.

We are now going to define a theory which is in the middle between the theory we just outlined
and the one with respect to a fixed branch of the logarithm. 
The idea is to restrict $v$ to be defined over $K$, i.e., to be
in $V_{x}$ rather than in $V_x\otimes \kst$. This gives a smaller
space. The reason for restricting to this smaller space of
functions is so that the theory of valuations built from such functions as
in Definition~\ref{D:volval} is $\Q$-valued instead of $\Qpb$-valued. (See
also Remark~\ref{R:whylogind}). The down side is that parallel transporting along the canonical path will not preserve
the property of being defined over $K$, so the theory now depends on the basepoint $x$.

\begin{defn}
A ($\kst$-valued) Vologodsky function valued in a locally free sheaf $\mathcal{F}$ is said to be \emph{branch independent} at the point $x\in
X(K)$ if there is a representation $(V,\nabla ,v,s)$ which is branch independent in the sense that
$v\in V_x \subset V_x\otimes \kst$. 
\end{defn}
An important point to
note is that if $F$ is branch independent at $x_1$ it may not be branch independent at $x_2\ne
x_1$.

\begin{rk}
Furusho~\cite{Fur04} has a notion of branch independence, which means something quite different.
Furusho proves that changing the branch of the logarithm produces isomorphic rings of
Coleman functions, with
the isomorphism commuting with integration. This fact, which in his terminology is called
branch independence of Coleman integration, is captured by the $\kst$-valued
integration theory.
\end{rk}

The following is clear:
\begin{lemma}\label{easyindep}
   An iterated integral with constants of integration at $x$ is branch independent at $x$ if all the
   constants of integration are in $K$.
\end{lemma}
Obviously, if $F$ is branch independent at  $x\in X(K)$, then $F(x)\in K$. The converse
is not true. To see this, consider the function $F= f \cdot \int \omega$ where $f$ has a zero at
$x$. Then, regardless of the constants of integration, even if they are
$\kst\setminus K$-valued, $F$ vanishes at $x$. 
\begin{prop}\label{easyres}
    Let $U\subset X$ be open and let $F$ be a $\kst$-valued Vologodsky function on $X$. Let
    $x\in U(K)$. Then $F$ is branch independent at $x$ if and only if $F|_U$ is branch independent at $x$.
\end{prop}
\begin{proof}
If $F$ is branch independent at $x$, then restricting the representative to $U$ tells us that $F|_U$ is also branch independent at $x$. 
For the other direction, we use the notion of minimal representatives, as defined in ~\cite[Definition~4.16]{Bes02}. According to 
Lemma~4.17 of loc. cit., every Vologodsky function has a minimal subquotient, and it is unique by
Proposition~4.18 of loc. cit..
Lemma~4.19 of loc. cit. states that the restriction of a minimal representative to an open
subset is also minimal. We claim that a Vologodsky function is branch
  independent at $x$ if and only if its minimal representative is branch
  independent at $x$. This is because the notion of branch independence is
  passed on to subquotients. By the above,
   we see that if $F$ is a minimal representative of a Vologodsky function on $X$ such that its restriction to $U$ has a representative that is branch independent, then the minimal representative $F|_U$ is also branch independent at $x$ and hence the original $F$ is branch independent at $x$.
\end{proof}

Similar to Vologodsky functions with respect to a fixed or variable branch of the $p$-adic
logarithm, we have the following result concerning branch independent Vologodsky functions at a
point $x$. 

\begin{thm}\label{logindepprop}
  Let $X$ be a smooth geometrically connected variety over $K$ and let $x\in X(K)$. The spaces $\ocols^i(X)_x \subset
    \ocols^i(X)$ of
    Vologodsky functions and forms on $X$ which are branch independent at $x$ satisfy the
    following properties:
    \begin{enumerate}[\upshape (a)]
        \item functoriality with respect to basepoint preserving maps: If $f\colon X\to Y$ is a
	morphism of $K$-varieties and $y=f(x)\in Y(K)$, then $f^\ast \ocols^i(Y)_y \subset
	\ocols^i(X)_x$. 
	\item The following sheaf property is satisfied: Suppose $X$ is covered by open $U_j$
	and suppose that $\omega_j \in \ocols^i(U_j)$ are compatible and
        for some $ U_j\ni x$,
        $\omega_j$ is branch independent at $x$. Then there is $\omega\in \ocols^i(X)_x$
	restricting to the $\omega_j$ for each $j$.
      \item\label{lip:ses} There is a short exact sequence, right exact when $X$ is affine,
    \begin{equation*}
    0\to \Omega^1(X)\to \ocolas^1(X)_x\xrightarrow{\delbr} \Omega^1(X) \otimes \hdr^1(X/K)\;.
  \end{equation*}
    \end{enumerate}
\end{thm}
\begin{proof}
Functoriality is obvious. The sheaf property is an immediate consequence of the sheaf property
  for $\kst$-valued forms and Proposition~\ref{easyres}. The proof
  of~\eqref{lip:ses} is in fact
    identical to the proof of the corresponding result for $K$-valued Vologodsky forms.
\end{proof}

We now consider the implications on the theory of log functions discussed
in \S\ref{subs32}.
In the theory, we can replace Vologodsky functions with respect to a fixed branch with those for a variable
branch. We would like to see how much of the theory can be made branch independent at a point.
\begin{rk}
Note that log functions are Vologodsky functions, but they are not
branch independent. Instead, we define branch independence of log functions
as follows.
\end{rk}

\begin{defn}
  Let $\L$ be a line bundle on $X$. A log function on $\L$ is called {\em branch independent
  at $x\in X(K)$}
    if given an open $U\in X$ containing $x$ and a section $s\in \L(U)$, the Vologodsky differential form
	$d\log(s)$ is branch independent at $x$. 
\end{defn}

One immediately observes that it is sufficient to have this condition for one $U $ and one
section $s$.  We have the following strengthening of
Proposition~\ref{P:excurve} when $X$ is nice.
\begin{thm}\label{indepmet}
  Let $X/K$ be a nice variety, let $x\in X(K)$ and let $L$ be a line bundle on $X$ such that
$$\ch_1(\L)\in \im\left(\cup\colon  \Omega^1(X)\otimes \hdr^1(X )\to
  \hdr^2(X)\right)\,.$$ 
  Let 
$\alpha\in \Omega^1(X)\otimes \hdr^1(X )$ such that $\cup                    
    \alpha = \ch_1(\L)$.
Then there exists a log function on $L$ with curvature $\alpha$ which is
  branch independent at $x$.
    It is unique up to the integral of a holomorphic form on $X$.  
    \end{thm}
\begin{proof}
One needs to follow the proof of Proposition~4.4 in~\cite{Bes05}. Given a covering $U_i$
and sections $s_i\in \L(U_i)$, the forms $d\log(s_i)$ are given as iterated integrals and one can certainly
choose the integrals for one $U_i$ containing $x$ to have values in $K$ at $x$, making it branch independent by
Lemma~\ref{easyindep}. One then needs to adjust the constants of integration (see page
336 of loc. cit.), but this can also be done over $K$. The fact that two branch independent log functions
differ by the integral of an $\omega\in \Omega^1(X)$ is clear from Theorem~\ref{logindepprop}.
\end{proof}
\begin{rk}
Note that the degree of freedom in the integral of a
    holomorphic differential is in $\kst$.  For our purposes, this will not matter.
\end{rk}

\subsection{Vologodsky valuations and contributions to local heights at all finite primes}\label{subsec:VolVal}

\newcommand{\ldp}{{(1)}}
\newcommand{\dR}{{\textup{dR}}}
\newcommand{\rig}{{\textup{rig}}}

We keep the notation of~\S\ref{subsec:VarVol}.
\begin{defn}
Let $F$ be a $\kst$-valued Vologodsky function on $X$. Its value at every
$x\in X(K)$ is a polynomial
  $\sum_i a_i (\log p)^i \in \kst$. We set $F^{\ldp}(x)\colonequals a_1$. In this way we get a
  function $F^{\ldp} \colon X(K)\to K$.
\end{defn}
It is easy to see the following.
\begin{lemma}\label{locanval}
For any Vologodsky function $F$ the function $F^{\ldp}$ is a locally
analytic function.
\end{lemma}

The behavior of the function $F^\ldp$ with respect to pull-backs is clear.
\begin{lemma}\label{pld}
  If $f\colon X\to Y$ is a morphism of $K$-varieties and $F$ is a $\kst$-valued
  function on $Y$, then $f^\ast F^{\ldp}= (f^\ast F)^{\ldp}$.
\end{lemma}
\begin{prop}\label{vholz}
  Let $\omega$ be a holomorphic form on 
  $X$. Let $F$ be a Vologodsky integral of $\omega$, branch independent at
  some $x\in X(K)$. Then $F^{\ldp}=0$.
\end{prop}
\begin{proof}
Since $F$ is branch independent at $x$, we have $F(x)\in K$ and by subtracting
it we may assume that $F(x)=0$.
Let $f:X\to A$ be the Albanese map of $X$ with base point $x$. There
exists an invariant differential $\omega'$ on $A$ such that $f^\ast
\omega'=\omega$. By Lemma~\ref{pld} it suffices to prove the result for the
form $\omega'$ and a Vologodsky integral of $\omega'$, branch independent at
$0$, and we may assume that it vanishes at $0$. But this is given by pairing $\omega'$ with the logarithm map of $A$,
independent of the choice of the branch of the logarithm.
\end{proof}
\begin{prop}
  Let $\L$ be a line bundle on a nice variety $X/\Qpb$ and let $\log_\L$ be
  a log function on $\L$.
  Then the function
    \begin{equation}\label{eq:volval}
      v_{\L} \colonequals \log_{\L}^{\ldp}
    \end{equation} 
    is a valuation on $\L$ in the sense of
    Definition~\ref{D:val}.
\end{prop}
\begin{proof}
  This follows because for the universal branch $\log$ of the $p$-adic
  logarithm on $X$ as in Equation~\eqref{E:loguniv}, we have
\begin{equation}\label{E:volscale}
  \log(z) = v(z)\log(p) + \log\left(\frac{z}{p^{v(z)}}\right)\,,
\end{equation}
where $v$ is the valuation on $\Qpb$.
\end{proof}

\begin{rk}
The higher derivatives of $\log_L$ with respect to $\log(p)$ are likely
$0$ for proper varieties, and certainly are for proper curves. See
Remark~\ref{R:linearlogp}.
\end{rk}
\begin{defn}\label{D:volval}
  Let $\L$ be a line bundle on a nice variety $X/\Qpb$. 
    The valuation~\eqref{eq:volval} associated with a log function $\log_{\L}$ on
    $\L$ is called a {\em Vologodsky valuation}.
\end{defn}

\begin{rk}\label{R:whylogind}
  Recall that a $\Q$-valuation is a valuation with values in $\Q$. In general a
Vologodsky valuation need not be a $\Q$-valuation. 
We first restrict to branch independent log functions, since otherwise, even the constant of integration which
has values in $\kst$ will add arbitrary $K$-constants to the valuation.
\end{rk}

We extend Definition~\ref{D:normalized} of a normalized log function
as follows:
\begin{defn}
  A log function
  $\log_\L$ is said to be {\em normalized with respect to a rigidification
  $r \in L_x(K)$}
   if it is branch
  independent at $x$ and $\log_{\L}(r)=0$.
\end{defn}
\begin{prop}\label{valdet}
  Suppose that $L$ is a line bundle on a proper variety $X$ that satisfies the
  conditions of Theorem~\ref{indepmet} and let $r\in L_x(K)$
  be a fixed rigidification of $L$. Then the Vologodsky valuation associated with a log function on $L$,
  normalized with respect to $r$, 
  depends only on its curvature form. 
\end{prop}
\begin{proof}
  By Theorem~\ref{indepmet} two log functions with the same curvature differ by the integral of a
holomorphic differential. If they are both normalized, then this integral
vanishes at the base point $x$, and is therefore branch independent at $x$. The
result now follows from Proposition~\ref{vholz}.
\end{proof}
\begin{rk}
 By Remark~\ref{R:change_rig2}, if we change $r$ to $r'$, then we need to add $v(r /r')$ to the normalized log
function to make it normalized again. Thus, the curvature, without
rigidification, determines the Vologodsky valuation up to an additive
constant. 
\end{rk}
We now show that in good reduction normalized log
functions give rise to model valuations. 
\begin{prop}
  Let $\mathcal{X}$ be a proper smooth $\O_K$-scheme such that
  $\mathcal{X}_K
  = X$ and let
  $\mathcal{L}/\mathcal{X}$ be a line bundle with  generic fiber $L /X$.
  Let $x\in X(K)$ and fix a
  rigidification $r\in L_x(K)$ which is integral with respect to
  $\mathcal{L}$. Then the Vologodsky valuation associated with any log function
  $\log_L$ which is normalized with respect to $r$ is the model valuation
  associated to $\mathcal{L}$ in the sense of
  Example~\ref{E:alg_val}.
\end{prop}
\begin{proof}
Let $\mathcal{Y}$ be the projectivization of $\mathcal{L} \oplus \mathcal{O}_X$ and let
  $Y=\mathcal{Y}_K$. The union of
the $0$ and $\infty$ sections is a divisor with normal crossings. In this
situation, the de Rham fundamental group of $L^{\times}=Y-\{0,\infty\}$
is isomorphic to the rigid fundamental group of $\mathcal{Y}_k -
  \{0,\infty\}$ with its induced Frobenius, where $k$ is the residue field
  of $K$, by~\cite[Proposition 2.4.1]{Chi-leS99}. Similarly, if $y_1,y_2$ are two
  points of $Y(K)$ reducing to  $\overline{y}_1, \overline{y}_2\in \mathcal{Y}_k -
\{0,\infty\}$, then the de Rham path space $P^{\dR}_{y_1,y_2}$ with its Frobenius
is isomorphic to the rigid path space
$P_{\overline{y}_1,\overline{y}_2}^{\rig}$. 
The rigid path spaces have a unique Frobenius invariant path, so it
follows that via the isomorphism with the de Rham path spaces it is sent
to the Vologodsky path (which in general is Frobenius invariant plus has
  some specified behavior with respect to monodromy \cite{Vol03}). Since the Frobenius invariant
path is defined over $K$ already, it follows that if a Vologodsky function
is branch independent at $x_1$ it will also be branch independent at $x_2$.
Applying this to $\log_L$, the assumptions imply that it is branch independent
at $r$ and it follows that the associated valuation is $0$ at every point
which reduces to  $\mathcal{Y}_k - \{0,\infty\}$. This proves the result.
\end{proof}

The behavior of log functions with respect to pull-backs (\cite[Proposition~4.6]{Bes05})  immediately extends as follows:
\begin{prop}\label{pullval}
  Let $f\colon X\to Y$ be a morphism of $K$-varieties. Let $x\in X(K)$ and $y=f(x)\in
  Y(K)$. Let $\L$ be a line bundle on
  $Y$ and let $\M$ be $f^\ast \L$. Let $\log_\L$ be a log function on $\L$
  which is branch independent at $y$ and let $v_\L$ be the associated
  Vologodsky valuation. Then $\log_\M=f^\ast \log_ \L$ is a log function on
  $\M$, branch independent at $x$, and the associated
  valuation satisfies $v_\M= f^\ast v_\L$.
\end{prop}
The following theorem is the main result of this section. It gives a purely 
$p$-adic analytic construction of the canonical valuation, introduced in
Proposition~\ref{P:good}. 

\begin{thm}\label{goodval}
  Let $A/K$ be an abelian variety and let $\L$ be a
  line bundle on $A$ with a rigidification $r\in \L_0(K)$. 
Let $v_\L$ be a      Vologodsky valuation on $\L$, associated to a normalized log function
    $\log_\L$. Then $v_\L$ is the canonical valuation associated to $(L, r)$.
\end{thm}
\begin{proof}
  By~\cite[Theorem~9.5.4, Theorem~9.5.7]{BG06} it suffices to show the
  following:
  \begin{enumerate}[(i)]
    \item $v_L$ is locally bounded;
    \item the unique 
isomorphism~\eqref{m-iso1} (respectively~\eqref{m-iso3}) of rigidified line
      bundles is an isometry for the valuations induced by $v_L$ if $\L$ is symmetric (respectively
      antisymmetric). 
  \end{enumerate}
 By Proposition~\ref{valdet} the valuation depends only on the choice of curvature form, but for abelian
varieties it is also independent of this. Indeed,
curvature forms cupping to $0$ are generated by forms $\omega\otimes [\omega]$ with $\omega$
holomorphic, and these correspond to functions $(\int \omega)^2$. Since
$(\int \omega)^{\ldp}=0$ by Proposition~\ref{vholz}, we have
$\left((\int \omega)^2\right)^{\ldp}=0$ as well, hence $v_L$ is independent
  of $\log_L$.  
  Therefore we may choose $\log_L$ to be a
  canonical
  log function, which implies (ii).

The valuation $v_L$ is locally analytic by Lemma~\ref{locanval}. Hence it
  suffices to show that $v_L$ is $\Q$-valued. 
Indeed, an analytic function with values in $\Q$ has to be constant, so if
  $v_L$ is $\Q$-valued, then it
is even locally constant.

Suppose first that $\L$ is antisymmetric and choose a good log function $\log_\L$.
According to~\cite[Proposition I.9.2]{Mil08} (which references in
turn~\cite[VII, Section 3]{Ser59}),
$\L^{\times}$ has the structure of a commutative group scheme  $G$, sitting
in a short exact sequence 
\begin{equation*}
  0\to \Gm \to G \to A \to 0\;.
\end{equation*} 
  By the discussion before the statement of~\cite[Proposition I.9.2]{Mil08}, the multiplication is induced by
the isomorphism~\eqref{m-iso2}. Thus, Definition~\ref{D:goodmetricanti} shows that
the form $\omega= d \log_{\L}$ is an invariant differential on $G$. The
identity element $e$ of $G$ corresponds to the rigidification $r\in \L|_0$. As
in~\cite[p. 2744]{Zar96}, we have a canonical open subgroup $G(K)_f \subset
G(K)$ characterized by the fact that $G(K) /G(K)_f$ has no nonzero torsion
elements, and for $g\in G(K)_f$, the integral $\log_{\L}(g)=\int_e^g \omega$ is uniquely determined by
functoriality and the local expansion near $e$ so
  $\log_{\L}^{\ldp}|_{G(K)_f}=0$. According to~\cite[page 2747, step 5]{Zar96},
for every $g\in G(K)$ there exist a nonzero integer $m$ and $g_f\in
G(K)_f$ such that $h=g^m g_f^{-1}\in \Gm(K)$. Since $\log_{\L}^{\ldp}(h) \in \Q$ by \eqref{E:volscale}, it follows that 
\begin{equation*}
  \log_{\L}^{\ldp}(g) = \frac{ \log_{\L}^{\ldp}(h) }{m}\in \Q\,.
\end{equation*}
Next we consider the case where $\L$ is the Poincar\'e bundle $\P$ on $A\times \hat{A}$. In
  this case, the valuation $v_L$ is $\Q$-valued because
fiberwise it induces the valuation associated with an antisymmetric line bundle, for which we
already proved the result.
Now let $\L$ be a symmetric line bundle on $A$. Then, by
Proposition~\ref{phiprop}, the
line bundle $\L^2$ is a pullback of $\P$, so for the pulled back log function we get the pulled-back
valuation, which is rational.
It is now immediate from the definition of the canonical log function that
$v_L$
is $\Q$-valued for general $\L$. This proves (i) and hence the
theorem.
\end{proof}
\begin{rk}
In general, a Vologodsky valuation need not be a
$\Q$-valued. For example, suppose that $v$ is a $\Q$-valuation on
  the trivial bundle, induced by a nontrivial log function $\log$. 
Multiplying  $\log\circ 1$ by a suitable constant, 
 the resulting valuation is not
  $\Q$-valued. 
\end{rk}

\begin{rk}
It is a property of abelian varieties that the Vologodsky valuation depends only on the underlying line
bundle and not on the curvature. In our
application to Quadratic Chabauty (see Section~\ref{sec:qc}), we have a
nontrivial line bundle $\L$ on the Jacobian $J$ of a
curve $C$ pulling back to the
trivial line bundle $\O_C$ on $C$. The canonical valuation on $\L$ is associated with
any log function on $\L$ and to determine it we only need, according to
Proposition~\ref{valdet}, to find a
curvature form cupping to $\ch_1(\L)$. By Proposition~\ref{pullval} the
pulled-back valuation on $\O_C$ is associated with the pulled-back curvature
form. Thus, a completely algebraic de Rham datum, easily computed from the
Chern class of $\L$, determines the local contributions to the height
showing up in our approach to Quadratic Chabauty. We will
describe an explicit procedure for computing the canonical valuation from the
curvature form in future work.
\end{rk}
\begin{rk}\label{R:linearlogp}
  By Section~\ref{sec:BDcomparison}, our Quadratic Chabauty functions
  $\log_p\circ 1$ in Section~\ref{sec:qc} are essentially the same as those
  in~\cite{BD18}. The latter are given in terms of the local height pairing between
  two parametrized divisors with disjoint support
  by~\cite[Theorem~1.2]{BD18}, and hence are integrals of differentials of
  the third kind without poles at the endpoints. Hence one could expect
  that our functions take values in $K \oplus K\log(p)\subset \kst$ (in
  fact, by Theorem~\ref{goodval}, in $K\oplus \Q\log(p)$). 
However, Vologodsky's general theory always takes values in the full $\kst$
and one would expect that $n$-iterated integrals will contain powers of $\log(p)$ up to
$\log(p)^n$. In particular, our functions, which are 2-iterated
  integrals, could see
quadratic terms in $\log(p)$. Heuristically, one could argue that 2-iterated
integrals on proper varieties would only see linear terms in $\log(p)$, roughly because the
first integration is of a form of the second kind and that should be independent of the branch.
It seems plausible that on curves one can turn this into a proof using the full force
of~\cite{Kat-Lit21}. Alternatively, we can give the
following ``pure thought'' argument that works in many situations: We defined the Vologodsky
valuation associated with a log function by taking the derivative with respect to $\log(p)$ and
substituting $\log(p)=0$. However, we could just as well set $\log(p)$ to any other value,
still obtaining a valuation. The only case where these valuations are all the same is when all
functions are valued in linear polynomials in $\log(p)$. Thus, when one has a uniqueness result
about valuations, e.g., Proposition~\ref{P:good} for the canonical valuation
  on line bundles on abelian varieties, the resulting Vologodsky functions
  are valued in $K\oplus K\log(p)$.
\end{rk}

\begin{rk}\label{R:uniform}
Theorem~\ref{goodval}
makes it possible to describe the canonical adelic metric
(see Definition~\ref{D:canadelicmetric}) in a uniform way in terms of
$p$-adic analysis.
We therefore get a uniform construction of the canonical $p$-adic
  height~\eqref{E:hhatbundle}.
Moreover, since the component at a prime above $p$ of the classical
  N\'eron-Tate height pairing is defined in terms of the canonical
  valuation, we may also describe it using $p$-adic analysis. We expect
  that this could be of independent interest, for instance in non-archimedean
  Arakelov theory.
\end{rk}

To give one application, we now explain an alternative proof 
  of Proposition~\ref{P:factor} using 
Vologodsky valuations. 
Similar to Section~\ref{sec:qc}, we write $\log_q$ for the pullback
$\iota^*\log_{\L,q}$, where $\L\in\Pic(J)\setminus \Pic^0(J)$ pulls back to $\O_C$
via $\iota$ and $\log_{\L,q}$ is the canonical $q$-adic log function on
$\L\otimes\Q_q$, valued in $\kst$, where $K=\Q_q$.
First note the following consequence of 
Proposition~\ref{pullval} and Theorem~\ref{goodval}.
\begin{lemma}\label{L:}
The function $\lambda_q$ defined in~\eqref{eq:lambdadef} depends only on
the derivative of the function $\log_{q}(1)$ with respect to $\log(q)$.
\end{lemma}
It therefore suffices to show the following: For any two points $x,y \in
C(\Q_q)$ reducing to the same component of the special fiber
$\mathcal{C}_q$ (with notation as in Proposition~\ref{P:factor})
the value $\log_{q}(1)(y)-\log_{q}(1)(x)$ is independent of
the chosen branch of the logarithm. We use Equation \eqref{E:logfromitint}
to write this last quantity as a Vologodsky integral,
\begin{equation*}
  \log_{q}(1)(y)-\log_{q}(1)(x) = \sum \int_x^y\left( \omega_i \int \eta_i
  \right) + \int_x^y \gamma \;,
\end{equation*}
with some holomorphic forms $\omega_i$, some forms of the second kind
$\eta_i$ and a meromorphic form $\gamma$.
Note that we may fix the constant of integration of the $\int \eta_i$ as we
like and the integrals themselves are independent of the branch, a
consequence of~\cite{Bes-Zer13}.
By~\cite{Kat-Lit21}, away from the residue discs of the singular points of the
$\eta_i$ and $\gamma$ this is just a Coleman integral, hence is independent
of the choice of the logarithm. Now we just note that we can pick different
representation for $\log_{q}(1)$ by picking different $\eta_i$ in their
cohomology class, thereby moving their singularities. 
This proves Proposition~\ref{P:factor}.

\subsection{Coleman--Gross heights and relation to Colmez's work}\label{S:Colmez}
We now discuss how local Coleman--Gross heights and N\'eron symbols
on curves can be expressed using Vologodsky valuations and we
relate this construction to work of Colmez~\cite{Col98}.
Suppose that $K$ is the completion of a number field at a place $\p$ above
$p$ 
and  that $C$ is a nice curve defined over this number field with an
Abel--Jacobi map $C\to J$ and associated theta divisor $\Theta$.
Let $\chi$ be a $\Q_v$-valued idel\'e class character, where $v$ is a
finite or infinite prime
of $\Q$. 
For now, suppose that  $v=p$, write $\chi_\p =t_\p\circ \log_\p$ 
and suppose that $\log_\p$ is the branch of the logarithm used for
integration.
  Let $z,w\in \Div^0(C_K)$ be
  degree~0 divisors with disjoint support. We assume that $w$ can be written as
  $w=w_1-w_2$, where $w_1,w_2\in \Div^g(C_K)$ are effective and
  non-special; by the proof of Proposition~\ref{P:MThtglobalpb}, this is
  not an essential restriction. 
  As in~\S\ref{subsec:jacs}, we write  
$b_i = \iota(w_i)\in J_K$ and 
$\Theta^-_{b_i}= t_{b_i}^*[-1]^*\Theta = \div(s_i)$, where $s_i$ is a
section of a line bundle $L_i$.
Fix a purely mixed curvature form $\alpha$ on the Poincar\'e bundle $\P$ on
$J_K$ and let $\log_{\L_i}$ be 
the canonical log function on $\L_i$ with respect to $\alpha$.
Let $z = \sum_j (c_j) - \sum_j(d_j)$, where $c_j,d_j \in C_K$. 

  By~\eqref{E:mt_cg_local} and Corollary~\ref{C:pairingspequal}, we can express the 
 local $p$-adic Coleman--Gross height pairing at $\p$ between $z$
 and $w$ with respect to $\alpha$ and $\chi_\p$  as
 \begin{equation}\label{cg:log}
h^{CG}_\p(z,w) =
   t_\p\left(\sum_j\left( \log_{\L_1}(s_1(c_j)) - \log_{\L_1}(s_1(d_j)) 
   -\log_{\L_2}(s_2(c_j)) +
   \log_{\L_2}(s_2(d_j))\right)\right)\,.
 \end{equation}
From this and Theorem~\ref{goodval}, we immediately get the following description of the local component at $\p$
of the $v$-adic height pairing, where $v\ne p$.
Note that any $\R$-valued idel\'e class character $\chi$ satisfies $\chi_\p(\p)=-\log_\R\mathrm{Nm}(\p)$ up to
a global constant, which we assume to be trivial.

\begin{cor}\label{C:cgvol}
  Let $v\ne p$ be a finite or infinite prime of $\Q$ and let 
  $\chi$ be an idel\'e class character with values in $\Q_v$.  Let $z,w
  \in \Div^0(C)$ have disjoint support and suppose that we can write
  $w=w_1-w_2$ as above. Let $[\cdot,\cdot]^v_\p$ denote either the local 
  component at $\p$ of 
  the $v$-adic Coleman--Gross height pairing with respect to $\chi_v$
  when $v\ne \infty$, or the
  local N\'eron symbol at $\p$ when $v=\infty$. Then, with notation as
  above, we have
  \begin{equation}\label{}
    [z,w]^v_\p = 
    \left(\sum_j\left( \log^{(1)}_{\L_1}(s_1(c_j)) - \log^{(1)}_{\L_1}(s_1(d_j)) 
    -\log^{(1)}_{\L_2}(s_2(c_j)) +
    \log^{(1)}_{\L_2}(s_2(d_j))\right)\right)\chi_\p(\p)\,.
  \end{equation}
\end{cor}

  In~\cite{Col98} Colmez describes an analytic construction of local height
  pairings on curves that is quite similar to our approach~\footnote{In fact, Colmez gives an elegant adelic description of
  his Green functions and height pairings. For ease of presentation, we have
  chosen to fix one $p$.}. 
  Similar to Vologodsky, Colmez also defines (single) integrals with values in
  $\kst$.
  Before~\cite[Lemme~II.2.16]{Col98}, he defines a 
  local height pairing at $\p$ (with respect to a choice of an isotropic complementary
  subspace to
  $\Omega^1(C) \subset \hdr^1(C)$) 
  with values in $K\oplus \Q\log(p)\subset \kst$ between divisors
  $z,w\in \Div^0(C_K)$  with disjoint support,
 essentially via the same definition~\eqref{eq:cgp} as Coleman--Gross, but using his
integration theory for forms of the third kind. 
As his integrals are
Vologodsky integrals~(see the discussion just before Definition~2.1
in~\cite{Bes17}), this gives
the extended Coleman--Gross height pairing~\eqref{eq:cgp} once
we choose a value for $\log(p)$, so the extended Coleman--Gross height 
  is due to Colmez. 
  Then, at the end of the proof
  of~\cite[Th\'eor\`eme~II.2.18]{Col98} on
page 101, he shows that the derivative with respect to $\log(p)$
of this height pairing (in his notation, a superscript (1)) is, up to a
constant, the local
contribution at $\p$ of the $v$-adic height pairing, for all primes
$v$ such that $\p \nmid v$, including $\infty$. 
This gives a $p$-adic analytic construction of this local contribution.

Colmez's integrals of differentials of the third kind on $C$ are in fact obtained by pulling back symmetric
Green functions on the theta-divisor as in~\S\ref{subsec:relcol}, but
valued in $\kst$. Let $G_\Theta$ be such a symmetric Green function. It
follows from Proposition~II.2.9 and Th\'eor\`eme~II.2.18  of loc. cit.
that,
with notation as above, we have
  \begin{equation}\label{colmezcg}
    [z,w]^v_\p = 
    \left(\sum_j\left( G^{(1)}_\Theta(c_j - w_1) -
G^{(1)}_\Theta(d_j - w_1) -
G^{(1)}_\Theta(c_j - w_2) +
G^{(1)}_\Theta(d_j - w_2) 
    \right)\right)\chi_\p(\p)\,.
  \end{equation}
  In fact,
  Theorem~\ref{T:loggreen} implies that~\eqref{colmezcg}  is equivalent to
  Corollary~\ref{C:cgvol}. This provides an alternative proof
  of~\cite[Theorem~II.2.18]{Col98}.

\bibliography{Besser_Mueller_Srinivasan_revised}

\end{document}